\documentclass[reqno,12pt,letterpaper]{amsart}
\usepackage{amsmath,amssymb,amsthm,graphicx,mathrsfs,url}
\usepackage[usenames,dvipsnames]{color}
\usepackage[colorlinks=true,linkcolor=Red,citecolor=Green]{hyperref}

\def\smallsection#1{\smallskip\noindent\textbf{#1}.}

\newtheorem{theo}{Theorem}
\newtheorem{prop}{Proposition}[section]
\newtheorem{defi}[prop]{Definition}

\newtheorem{lemm}[prop]{Lemma}

\numberwithin{equation}{section}

\newcommand{\mc}{\mathcal}

\newcommand{\cc}{\mathbb{C}}

\newcommand{\la}{\lambda}

\newcommand{\pl}{\partial}
\newcommand{\x}{\times}

\newcommand{\bbar}{\overline}

\newcommand{\cjd}{\rangle}
\newcommand{\cjg}{\langle}

\DeclareMathOperator{\Con}{Con}
\DeclareMathOperator{\Res}{Res}
\DeclareMathOperator{\comp}{comp}
\DeclareMathOperator{\Ell}{ell}
\DeclareMathOperator{\End}{End}
\DeclareMathOperator{\Hol}{Hol}
\DeclareMathOperator{\Hom}{Hom}

\let\Im=\Imag

\DeclareMathOperator{\Ran}{Ran}
\DeclareMathOperator{\rank}{rank}

\let\Re=\Real
\DeclareMathOperator{\sgn}{sgn}

\DeclareMathOperator{\supp}{supp}

\DeclareMathOperator{\WF}{WF}
\DeclareMathOperator{\tr}{tr}
\def\WFh{\WF_h}
\def\indic{\operatorname{1\hskip-2.75pt\relax l}}

\title{Pollicott--Ruelle resonances for open systems}
\author{Semyon Dyatlov}
\email{dyatlov@math.mit.edu}
\address{Department of Mathematics, Massachusetts Institute of Technology,
Cambridge, MA 02139, USA}
\author{Colin Guillarmou}
\email{cguillar@dma.ens.fr}
\address{DMA, U.M.R. 8553 CNRS, \'Ecole Normale Superieure, 45 rue d'Ulm,
75230 Paris cedex 05, France}
\begin{document}

\begin{abstract}
We define Pollicott--Ruelle resonances for geodesic flows on noncompact asymptotically
hyperbolic negatively curved manifolds, as well as for more general
open hyperbolic systems related to Axiom A flows. These resonances are the poles
of the meromorphic continuation of the resolvent of the generator of the flow
and they describe decay of classical correlations. As an application, we show
that the Ruelle zeta function extends meromorphically to the entire complex plane.
\end{abstract}

\maketitle

%%%%%%%%%%%%%%%%%%%%%%%%%%%%%%%%%%%%%%%%%%%%%%%%%%%%%%%%%%%%%%%%%%%%%%%%%%%%%%%%
%                                 INTRODUCTION                                 %
%%%%%%%%%%%%%%%%%%%%%%%%%%%%%%%%%%%%%%%%%%%%%%%%%%%%%%%%%%%%%%%%%%%%%%%%%%%%%%%%
\addtocounter{section}{1}
\addcontentsline{toc}{section}{1. Introduction}

For an Anosov flow on a compact manifold, \emph{Pollicott--Ruelle resonances}
are complex numbers which describe fine features of decay of correlations~\cite{Po,Ru}.
They also are the singularities of the meromorphic extension of
the Ruelle zeta function, whose existence (conjectured
by Smale~\cite{Sm}) has recently been proved on compact manifolds
by Giulietti--Liverani--Pollicott~\cite{glp}, see also
Dyatlov--Zworski~\cite{DyZw}.

The purpose of this paper is to define Pollicott--Ruelle resonances
for \emph{open hyperbolic systems}. An example is the geodesic flow $\varphi^t=e^{tX}:SM\to SM$
on an asymptotically hyperbolic negatively curved noncompact Riemannian manifold $M$
(see~\S\ref{s:examples-geodesic}).
Building on the microlocal approach of Faure--Sj\"ostrand~\cite{FaSj} and~\cite{DyZw}, we show that:
\begin{itemize}
\item the resolvent $(X+\lambda)^{-1}:L^2(SM)\to L^2(SM)$, $\Re\lambda>0$,
continues meromorphically to $\mathbf R(\lambda):C_0^\infty(SM)\to \mathcal D'(SM)$, $\lambda\in\mathbb C$ (Theorem~\ref{t:mer});
\item the singular part of $\mathbf R(\lambda)$ at its poles (called Pollicott--Ruelle resonances)
is described in terms of support and wavefront set
(Theorem~\ref{t:res-states});
\item the Ruelle zeta function $\zeta(\lambda)=\prod_{\gamma^\sharp} (1-e^{-\lambda T_{\gamma^\sharp}})$, where
$T_{\gamma^\sharp}>0$ are the lengths of primitive closed geodesics,
extends meromorphically to $\lambda\in\mathbb C$ (Theorem~\ref{t:zeta}).
\end{itemize}
These results are motivated by decay of correlations, counting closed trajectories,
linear response, and boundary rigidity in geometric inverse problems~-- see the discussion below.

Rather than consider the flow on the entire $SM$, it suffices to work with its restriction
to $\mathcal U=S U$, where $U\subset M$ is a large convex compact set containing all trapped trajectories.
Our results hold under the following general assumptions
(see~\S\ref{s:examples-basic} for a basic example
and~\S\ref{s:boundary} for applications to boundary problems):
\begin{enumerate}
\item[\textbf{\hypertarget{AA1}{(A1)}}]
$\overline{\mathcal U}$ is an $n$-dimensional compact manifold with interior $\mathcal U$ and boundary $\partial\mathcal U$, $X$
is a smooth ($C^\infty$) nonvanishing vector field on $\overline{\mathcal U}$, and
$\varphi^t=e^{tX}$ is the corresponding flow;
\item[\textbf{(A2)}] $\rho\in C^\infty(\overline{\mathcal U})$ is a boundary defining function,
that is $\rho>0$ on $\mathcal U$, $\rho=0$ on $\partial\mathcal U$,
and $d\rho\neq 0$ on $\partial\mathcal U$;
\item[\textbf{\hypertarget{AA3}{(A3)}}] the boundary $\partial\mathcal U$ is \emph{strictly convex} in the sense that
\begin{equation}
  \label{e:convex}
x\in\partial\mathcal U,\
X\rho(x)=0\quad \Longrightarrow\quad X^2\rho(x)<0.
\end{equation}
\end{enumerate}
The condition~\eqref{e:convex} does not depend on the choice of $\rho$.
We embed $\overline{\mathcal U}$ into some {\em com\-pact} manifold $\mathcal M$ without boundary
(which is unrelated to the noncompact manifold $SM$ used in the example above)
and extend the vector field $X$ there, so that the flow $\varphi^t$ is defined for
all times, see~\S\ref{s:dyn1}. We choose the extension of $X$ to $\mathcal M$ (also denoted $X$)
so that $\mathcal U$ is convex in the sense that (see Lemma~\ref{l:extended})
\begin{equation}
  \label{e:convex2}
x,\varphi^T(x)\in\mathcal U,\ T\geq 0\quad\Longrightarrow\quad
\varphi^t(x)\in\mathcal U\quad\text{for all }t\in [0,T].
\end{equation}
Define the \emph{incoming/outgoing tails} $\Gamma_\pm\subset\overline{\mathcal U}$ and the \emph{trapped set} $K$
by
\begin{equation}
  \label{e:gpm}
\Gamma_\pm:=\bigcap_{\pm t\geq 0} \varphi^t(\overline{\mathcal U}),\quad
K:=\Gamma_+\cap\Gamma_-.
\end{equation}
We have $K\subset\mathcal U$, see~\S\ref{s:dyn1}.
We make the assumption that the flow is \emph{hyperbolic} on $K$:
\begin{enumerate}
\item[\textbf{\hypertarget{AA4}{(A4)}}] for each $x\in K$, there is a splitting
\begin{equation}
  \label{e:hypersplit}
T_x \mathcal M=E_0(x)\oplus E_s(x)\oplus E_u(x),\quad
E_0(x):=\mathbb R X(x)
\end{equation}
continuous in $x$, invariant under $\varphi^t$, and such that for some constants $C,\gamma>0$,
\begin{equation}
  \label{e:hyper}
|d\varphi^t(x)\cdot v|\leq Ce^{-\gamma |t|}|v|,\quad
\begin{cases}
t\geq 0,\ v\in E_s(x);\\
t\leq 0,\ v\in E_u(x).
\end{cases}
\end{equation}
\end{enumerate}
In the terminology of~\cite[Definitions~6.4.18 and~17.4.1]{KaHa}, $K$ is a \emph{locally maximal hyperbolic set}
for the flow $\varphi^t$. In fact, each locally maximal set has a strictly
convex neighborhood (see~\cite[Theorem~1.5]{conley-easton}
and the discussion following~\cite[Corollary~B]{robinson}),
thus our results hold for any basic set of an Axiom A flow~\cite[\S7]{Po0}.
(We keep the strict convexity assumption because it simplifies the proofs and is satisfied in many cases.)
We remark that our proofs never use that the periodic orbits are dense in $K$ (which is part of Axiom A).

We finally assume that
\begin{enumerate}
  \item[\textbf{\hypertarget{AA5}{(A5)}}] we fix a smooth
complex vector bundle $\mathcal E$ over $\overline{\mathcal U}$ and a first order
differential operator $\mathbf X:C^\infty(\overline{\mathcal U};\mathcal E)\to C^\infty(\overline{\mathcal U};\mathcal E)$ such that
\begin{equation}
  \label{e:X-bundles}
\mathbf X(f\mathbf u)=(Xf)\mathbf u+f(\mathbf X\mathbf u),\quad
f\in C^\infty(\overline{\mathcal U}),\
\mathbf u\in C^\infty(\overline{\mathcal U};\mathcal E).
\end{equation}
\end{enumerate}
In the scalar case, where $\mathcal E=\mathbb R$ is the trivial bundle, \eqref{e:X-bundles} means
that $\mathbf X=X-V$, where $V\in C^\infty(\overline{\mathcal U};\mathbb C)$ is a potential.
We extend $\mathbf X$ arbitrarily to $\mathcal M$ so that~\eqref{e:X-bundles} holds.

One important special case is that of \emph{Anosov flows}, when $\mathcal U=\mathcal M$ is a compact
manifold without boundary, and thus $K=\mathcal U$. In this situation, Pollicott--Ruelle resonances have been
studied extensively, see the overview of previous work below.

%%%%%%%%%%%%%%%%%%%%%%%%%%%%%%%%%%%%%%%%%%%%%%%%%%%%%%%%%%%%%%%%%%%%%%%%%%%%%%%%
\smallsection{Meromorphic continuation of the resolvent}
Fix a smooth measure $\mu$ on $\mathcal M$
and a smooth metric on the fibers of $\mathcal E$ (neither needs to be invariant under the flow);
this fixes a norm on $L^2(\mathcal M;\mathcal E)$.
We consider the \emph{transfer operator} $e^{-t\mathbf X}:L^2(\mathcal M;\mathcal E)\to L^2(\mathcal M;\mathcal E)$. For the scalar
case $\mathcal E=\mathbb R$, $\mathbf X=X-V$, it has the form
\begin{equation}
  \label{e:Xpot}
e^{-t\mathbf X}f(x)=\exp\Big(\int_0^t V(\varphi^{-s}(x))\,ds\Big) f(\varphi^{-t}(x)).
\end{equation}
Note that~\eqref{e:X-bundles} implies the following property characterizing the support of $e^{-t\mathbf X}$:
\begin{equation}
  \label{e:X-support}
e^{-t\mathbf X}(f\mathbf u)=(f\circ\varphi^{-t})e^{-t\mathbf X}\mathbf u,\quad
f\in C^\infty(\mathcal M),\
\mathbf u\in C^\infty(\mathcal M;\mathcal E).
\end{equation}
For a large constant $C_0$, we have
\begin{equation}
  \label{e:C0-def}
\|e^{-t\mathbf X}\|_{L^2(\mathcal M;\mathcal E)\to L^2(\mathcal M;\mathcal E)}\leq e^{C_0t},\quad
t\geq 0.
\end{equation}
For $\Re\lambda > C_0$, the resolvent
$(\mathbf X+\lambda)^{-1}$ on $L^2(\mathcal M;\mathcal E)$
is given by the formula
\begin{equation}
  \label{e:res-upper}
(\mathbf X+\lambda)^{-1}\mathbf f=\int_0^\infty e^{-t(\mathbf X+\lambda)}\mathbf f \,dt.
\end{equation}
For the purposes of meromorphic continuation, we consider the \emph{restricted resolvent}
\begin{equation}
  \label{e:res}
\mathbf R(\lambda)=\indic_{\mathcal U}(\mathbf X+\lambda)^{-1}\indic_{\mathcal U}:C_0^\infty(\mathcal U;\mathcal E)\to \mathcal D'(\mathcal U;\mathcal E),\quad
\Re\lambda>C_0.
\end{equation}
Here $\mathcal D'(\mathcal U;\mathcal E)$ is the space of distributions on $\mathcal U$ with values in $\mathcal E$.
By~\eqref{e:convex2}, \eqref{e:X-support}, and~\eqref{e:res-upper},
$\mathbf R(\lambda)$ does not depend on the values of $\mathbf X$ outside of $\mathcal U$.
Our main result is
%%%%%%%%%%%%%%%%%%%%%%%%%%%%%%%%%%%%%%%%%%%%%%%%%%%%%%%%%%%%%%%%%%%%%%%%%%%%%%%%
\begin{theo}
  \label{t:mer}
Under the assumptions~\hyperlink{AA1}{\rm(A1)}--\hyperlink{AA5}{\rm(A5)},
the family $\mathbf R(\lambda)$ defined in~\eqref{e:res} continues meromorphically to $\lambda\in\mathbb C$,
with poles of finite rank. These poles
are called Pollicott--Ruelle resonances of $\mathbf X$. 
\end{theo}
%%%%%%%%%%%%%%%%%%%%%%%%%%%%%%%%%%%%%%%%%%%%%%%%%%%%%%%%%%%%%%%%%%%%%%%%%%%%%%%%
In fact, we can define $\mathbf R(\lambda)$ as the restricted resolvent of a Fredholm problem on certain anisotropic Sobolev
spaces~-- see~\S\ref{s:proofs}. For the case of Anosov flows, our definition
of resonances coincides with that of~\cite{FaSj}, the only difference being
the convention for the spectral parameter~-- if $\{\lambda_j\}$ are the resonances
in Theorem~\ref{t:mer}, then the resonances of~\cite{FaSj}
are $\{i\lambda_j\}$.

%%%%%%%%%%%%%%%%%%%%%%%%%%%%%%%%%%%%%%%%%%%%%%%%%%%%%%%%%%%%%%%%%%%%%%%%%%%%%%%%
\smallsection{Characterization of resonant states}
We next study the singular parts of $\mathbf R(\lambda)$.
For each $\lambda\in\mathbb C$, $j\geq 1$ define the space of \emph{generalized resonant states}
\begin{equation}
  \label{e:res-spaces}
\Res_{\mathbf X}^{(j)}(\lambda)=\{\mathbf u\in \mathcal D'(\mathcal U;\mathcal E)\mid
\supp\mathbf u\subset\Gamma_+,\
\WF(\mathbf u)\subset E_+^*,\
(\mathbf X+\lambda)^j\mathbf u=0
\}.
\end{equation}
Here $E_+^*\supset E_u^*$ is the extended unstable bundle
over $\Gamma_+$, constructed in Lemma~\ref{l:extended}.
We will also use the extended stable bundle $E_-^*\supset E_s^*$ over $\Gamma_-$.
The symbol $\WF$ denotes the wavefront set, see~\S\ref{s:notation}.
%%%%%%%%%%%%%%%%%%%%%%%%%%%%%%%%%%%%%%%%%%%%%%%%%%%%%%%%%%%%%%%%%%%%%%%%%%%%%%%%
\begin{theo}
  \label{t:res-states}
For each $\lambda_0\in\mathbb C$, we have the expansion
\begin{equation}
  \label{e:xpansion}
\mathbf R(\lambda)=\mathbf R_H(\lambda)+\sum_{j=1}^{J(\lambda_0)} {(-1)^{j-1}(\mathbf X+\lambda_0)^{j-1}\Pi_{\lambda_0}\over (\lambda-\lambda_0)^j}
\end{equation}
for some family $\mathbf R_H(\lambda):C_0^\infty(\mathcal U;\mathcal E)\to \mathcal D'(\mathcal U;\mathcal E)$ holomorphic near $\lambda_0$
and some finite rank operator $\Pi_{\lambda_0}:C_0^\infty(\mathcal U;\mathcal E)\to\mathcal D'(\mathcal U;\mathcal E)$. Moreover,
if $K_{\Pi_{\lambda_0}}$ is the Schwartz kernel of $\Pi_{\lambda_0}$ and $\WF'(\Pi_{\lambda_0})$ is its wavefront set (see~\eqref{e:wfprime}, \eqref{e:schwartz}),
then
\begin{align}
  \label{e:piprop1}
\supp K_{\Pi_{\lambda_0}}\ \subset\ \Gamma_+\times\Gamma_-,&\quad
\WF'(\Pi_{\lambda_0})\ \subset\ E_+^*\times E_-^*;\\
  \label{e:piprop2}
\Pi_{\lambda_0}^2=\Pi_{\lambda_0},\quad
\mathbf X\Pi_{\lambda_0}=\Pi_{\lambda_0}\mathbf X,&\quad
\Ran(\Pi_{\lambda_0})=\Res_{\mathbf X}^{(J(\lambda_0))}(\lambda_0).
\end{align}
\end{theo}
%%%%%%%%%%%%%%%%%%%%%%%%%%%%%%%%%%%%%%%%%%%%%%%%%%%%%%%%%%%%%%%%%%%%%%%%%%%%%%%%
The operator products in~\eqref{e:piprop2} are understood in the sense of distributions.
The operator $\Pi_{\lambda_0}^2:C_0^\infty(\mathcal U;\mathcal E)\to \mathcal D'(\mathcal U;\mathcal E)$ is well-defined due to~\eqref{e:piprop1},
since $\Gamma_+\cap\Gamma_-=K$ is a compact subset of $\mathcal U$
and $E_+^*$, $E_-^*$ only intersect at the zero section~-- see~\cite[Theorem~8.2.14]{ho1}.

Note that Theorem~\ref{t:res-states} implies that $\lambda_0$ is a resonance
if and only if the space $\Res_{\mathbf X}^{(1)}(\lambda_0)$ of resonant states
is nontrivial.

We can apply Theorem~2
to the operator $\mathbf X^*$, which satisfies~\eqref{e:X-bundles} with $X$ replaced by $-X$
and $\mathcal E$ replaced by $\mathcal E^*\otimes |\Omega|^1$, with $|\Omega|^1$ the
bundle of densities on $\mathcal U$.
The direction of the flow is reversed, which means that $\Gamma_+,E_+^*$ switch
places with $\Gamma_-,E_-^*$. Therefore,
$$
\Ran((\Pi_{\lambda_0})^*)=\Res_{\mathbf X^*}^{J(\lambda_0)}(\lambda_0),
$$
where $\Res_{\mathbf X^*}^{(j)}(\lambda)$ is the space of generalized
coresonant states:
$$
\Res_{\mathbf X^*}^{(j)}(\lambda)=\{\mathbf v\in\mathcal D'(\mathcal U;\mathcal E^*\otimes |\Omega|^1)\mid
\supp \mathbf v\subset\Gamma_-,\
\WF(\mathbf v)\subset E_-^*,\
(\mathbf X^*+\bar\lambda)^j\mathbf v=0\}.
$$
Note that for each $\mathbf u\in\Res^{(1)}_{\mathbf X}(\lambda),\mathbf v\in\Res^{(1)}_{\mathbf X^*}(\lambda)$,
the pointwise product $\mathbf u\cdot\bbar{\mathbf v}\in\mathcal D'(\mathcal U;|\Omega|^1)$ is well-defined
and supported on $K$, thanks to~\cite[Theorem~8.2.10]{ho1}. Moreover,
$\mathcal L_X(\mathbf u\cdot \bbar{\mathbf v})=0$.

%%%%%%%%%%%%%%%%%%%%%%%%%%%%%%%%%%%%%%%%%%%%%%%%%%%%%%%%%%%%%%%%%%%%%%%%%%%%%%%%
\smallsection{Ruelle zeta function}
Let $V\in C^\infty(\mathcal U;\mathbb C)$. For a primitive closed trajectory $\gamma^\sharp:[0,T_{\gamma^\sharp}]\to K$
of $\varphi^t$ of period $T_{\gamma^\sharp}$, let
\begin{equation}
  \label{e:V-gamma}
V_{\gamma^\sharp}={1\over T_{\gamma^\sharp}}\int_0^{T_{\gamma^\sharp}} V(\gamma^\sharp(t))\,dt
\end{equation}
be the average of $V$ over $\gamma^\sharp$. Define the \emph{Ruelle zeta function}
as the following product over all primitive closed trajectories of $\varphi^t$ on $K$:
$$
\zeta_V(\lambda):=\prod_{\gamma^\sharp}\big(1-\exp(-T_{\gamma^\sharp}(\lambda+V_{\gamma^\sharp}))\big),\quad
\Re\lambda\gg 1.
$$
The product converges for $\Re\lambda$ large enough since the number of closed trajectories
of period no more than $T$ grows at most exponentially in $T$~-- see Lemma~\ref{l:recur3}.

%%%%%%%%%%%%%%%%%%%%%%%%%%%%%%%%%%%%%%%%%%%%%%%%%%%%%%%%%%%%%%%%%%%%%%%%%%%%%%%%
\begin{theo}
  \label{t:zeta}
Assume that the stable/unstable foliations $E_u,E_s$ are orientable.
Then the function $\zeta_V(\lambda)$ admits a meromorphic continuation to $\lambda\in\mathbb C$.
\end{theo}
%%%%%%%%%%%%%%%%%%%%%%%%%%%%%%%%%%%%%%%%%%%%%%%%%%%%%%%%%%%%%%%%%%%%%%%%%%%%%%%%
Theorem~\ref{t:zeta} was established in~\cite{glp} in the special
case of Anosov flows when additionally $V=0$.
Another argument based on microlocal methods was
presented in~\cite{DyZw} and served as the starting point of our
proof.
The singularities (zeroes and poles) or $\zeta_V$ are Pollicott--Ruelle resonances
for certain operators on the bundle of differential forms, see~\S\ref{s:zeta}.
Our methods actually prove meromorphic continuation of more general dynamical
traces~-- see Theorem~\ref{t:trace} in~\S\ref{s:trace}.
The orientability condition can be relaxed, see~\eqref{e:orientability} and
the remarks following it.

In Theorem~\ref{t:zeta} we assumed that the potential $V$ is smooth.
However it is likely that this statement also holds for certain nonsmooth potentials
arising from (un)stable Jacobians by passing to the Grassmanian bundle
of $M$ as in~\cite[\S 2]{FaTs3}. The framework of the present paper
appears convenient for that goal since the lifted flow on a neighborhood
of the unstable bundle in the Grassmanian bundle of an open hyperbolic
system produces another open hyperbolic system.

%%%%%%%%%%%%%%%%%%%%%%%%%%%%%%%%%%%%%%%%%%%%%%%%%%%%%%%%%%%%%%%%%%%%%%%%%%%%%%%%
\smallsection{Applications to boundary value problems}
A useful corollary of our work is the well-posedness 
(up to a finite dimensional space corresponding to resonant states) 
of the two boundary value problems for the transport equation 
\[
(X-V)u=f \textrm{ in }\mc{U}, \quad  u|_{\pl_\pm \mc{U}}=0
\] 
for $u,f$ in certain anisotropic Sobolev spaces, where $\pl_\pm \mc{U}:=\{x\in \pl \mc{U}\mid \mp X\rho>0\}$ 
and $V$ is a potential; see for instance Proposition \ref{prop:bvp} and particularly~\cite[\S4]{G}.
The microlocal description of solutions is crucial in the proof of lens 
rigidity of surfaces with hyperbolic trapped sets and no conjugate points in \cite{G}. 

%%%%%%%%%%%%%%%%%%%%%%%%%%%%%%%%%%%%%%%%%%%%%%%%%%%%%%%%%%%%%%%%%%%%%%%%%%%%%%%%
\smallsection{Motivation and discussion}
We call a resonance $\lambda_0$ the \emph{first resonance} of $\mathbf X$ if $\lambda_0$
is simple (that is, $\rank\Pi_{\lambda_0}=1$), $\lambda_0\in\mathbb R$,
and there are no other resonances with $\Re\lambda\geq \Re\lambda_0$.
We say that there is a \emph{spectral gap}
of size $\nu>0$ if there are no resonances with $\Re\lambda\geq \Re\lambda_0-\nu$,
and an \emph{essential spectral gap} if the number of resonances with $\Re\lambda\geq \Re\lambda_0-\nu$ is
finite. The size of an essential spectral gap on compact manifolds
is bounded from above, see Jin--Zworski~\cite{long}; a combination of
the techniques of~\cite{long} with those of the present paper could potentially lead
to a similar result in our more general setting.

The first resonances and the corresponding resonant states
capture key dynamical features of the flow. For Anosov flows in the scalar
case $\mathcal E=\mathbb C,\mathbf X=X$, zero is always a resonance since the function $1$ is a resonant state.
Moreover, resonances with $\Re\lambda=0$ have equal algebraic and geometric
multiplicity ($\Res_{\mathbf X}^{(j)}(\lambda)=\Res_{\mathbf X}^{(1)}(\lambda)$)
and the associated projectors $\Pi_\lambda$ are bounded on the space $C^0(\mathcal U)$
of continuous functions; this follows from Theorem~\ref{t:res-states} and the fact
that $\|\mathbf R(\lambda)\|_{C^0\to C^0}\leq (\Re\lambda)^{-1}$ for $\Re\lambda>0$.
The space of coresonant states $\Res_{\mathbf X^*}^{(1)}(0)$ consists
of Sinai--Ruelle--Bowen (SRB) measures. One consequence of this relation is the fact
that SRB measures depend smoothly on a parameter, if the vector field $X$ depends
smoothly on that parameter~-- this is known as \emph{linear response}.
Moreover, ergodicity of the flow with respect to the SRB measure
is equivalent to zero being a simple resonance,
and mixing is equivalent to zero being the first resonance.
See~\cite{but-liv} for details.

For weakly topologically mixing Axiom A flows, the first pole of the Ruelle zeta function $\zeta_V(\lambda)$ for $V\equiv 0$ is the topological
entropy $h_{\mathrm{top}}$ of the flow $\varphi^t$
and for general $V$ it gives the topological
pressure~-- see~\cite[Theorems~9.1, 9.2]{parry-pollicott}. This implies the asymptotic formula
$N^\sharp(T)\sim e^{h_{\mathrm{top}}T}/(h_{\mathrm top}T)$ for the number $N^\sharp(T)$ of primitive closed
trajectories of period less than $T$~-- see~\cite{parry-pollicott}. The associated (co)resonant
states in the Anosov case
are related to Margulis measures on the stable/unstable foliations and their
product is the measure of maximal entropy~-- see~\cite{margulis,bowen-marcus,glp}.

If one has an essential spectral gap of size $\nu$ with a polynomial (in $\lambda$) resolvent bound, then
there is a \emph{resonance expansion} with remainder $\mathcal O(e^{-(\nu-\Re\lambda_0)t})$ for
correlations $\langle e^{-t\mathbf X}\mathbf u,\mathbf v\rangle$, $\mathbf u\in C_0^\infty(\mathcal U;\mathcal E)$,
$\mathbf v\in C_0^\infty(\mathcal U;\mathcal E^*\otimes |\Omega|^1)$~-- see~\cite[Corollary~5]{NoZw2}.
For the Anosov case and $\mathcal E=\mathbb C$, $\mathbf X=X$, existence of a spectral gap was proved by Dolgopyat~\cite{dolgopyat} and Liverani~\cite{liverani}
(with subexponential decay of correlations earlier established by Chernov~\cite{chernov}); the precise size
of the essential gap was given by Tsujii~\cite{tsujii}. This followed earlier work
of Ratner~\cite{ratner} and Moore~\cite{moore} for locally symmetric spaces, including geodesic flows
on compact hyperbolic manifolds. (See~\cite{DFG} for a detailed description of resonances in the latter case.)

Regarding the noncompact case, Naud~\cite{naud}
established a spectral gap for $\mathcal E=\mathbb C$, $\mathbf X=X$ on convex co-compact hyperbolic surfaces. The
first resonance in that case is given by $\delta-1$, where $\delta$ is the exponent of convergence
of the Poincar\'e series of the fundamental group.
Similarly, a spectral gap for the Ruelle zeta function
implies an asymptotic formula for $N^\sharp(T)$ with an exponentially small remainder, see~\cite{glp}.
We also mention the work of Stoyanov~\cite{stoyanov2,stoyanov3}
on the spectral gap for the Ruelle zeta function of Axiom A flows under additional assumptions, as well as decay of correlations
for Gibbs measures, as well as the work~\cite{stoyanov} addressing these questions for contact Anosov flows.
The recent preprint of Petkov--Stoyanov~\cite{petkov-stoyanov} provides a spectral gap for transfer operators
depending on two complex parameters. We note that results mentioned in this paragraph give a holomorphic continuation
of the zeta function to a small strip past the domain of convergence under additional assumptions, such as
contact structure of the flow or the local non-integrability condition; our result gives a meromorphic continuation
to the entire complex plane without such additional assumptions, but at the cost of not establishing
a spectral gap.

%%%%%%%%%%%%%%%%%%%%%%%%%%%%%%%%%%%%%%%%%%%%%%%%%%%%%%%%%%%%%%%%%%%%%%%%%%%%%%%%
\smallsection{Previous work and methods of the proofs}
We finally give a brief overview of the history of the subject and explain
the methods used in the present paper.

Smale~\cite{Sm} defined Axiom~A flows and formulated a conjecture~\cite[pp.~802--803]{Sm} whether a certain
zeta function, related trivially to the Ruelle zeta function, extends meromorphically to $\mathbb C$,
admitting that `a positive answer would be a little shocking'. Ruelle~\cite{Ru0} gave a positive
answer to Smale's question for real analytic Anosov flows with analytic stable/unstable foliations;
the analyticity assumption on the foliations, but not on the flow itself, was removed in the works of
Rugh~\cite{rugh2} in dimension 3 and Fried~\cite{fried} in general dimensions.
Pollicott~\cite{Po0,Po} and Ruelle~\cite{Ru,Ru1} extended the zeta function to a small strip past the first pole
for general Axiom A flows and related its poles to decay of correlations. These papers,
as well as the previously mentioned work~\cite{dolgopyat,naud,stoyanov2,stoyanov} use \emph{Ruelle transfer operators},
which conjugate the flow to a shift on a space constructed by symbolic dynamics.
We refer the reader to the book of Parry--Pollicott~\cite{papo} for a detailed description of this approach.

Later, Pollicott--Ruelle resonances for the special case of Anosov flows
were interpreted as the eigenvalues of the generator of the flow
(or of the transfer operator $e^{-t\mathbf X}$) on suitably designed \emph{anisotropic spaces} which consist of functions
which are regular in the stable directions and irregular in the unstable directions.
These spaces fall into two categories:
\begin{itemize}
\item anisotropic H\"older spaces, studied by Liverani~\cite{liverani}, Butterley--Liverani~\cite{but-liv}, and for the related case of Anosov maps, by Blank--Keller--Liverani~\cite{b-k-l}, Liverani~\cite{liverani2}, and
Gou\"ezel--Liverani~\cite{go-liv}; and
\item anisotropic Sobolev spaces, studied for Anosov maps by Baladi--Tsujii~\cite{bats} and used in the
microlocal works discussed below.
\end{itemize}
Some similar ideas appeared already in the works of Rugh~\cite{rugh} in the analytic category and Kitaev~\cite{kitaev}.
Using anisotropic spaces, Pollicott--Ruelle resonances were defined in the entire complex plane and a meromorphic
continuation of the Ruelle zeta function to $\mathbb C$ for Anosov flows was proved by Giulietti--Liverani--Pollicott~\cite{glp}.

The work of Faure--Sj\"ostrand~\cite{FaSj} for Anosov flows (following the earlier work~\cite{fa-ro-sj} for Anosov maps
and the work~\cite{faure} on the prequantum cat map)
interpreted the equation $(\mathbf X+\lambda)\mathbf u=\mathbf f$ on anisotropic Sobolev spaces as a
\emph{scattering problem}. This used the methods of \emph{microlocal analysis} to consider
the operator $e^{-t\mathbf X}$ as quantizing a Hamiltonian flow $e^{tH_p}$ (see~\eqref{e:hammertime}) 
on the phase space $T^*\mathcal U$. In contrast with standard scattering problems (for the operator $-\Delta-\lambda^2$
on a noncompact Riemannian manifold), waves escape not to the spatial infinity $\{|x|=\infty\}$,
but to the fiber infinity $\{|\xi|=\infty\}$, and the anisotropic spaces provide the correct regularity
at the fiber infinity to make $(\mathbf X+\lambda)\mathbf u=\mathbf f$ into a Fredholm problem.
The microlocal approach made it possible to apply the methods of scattering theory to Anosov flows, resulting
in:
\begin{itemize}
\item sharp upper bound for the number of resonances in strips~\cite{DDZ} (improving the bound of~\cite{FaSj});
\item an essential spectral gap of optimal size~\cite{tsujii,NoZw2};
\item band structure for resonances of contact Anosov flows, including a Weyl law under the pinching condition
and meromorphic continuation of the Gutzwiller--Voros zeta function~\cite{FaTs1,FaTs2,FaTs3};
\item a microlocal proof of meromorphic continuation of Ruelle zeta function~\cite{DyZw}
(recovering the result of~\cite{glp});
\item definition of resonances as limits of the eigenvalues of $\mathbf X-\varepsilon\Delta$
as $\varepsilon\to 0+$, and stochastic stability of resonances~\cite{DyZw2}.
\end{itemize}
Our present work uses the microlocal method to obtain results for general open hyperbolic systems.
In particular, we use anisotropic Sobolev spaces to control the singularities at fiber infinity
and obtain $\mathbf R(\lambda)$ as the restriction of the resolvent of a Fredholm problem
in these spaces.
To show meromorphic continuation
of the zeta function, we use a wavefront set condition on $\mathbf R(\lambda)$ to ensure
that a certain flat trace can be defined; this trace is the continuation of $\zeta_V'/\zeta_V$.

However, compared to the Anosov case studied in~\cite{FaSj,DyZw}, the case of open hyperbolic
systems presents several additional difficulties. First of all,
the radial sets corresponding to the stable/unstable foliations are no longer sources or sinks,
but rather saddle sets (see Figure~\ref{f:funny3d} on page~\pageref{f:funny3d});
to handle them, we prove a propagation of singularities result (Lemma~\ref{l:ultimate})
which applies to a broad class of dynamical situations.

We also need to capture singularities which escape from $\mathcal U$. To do that,
we surround $\mathcal U$ by a slightly larger strictly convex set and multiply $X$ by a boundary
defining function of this set to make it vanish on the boundary. We then use complex absorbing potentials
on the boundary and complex absorbing pseudodifferential operators beyond the boundary (and near the glancing points)
to obtain a global Fredholm problem for the extension of $\mathbf X$ to a compact manifold
without boundary $\mathcal M$. 

We finally remark that the work of Arnoldi--Faure--Weich~\cite{afw} defined resonances
for certain open hyperbolic maps, while~\cite{FaTs3} defined resonances for the Grassmanian bundle
of an Anosov flow, which can be viewed as special cases of the open hyperbolic systems studied in the present paper.

%%%%%%%%%%%%%%%%%%%%%%%%%%%%%%%%%%%%%%%%%%%%%%%%%%%%%%%%%%%%%%%%%%%%%%%%%%%%%%%%
\smallsection{Structure of the paper}
Sections~\ref{s:dyn} and~\ref{s:semi} contain the necessary preliminary constructions;
Section~\ref{s:dyn} concerns hyperbolic dynamical systems and Section~\ref{s:semi},
microlocal and semiclassical analysis. The proofs of Theorems~\ref{t:mer} and~\ref{t:res-states}
are contained in Section~\ref{s:res}. Theorem~\ref{t:zeta} and the closely related Theorem~\ref{t:trace}
are proved in Section~\ref{s:trace-zeta}.
Finally, Section~\ref{s:examples} gives several examples of open hyperbolic systems, including
geodesic flows on certain complete negatively curved Riemannian manifolds.

%%%%%%%%%%%%%%%%%%%%%%%%%%%%%%%%%%%%%%%%%%%%%%%%%%%%%%%%%%%%%%%%%%%%%%%%%%%%%%%%
%%%%%%%%%%%%%%%%%%%%%%%%%%%%%%%%%%%%%%%%%%%%%%%%%%%%%%%%%%%%%%%%%%%%%%%%%%%%%%%%
\section{Dynamical preliminaries}
\label{s:dyn}

In this section, we discuss several dynamical corollaries of assumptions~\hyperlink{AA1}{\rm(A1)}--\hyperlink{AA5}{\rm(A5)}
in the introduction. In particular, in~\S\ref{s:dyn2}, we show how to extend the stable/unstable bundles
to $\Gamma_\pm$ (Lemma~\ref{l:extended}) and construct the components of the weight function
for the anisotropic Sobolev space (Lemma~\ref{l:functions}).

%%%%%%%%%%%%%%%%%%%%%%%%%%%%%%%%%%%%%%%%%%%%%%%%%%%%%%%%%%%%%%%%%%%%%%%%%%%%%%%%
\subsection{Basic properties}
  \label{s:dyn1}
  
We start by showing that the vector field $X$ can be extended from $\overline{\mathcal U}$ to a compact manifold
without boundary so that $\mathcal U$ is convex:
%%%%%%%%%%%%%%%%%%%%%%%%%%%%%%%%%%%%%%%%%%%%%%%%%%%%%%%%%%%%%%%%%%%%%%%%%%%%%%%%
\begin{lemm}
  \label{l:extconv}
Let $\mathcal U,X,\rho$ satisfy the assumptions~\hyperlink{AA1}{\rm(A1)}--\hyperlink{AA3}{\rm(A3)}
in the introduction. Then there exists a compact manifold without boundary
$\mathcal M\supset\overline{\mathcal U}$ and a smooth extension of $X$ to a vector
field on $\mathcal M$ such that~\eqref{e:convex2} holds.
Moreover, $\overline{\mathcal U}$ satisfies the convexity condition~\eqref{e:convex2} as well.
\end{lemm}
%%%%%%%%%%%%%%%%%%%%%%%%%%%%%%%%%%%%%%%%%%%%%%%%%%%%%%%%%%%%%%%%%%%%%%%%%%%%%%%%
\begin{proof}
We first embed $\overline{\mathcal U}$ into some compact manifold without boundary $\mathcal M$
(for example, by letting $\mathcal M$ be the doubling of $\overline{\mathcal U}$ across the boundary)
and extend the function $\rho$ to $\mathcal M$ so that
$\rho<0$ on $\mathcal M\setminus\overline{\mathcal U}$.
We next extend $X$ in an abitrary way to $\mathcal M$ and call the resulting vector field
$X_1$. It follows from~\eqref{e:convex} that for some constant $C>0$,
$$
X^2\rho<C(X\rho)^2\quad\text{on }\partial\mathcal U.
$$
By continuity, there exists $\varepsilon>0$ such that
\begin{equation}
  \label{e:extconv0}
X_1^2\rho<C(X_1\rho)^2\quad\text{on }\{|\rho|\leq 2\varepsilon\}.
\end{equation}
We now take $\psi\in C^\infty(\mathbb R)$ such that
$$
\psi(s)=1\quad\text{for }s\geq 0,\quad
\sgn \psi(s)=\sgn(s+\varepsilon),\quad
\psi'(-\varepsilon)>0.
$$
The extension of $X$ to $\mathcal M$ is then defined by
$$
X:=\psi(\rho) X_1;\quad
X\rho=\psi(\rho)X_1\rho,\quad
X^2\rho=\psi(\rho)^2X_1^2\rho+\psi(\rho)\psi'(\rho)(X_1\rho)^2.
$$
It follows from~\eqref{e:extconv0} that
\begin{equation}
  \label{e:extconv1}
|\rho(x)|\leq 2\varepsilon,\
\rho(x)\neq -\varepsilon,\
X\rho(x)=0\quad\Longrightarrow\quad
X^2\rho(x)<0.
\end{equation}
We now show that~\eqref{e:convex2} holds. Assume that
$x,\varphi^T(x)\in \mathcal U$ for some $T\geq 0$,
but $\varphi^t(x)\notin\mathcal U$ for some $t\in [0,T]$.
Denote $f(t)=\rho(\varphi^t(x))$.
Let $t_0$ be the point that
minimizes the value of $f$ on $[0,T]$.
By our assumptions, $f(t_0)\leq 0$ and thus
$t_0\in (0,T)$; it follows that $f'(t_0)=0$ and $f''(t_0)\geq 0$.
On the other hand, since $X$ vanishes
on $\{\rho=-\varepsilon\}$, we have
$f(t_0)>-\varepsilon$. By~\eqref{e:extconv1},
we have $f''(t_0)<0$, giving a contradiction.
The condition~\eqref{e:convex2} is verified for $\overline{\mathcal U}=\{\rho\geq 0\}$
by the same argument.
\end{proof}
%%%%%%%%%%%%%%%%%%%%%%%%%%%%%%%%%%%%%%%%%%%%%%%%%%%%%%%%%%%%%%%%%%%%%%%%%%%%%%%%
We henceforth assume that $X$ is extended to $\mathcal M$ in the manner
described in Lemma~\ref{l:extconv}, and put $\varphi^t:=e^{tX}$.
We next establish the topological properties of $\Gamma_\pm$ and $K$:
%%%%%%%%%%%%%%%%%%%%%%%%%%%%%%%%%%%%%%%%%%%%%%%%%%%%%%%%%%%%%%%%%%%%%%%%%%%%%%%%
\begin{lemm}
  \label{l:gpmclosed}
Let $K$ be defined in~\eqref{e:gpm}. Then $K\subset\mathcal U$.
\end{lemm}
%%%%%%%%%%%%%%%%%%%%%%%%%%%%%%%%%%%%%%%%%%%%%%%%%%%%%%%%%%%%%%%%%%%%%%%%%%%%%%%%
\begin{proof}
From~\eqref{e:gpm}, we see that $K\subset\overline{\mathcal U}$;
therefore it suffices to show that $K\cap\partial\mathcal U=\emptyset$.
Assume that $x\in K\cap \partial\mathcal U$. Then $\varphi^t(x)\in \overline{\mathcal U}$
for all $t\in\mathbb R$. Therefore, the function $f(t):=\rho(\varphi^t(x))$
has a local minimum at $t=0$, which contradicts~\eqref{e:convex}.
\end{proof}
%%%%%%%%%%%%%%%%%%%%%%%%%%%%%%%%%%%%%%%%%%%%%%%%%%%%%%%%%%%%%%%%%%%%%%%%%%%%%%%%

%%%%%%%%%%%%%%%%%%%%%%%%%%%%%%%%%%%%%%%%%%%%%%%%%%%%%%%%%%%%%%%%%%%%%%%%%%%%%%%%
\begin{lemm}
  \label{l:convergence}
Assume that $x\in \Gamma_\pm$. Then we have uniformly in $x$,
$$
\varphi^t(x)\to K\quad\text{as }t\to\mp\infty,
$$
where convergence is understood as follows: for each neighborhood of $K$, $\varphi^t(\Gamma_\pm)$ lies inside that neighborhood
for $\mp t$ large enough.
\end{lemm}
%%%%%%%%%%%%%%%%%%%%%%%%%%%%%%%%%%%%%%%%%%%%%%%%%%%%%%%%%%%%%%%%%%%%%%%%%%%%%%%%
\begin{proof}
We consider the case of $\Gamma_-$; the case of $\Gamma_+$ is handled similarly.
Since $\mathcal M$ is compact, it suffices to show that for each sequences
$t_j\to +\infty$, $x_j\in \Gamma_-$, if $\varphi^{t_j}(x_j)\to x_\infty\in \mathcal M$, then $x_\infty\in K$;
that is, $\varphi^t(x_\infty)\in \overline{\mathcal U}$ for all $t\in \mathbb R$.
This is true since $\varphi^t(x_\infty)$ is the limit of $\varphi^{t+t_j}(x_j)$; it remains to use that
$\varphi^{t+t_j}(x_j)\in \overline{\mathcal U}$ whenever $t+t_j\geq 0$, which happens for $j$ large enough.
\end{proof}
%%%%%%%%%%%%%%%%%%%%%%%%%%%%%%%%%%%%%%%%%%%%%%%%%%%%%%%%%%%%%%%%%%%%%%%%%%%%%%%%

%%%%%%%%%%%%%%%%%%%%%%%%%%%%%%%%%%%%%%%%%%%%%%%%%%%%%%%%%%%%%%%%%%%%%%%%%%%%%%%%
\begin{lemm}
  \label{l:long}
Let $V\subset \mathcal U$ be a neighborhood of $K$. Then there exists $T>0$
such that for each $x\in \overline{\mathcal U}$ such that $\varphi^{T}(x),\varphi^{-T}(x)\in\overline{\mathcal U}$, we have
$x\in V$.
\end{lemm}
%%%%%%%%%%%%%%%%%%%%%%%%%%%%%%%%%%%%%%%%%%%%%%%%%%%%%%%%%%%%%%%%%%%%%%%%%%%%%%%%
\begin{proof}
It suffices to show that for each sequences $T_j\to +\infty$, $x_j\in \overline{\mathcal U}$,
if  $x_j\to x_\infty$ and $\varphi^{T_j}(x_j),\varphi^{-T_j}(x_j)\in\overline{\mathcal U}$, then $x_\infty\in K$.
By Lemma~\ref{l:extconv}, we have $\varphi^t(x_j)\in\overline{\mathcal U}$ for $|t|\leq T_j$. Therefore,
$\varphi^t(x_\infty)\in\overline{\mathcal U}$ for all $t$, implying
that $x_\infty\in K$.
\end{proof}
%%%%%%%%%%%%%%%%%%%%%%%%%%%%%%%%%%%%%%%%%%%%%%%%%%%%%%%%%%%%%%%%%%%%%%%%%%%%%%%%

%%%%%%%%%%%%%%%%%%%%%%%%%%%%%%%%%%%%%%%%%%%%%%%%%%%%%%%%%%%%%%%%%%%%%%%%%%%%%%%%
\begin{lemm}
  \label{l:convfun}
Assume that $\chi\in C_0^\infty(\mathcal U)$. Then there exists
$\chi'\in C_0^\infty(\mathcal U)$ such that
\begin{equation}
  \label{e:convfun}
x,\varphi^T(x)\in \supp(\chi),\ T\geq 0\ \Longrightarrow\
\varphi^t(x)\notin\supp (1-\chi')\quad\text{for all }t\in [0,T].
\end{equation}
\end{lemm}
%%%%%%%%%%%%%%%%%%%%%%%%%%%%%%%%%%%%%%%%%%%%%%%%%%%%%%%%%%%%%%%%%%%%%%%%%%%%%%%%
\begin{proof}
Take $\delta\in(0,2\varepsilon)$ small enough such that $\supp\chi\subset \{\rho\geq\delta\}$
and choose $\chi'\in C_0^\infty(\mathcal U)$ such that
$\chi'=1$ near $\{\rho\geq\delta\}$. Similarly to the proof of Lemma~\ref{l:extconv},
we derive from \eqref{e:extconv0} that $\{\rho\geq\delta\}$ is convex; therefore, \eqref{e:convfun} holds.
\end{proof}
%%%%%%%%%%%%%%%%%%%%%%%%%%%%%%%%%%%%%%%%%%%%%%%%%%%%%%%%%%%%%%%%%%%%%%%%%%%%%%%%
We next derive several properties of the vector fields $X$ and $X_1$
near $\partial\mathcal U$:
%%%%%%%%%%%%%%%%%%%%%%%%%%%%%%%%%%%%%%%%%%%%%%%%%%%%%%%%%%%%%%%%%%%%%%%%%%%%%%%%
\begin{lemm}
  \label{l:marron0}
Let $\varepsilon,X_1,\rho$ be chosen in the proof of Lemma~\ref{l:extconv} and
take $\alpha,\beta\in [-2\varepsilon,2\varepsilon]$ such that
$\alpha\leq\beta$. Let $x\in\{\alpha\leq \rho\leq\beta\}$.

1. There exists $T\geq 0$ such that $e^{TX_1}(x)\in\{\rho=\alpha\}\cup\{\rho=\beta\}$.

2. If additionally $X_1\rho(x)\leq 0$, then there exists $T\geq 0$
such that $e^{TX_1}(x)\in\{\rho=\alpha\}$ and 
$X_1\rho(e^{tX_1}(x))<0$, $\rho(e^{tX_1}(x))\in [\alpha,\beta)$ for all $t\in (0, T]$.

Same is true when $X_1$ is replaced by $-X_1$.
\end{lemm}
%%%%%%%%%%%%%%%%%%%%%%%%%%%%%%%%%%%%%%%%%%%%%%%%%%%%%%%%%%%%%%%%%%%%%%%%%%%%%%%%
\begin{proof}
Denote $f(t):=\rho(e^{tX_1}(x))$. Then by~\eqref{e:extconv0}, there exists $\delta>0$
such that
$$
f''(t)+\delta\leq C(f'(t))^2\quad\text{if }|f(t)|\leq 2\varepsilon.
$$
Then for some $\delta_1>0$,
$$
g''(t)>\delta_1\quad\text{if }|f(t)|\leq 2\varepsilon,\quad
g(t):=e^{-Cf(t)}.
$$
It follows immediately that we cannot have $f(t)\in [\alpha,\beta]$ for all $t\geq 0$;
this implies part~1. To see part~2, we note that $X_1\rho(x)\leq 0$
implies that $g'(0)\geq 0$; then there exists $T\geq 0$ such that
$g(T)=e^{-C\alpha}$ and $g'(t)>0$, $g(t)\in (e^{-C\beta},e^{-C\alpha}]$ for all $t\in (0, T]$.
\end{proof}
%%%%%%%%%%%%%%%%%%%%%%%%%%%%%%%%%%%%%%%%%%%%%%%%%%%%%%%%%%%%%%%%%%%%%%%%%%%%%%%%

%%%%%%%%%%%%%%%%%%%%%%%%%%%%%%%%%%%%%%%%%%%%%%%%%%%%%%%%%%%%%%%%%%%%%%%%%%%%%%%%
\begin{lemm}
  \label{l:marron}
Let $\varepsilon,X_1$ be chosen in the proof of Lemma~\ref{l:extconv} and
consider the sets
\begin{equation}
  \label{e:Sigma-def}
\Sigma_\pm:=\bigcup_{\pm t\geq 0}\varphi^t(\mathcal U),\quad
\Sigma:=\Sigma_+\cup\Sigma_-.
\end{equation}
Then $\overline{\Sigma}\cap \{\rho=-\varepsilon\}\cap \{X_1\rho=0\}=\emptyset$.
\end{lemm}
%%%%%%%%%%%%%%%%%%%%%%%%%%%%%%%%%%%%%%%%%%%%%%%%%%%%%%%%%%%%%%%%%%%%%%%%%%%%%%%%
\begin{proof}
Take $x\in \{\rho=-\varepsilon\}\cap \{X_1\rho=0\}$. Then by~\eqref{e:extconv0},
the function $t\mapsto\rho(e^{tX_1}(x))$ has a nondegenerate local maximum at $t=0$.
Therefore, there exists $\delta>0$ such that
\begin{equation}
  \label{e:marron1}
e^{\pm\delta X_1}(x)\in \{\rho<-\varepsilon\},\quad
e^{tX_1}(x)\notin\overline{\mathcal U}\quad\text{for all }t\in [-\delta,\delta].
\end{equation}
Fix $\delta$ and take $x'$ in a small neighborhood of $x$. Then~\eqref{e:marron1} holds also for $x'$.
Since $X$ is a multiple of $X_1$ which vanishes on $\{\rho=-\varepsilon\}$,
it follows that the trajectory $\varphi^t(x')$ never passes through $\overline{\mathcal U}$; that is,
$x'\notin\Sigma$. It follows that $x\notin\overline\Sigma$, finishing the proof.
\end{proof}
%%%%%%%%%%%%%%%%%%%%%%%%%%%%%%%%%%%%%%%%%%%%%%%%%%%%%%%%%%%%%%%%%%%%%%%%%%%%%%%%

%%%%%%%%%%%%%%%%%%%%%%%%%%%%%%%%%%%%%%%%%%%%%%%%%%%%%%%%%%%%%%%%%%%%%%%%%%%%%%%%
\begin{lemm}
  \label{l:marron3}
Let $V_\pm\subset \Sigma_\pm\setminus\mathcal U$ be a compact set. Then there exists
a function
$$
\chi_\pm\in C_0^\infty(\{-2\varepsilon<\rho<\varepsilon\}\cap \{\pm X_1\rho<0\};[0,1])
$$
such that
$\pm X\chi_\pm\geq 0$ everywhere and $\pm X\chi_\pm>0$ on $V_\pm$.
(See Figure~\ref{f:marron}.)
\end{lemm}
%%%%%%%%%%%%%%%%%%%%%%%%%%%%%%%%%%%%%%%%%%%%%%%%%%%%%%%%%%%%%%%%%%%%%%%%%%%%%%%%
\begin{figure}
\includegraphics{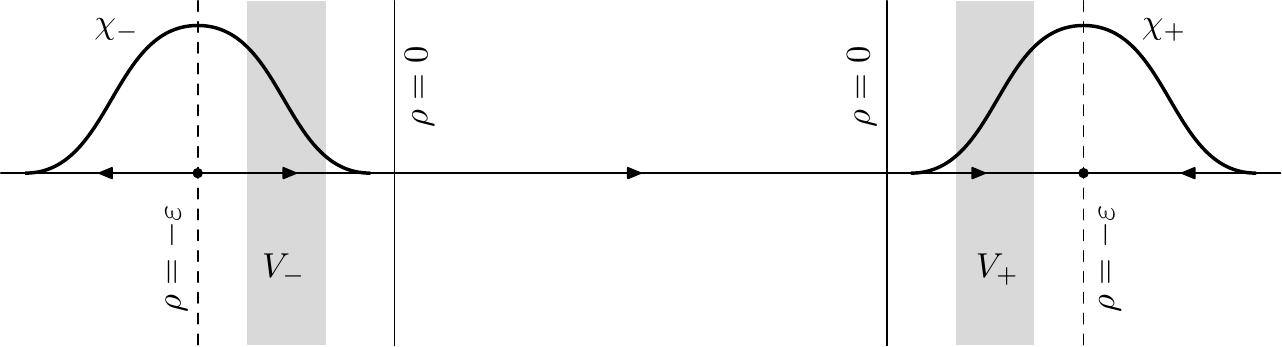}
\caption{A nontrapped trajectory of the vector field $X_1$ with the sets
$V_\pm$ and the functions $\chi_\pm$. The arrows indicate
the direction of the field $X=\psi(\rho)X_1$.}
\label{f:marron}
\end{figure}
%%%%%%%%%%%%%%%%%%%%%%%%%%%%%%%%%%%%%%%%%%%%%%%%%%%%%%%%%%%%%%%%%%%%%%%%%%%%%%%%
\begin{proof}
We construct $\chi_+$; the function $\chi_-$ is constructed similarly, reversing
the direction of the flow.
By compactness of $V_+$, it suffices to prove the lemma for the case
when $V_+=\{x_0\}$, where $x_0\in\Sigma_+\setminus\mathcal U$. Note that
$x_0\in \{\rho>-\varepsilon\}$ since $x_0\in\Sigma_+$ and $X$ vanishes on $\{\rho=-\varepsilon\}$.

We first claim that $X_1 \rho(x_0)<0$. Indeed, assume that $X_1\rho(x_0)\geq 0$.
Then by part~2 of Lemma~\ref{l:marron0} (with $[\alpha,\beta]=[-\varepsilon,0])$, there exists $T\geq 0$ such that
$e^{-TX_1}(x_0)\in\{\rho=-\varepsilon\}$ and $e^{-tX_1}(x_0)\notin\mathcal U$ for all $t\in [0,T]$.
Since $X=\psi(\rho)X_1$, we see that $\varphi^{-t}(x_0)\notin\mathcal U$ for all
$t\geq 0$, contradicting the fact that $x_0\in\Sigma_+$.

By part~2 of Lemma~\ref{l:marron0} (with $[\alpha,\beta]=[-\varepsilon,0]$), there exists
$T\geq 0$ such that $x_1:=e^{TX_1}(x_0)\in \{\rho=-\varepsilon\}$ and
$e^{tX_1}(x_0)\in\{-\varepsilon\leq \rho\leq 0\}$, $X_1\rho(e^{tX_1}(x_0))<0$ for all $t\in [0,T]$.
Let $U_1$ be a small neighborhood of $x_1$ in the surface $\{\rho=-\varepsilon\}$. Then for $\delta>0$ small enough,
the map
\begin{equation}
  \label{e:funkco}
(x',t)\in U_1\times (-T-\delta,\delta)\ \mapsto\ e^{tX_1}(x')\in\mathcal M
\end{equation}
is a diffeomorphism onto some open subset of $\{-2\varepsilon<\rho<\varepsilon\}\cap \{X_1\rho<0\}$.
Note that in the $(x',t)$ coordinates, $X=\psi(\rho)\partial_t$ and
$\sgn\psi(\rho)=-\sgn t$.
It remains to put in the $(x',t)$ coordinates,
$$
\chi_+(x',t)=\chi_0(x')\chi_1(t),
$$
where $\chi_0\in C_0^\infty(U_1;[0,1])$ satisfies $\chi_0(x_1)=1$
and $\chi_1\in C_0^\infty((-T-\delta,\delta);[0,1])$ satisfies
$t\chi'_1(t)\leq 0$ everywhere and $\chi'_1(-T)>0$. We finally extend
$\chi_+$ by zero to the entire $\mathcal M$.
\end{proof}
%%%%%%%%%%%%%%%%%%%%%%%%%%%%%%%%%%%%%%%%%%%%%%%%%%%%%%%%%%%%%%%%%%%%%%%%%%%%%%%%
We finally give the following property of the resolvent
$\mathbf R(\lambda)$ defined in~\eqref{e:res}.
%%%%%%%%%%%%%%%%%%%%%%%%%%%%%%%%%%%%%%%%%%%%%%%%%%%%%%%%%%%%%%%%%%%%%%%%%%%%%%%%
\begin{lemm}
  \label{l:outgoing2}
Assume that $\psi_1,\psi_2\in C_0^\infty(\mathcal U)$ satisfy
$\supp\psi_1\cap\Gamma_-=\supp\psi_2\cap\Gamma_+=\emptyset$.
Then the operators
$$
\mathbf R(\lambda)\psi_1,\ \psi_2\mathbf R(\lambda):C_0^\infty(\mathcal U)\to \mathcal D'(\mathcal U),\quad
\Re\lambda>C_0,
$$
extend holomorphically to $\lambda\in\mathbb C$.
\end{lemm}
%%%%%%%%%%%%%%%%%%%%%%%%%%%%%%%%%%%%%%%%%%%%%%%%%%%%%%%%%%%%%%%%%%%%%%%%%%%%%%%%
\begin{proof}
We establish holomorphic extension of $\mathbf R(\lambda)\psi_1$; the extension
of $\psi_2\mathbf R(\lambda)$ is handled similarly.
There exists $T>0$ such that $\varphi^t(\supp\psi_1)\cap\mathcal U=\emptyset$
for all $t\geq T$. Indeed, it is enough to show this for some $T=T(x_0)$ when the compact set $\supp\psi_1$ is replaced
by a small neighborhood $U_0$ of some fixed $x_0\in\supp\psi_1$. Since $x_0\notin\Gamma_-$, there exists
$T>0$ such that $\varphi^T(x_0)\notin\overline{\mathcal U}$. It follows that
$\varphi^T(x)\notin\overline{\mathcal U}$ when $x$ lies in a small neighborhood of $U_0$ of $x_0$.
By convexity of $\mathcal U$, it follows that $\varphi^t(x)\notin\mathcal U$ for $t\geq T$ and $x\in U_0$.

The holomorphic extension of $\mathbf R(\lambda)\psi_1$ is now given by the formula
\begin{equation}
  \label{e:holodeck}
\mathbf R(\lambda)\psi_1\mathbf f=\int_0^T \big(e^{-t(\mathbf X+\lambda)}\psi_1\mathbf f\big)\big|_{\mathcal U}\,dt ,
\end{equation}
where we used the fact that $\big(e^{-t(\mathbf X+\lambda)}\psi_1\mathbf f\big)\big|_{\mathcal U}=0$
for $t\geq T$, following from~\eqref{e:X-support}.
\end{proof}
%%%%%%%%%%%%%%%%%%%%%%%%%%%%%%%%%%%%%%%%%%%%%%%%%%%%%%%%%%%%%%%%%%%%%%%%%%%%%%%%

%%%%%%%%%%%%%%%%%%%%%%%%%%%%%%%%%%%%%%%%%%%%%%%%%%%%%%%%%%%%%%%%%%%%%%%%%%%%%%%%
\subsection{Hyperbolic sets}
  \label{s:dyn2}

We next express the assumption~\hyperlink{AA4}{\rm(A4)}
from the introduction in terms of the action of the differential on the dual space.
Define the function on the cotangent bundle $T^*\mathcal M$
\begin{equation}
  \label{e:symbol}
p(x,\xi)=\langle X(x),\xi\rangle,
\end{equation}
then its Hamiltonian flow is the action of $d\varphi^t$ on covectors:
\begin{equation}
  \label{e:hammertime}
e^{tH_p}(x,\xi)=(\varphi^t(x),(d\varphi^t(x))^{-T}\cdot \xi),\quad
x\in \mathcal M,\
\xi\in T^*_x\mathcal M,
\end{equation}
where $(d\varphi^t(x))^{-T}:T^*_x\mathcal M\to T^*_{\varphi^t(x)}\mathcal M$
is the inverse transpose of $d\varphi^t(x):T_x\mathcal M\to T_{\varphi^t(x)}\mathcal M$.
For each $x\in K$, define the dual stable/unstable decomposition
\begin{equation}
  \label{e:decdual}
T_x^*\mathcal M=E_0^*(x)\oplus E_s^*(x)\oplus E_u^*(x),
\end{equation}
where $E_0^*$ is the annihilator of $E_s\oplus E_u$, $E_s^*$ is the annihilator
of $E_0\oplus E_s$, and $E_u^*$ is the annihilator of $E_0\oplus E_u$. Note the
reversal of roles of $E_s,E_u$. By~\eqref{e:hyper}, we have
\begin{equation}
  \label{e:hyperdual}
|(d\varphi^t(x))^{-T}\cdot \xi|\leq Ce^{-\gamma|t|}|\xi|,\quad
\begin{cases}
t\geq 0,\ \xi\in E_s^*(x);\\
t\leq 0,\ \xi\in E_u^*(x).
\end{cases}
\end{equation}
We now extend the bundles $E_s^*,E_u^*$ to $\Gamma_-,\Gamma_+$ respectively, and
study the global dynamics of the flow $e^{tH_p}$:
%%%%%%%%%%%%%%%%%%%%%%%%%%%%%%%%%%%%%%%%%%%%%%%%%%%%%%%%%%%%%%%%%%%%%%%%%%%%%%%%
\begin{lemm}
  \label{l:extended}
There exist vector subbundles $E^*_\pm\subset T_{\Gamma_\pm}^* \mathcal M$ over $\Gamma_\pm$ such that:

1. $E_+^*|_K=E_u^*$, $E_-^*|_K=E_s^*$, and $E_\pm^*(x)$ depend continuously on $x\in\Gamma_\pm$.

2. $E_\pm^*$ are invariant under the flow $\varphi^t$
and $\langle X,\eta\rangle=0$ for $\eta\in E_\pm^*$.

3. If $x\in\Gamma_\pm$ and $\xi\in E_\pm^*(x)$, then as $t\to \mp\infty$
\begin{equation}
  \label{e:lapi3}
|(d\varphi^t(x))^{-T}\xi|\leq \widetilde Ce^{-\tilde\gamma |t|}|\xi|
\end{equation}
for some constants $\widetilde C,\tilde\gamma>0$ independent of $x,\xi$.

4. If $x\in\Gamma_\pm$ and $\xi\in T^*_x\mathcal M$ satisfies
$p(x,\xi)=0$ and $\xi\notin E_\pm^*(x)$, then
as $t\to \mp\infty$
\begin{equation}
  \label{e:lapi4}
|(d\varphi^t(x))^{-T} \xi|\to \infty,\quad
{(d \varphi^t(x))^{-T} \xi\over |(d\varphi^t(x))^{-T} \xi|}\to E_\mp^*|_K.
\end{equation}
\end{lemm}
%%%%%%%%%%%%%%%%%%%%%%%%%%%%%%%%%%%%%%%%%%%%%%%%%%%%%%%%%%%%%%%%%%%%%%%%%%%%%%%%
\begin{proof} We construct $E_-^*$; the bundle $E_+^*$ is contructed similarly.
The lemma is a natural consequence of the lamination of $\Gamma_-$ by the weak stable manifolds
$(W_s(x))_{x\in K}$ of the flow, where we put $E_-^*$ to be the annihilator of
the tangent space of $W_s(x)$, see for example~\cite[\S3.3]{NoZw}; the construction
of $W_s(x)$ ultimately relies on the Hadamard--Perron Theorem~\cite[Theorem~6.2.8]{KaHa}.
However, to make the paper more self-contained and since we only need a small portion
of the proof of the Hadamard--Perron theorem, we sketch a direct proof of the lemma below.

We fix some smooth Riemannian metric $\tilde g$ on $\mathcal M$ and measure the norms of
cotangent vectors with respect to this metric. Denote by $d_{\tilde g}(\cdot,\cdot)$ the
distance function induced by $\tilde g$. Take $\varepsilon>0$ small enough
to be fixed later; we in particular let $\varepsilon$ be smaller than the
injectivity radius of $(\mathcal M,\tilde g)$. (This constant is unrelated
to the one in Lemma~\ref{l:extconv}.) For $x,y\in\mathcal M$ such that
$d_{\tilde g}(x,y)<\varepsilon$, let
$$
\tau_{x\to y}:T^*_x\mathcal M\to T^*_y\mathcal M
$$
be the parallel transport along the shortest geodesic from $x$ to $y$.

Using~\eqref{e:hyperdual}, fix $t_0>0$ such that
for each $t\geq t_0$, $y\in K$ and $\eta\in T^*_y\mathcal M$,
$$
\begin{aligned}
|(d\varphi^{t}(y))^{-T} \eta|\leq {1\over 10}|\eta|,&\quad \eta\in E_s^*(y);\\
|(d\varphi^{t}(y))^{-T} \eta|\geq 10|\eta|,&\quad \eta\in E_u^*(y).
\end{aligned}
$$
For each $y\in K$, let
$$
\pi_s(y):T_y^*\mathcal M\to E_s^*(y),\quad
\pi_u(y):T_y^*\mathcal M\to E_u^*(y)
$$
be the projection maps corresponding to the decomposition~\eqref{e:hyperdual}.

For $x\in \mathcal M$, $y\in K$, and $d_{\tilde g}(x,y)<\varepsilon$, define the
\emph{dual stable/unstable cones} inside the annihilator of $X$ in $T_x^*\mathcal M$:
\begin{equation}
  \label{e:dualcones}
\begin{gathered}
\mathcal C^{(y)}_s(x)=\{\xi\in T_x^*\mathcal M\mid p(x,\xi)=0,\
|\pi_s(y)\tau_{x\to y}\xi|\geq |\pi_u(y)\tau_{x\to y}\xi|\},\\
\mathcal C^{(y)}_u(x)=\{\xi\in T_x^*\mathcal M\mid p(x,\xi)=0,\
|\pi_u(y)\tau_{x\to y}\xi|\geq |\pi_s(y)\tau_{x\to y}\xi|\}.
\end{gathered}
\end{equation}
Then for $\varepsilon$ small enough and each $t\in [t_0,2t_0]$, $y,y'\in K$, and $x\in\mathcal M$ such that
$d_{\tilde g}(x,y)<\varepsilon$, $d_{\tilde g}(\varphi^{t}(x),y')<\varepsilon$,
we have similarly to~\cite[Lemma~6.2.10]{KaHa}
\begin{equation}
  \label{e:cones1}
\begin{aligned}
(d\varphi^t(x))^{-T} \mathcal C^{(y)}_u(x)&\ \Subset\ \mathcal C^{(y')}_u(\varphi^t(x)),\\
(d\varphi^t(x))^{-T} \mathcal C^{(s)}_s(x)&\ \Supset\ \mathcal C^{(y')}_s(\varphi^t(x)).
\end{aligned}
\end{equation}
Indeed, \eqref{e:cones1} is verified directly for the case $x=y,\varphi^{t}(x)=y'$, and
it follows for small $\varepsilon$ by continuity.
Moreover, similarly to~\cite[Lemma~6.2.11]{KaHa} we find for $t\in [t_0,2t_0]$,
\begin{equation}
  \label{e:cones2}
\begin{aligned}
|(d\varphi^{t}(x))^{-T}\xi|\geq 4 |\xi|,&\quad \xi\in \mathcal C^{(y)}_u(x);\\
|(d\varphi^{-t}(x))^{-T} \xi|\geq 4 |\xi|,&\quad \xi\in \mathcal C^{(y)}_s(x).
\end{aligned}
\end{equation}
For $x\in \Gamma_-$, we define $E^*_-(x)$ as follows:
$\xi\in T_x^*\mathcal M$ lies in $E^*_-(x)$ if and only if $p(x,\xi)=0$
and there exists $t_1\geq 0$ such that
for all $t\geq t_1$ and each $y\in K$ such that $d_{\tilde g}(\varphi^t(x),y)<\varepsilon$,
we have $(d\varphi^t(x))^{-T}\xi\in \mathcal C^{(y)}_s(\varphi^t(x))$.
(Recall that $d_{\tilde g}(\varphi^t(x),K)\to 0$ as $t\to+\infty$ by Lemma~\ref{l:convergence}.)

By a straightforward adaptation of the proof
of~\cite[Proposition~6.2.12]{KaHa}, we see that $E^*_-(x)$ is a linear
subbundle of $T^*_{\Gamma_-}\mathcal M$ invariant under $\varphi^t$.
In fact, for each $t_j\to +\infty$ and $y_j\in K$ with $d_{\tilde g}(\varphi^{t_j}(x),y_j)<\varepsilon$,
we have
$$
E^*_-(x)=\lim_{j\to \infty} (d\varphi^{t_j}(x))^T\tau_{y_j\to \varphi^{t_j}(x)} E_s^*(y_j)
$$
where the limit is taken in the Grassmanian of $T_x^*\mathcal M$.
The fact that $E^*_-(x)=E_s^*(x)$ for $x\in K$ follows from here immediately, as
we can take $y_j:=\varphi^{t_j}(x)$. The bound~\eqref{e:lapi3} follows directly
from~\eqref{e:cones2}.

To show~\eqref{e:lapi4}, take $x\in\Gamma_-$ and $\xi\in T^*_x\mathcal M$
such that $p(x,\xi)=0$ and $\xi\notin E_-^*(x)$. By Lemma~\ref{l:convergence},
there exists $t_1\geq 0$ and $y_1\in K$ such that
$d_{\tilde g}(\varphi^{t_1}(x),y_1)<\varepsilon$ and
\begin{equation}
  \label{e:coid}
\begin{gathered}
d_{\tilde g}(\varphi^t(x),K)<\varepsilon\quad\text{for }t\geq t_1;\quad
(d\varphi^{t_1}(x))^{-T}\xi\notin \mathcal C_s^{(y_1)}(\varphi^{t_1}(x)).
\end{gathered}
\end{equation}
Iterating~\eqref{e:cones1}, we see
that
\begin{equation}
  \label{e:cones11}
(d\varphi^t(x))^{-T}\xi\in \mathcal C_u^{(y)}(\varphi^t(x)),\quad
t\geq t_0+t_1,
\end{equation}
for each $y\in K$ such that $d_{\tilde g}(\varphi^t(x),y)<\varepsilon$. Iterating~\eqref{e:cones2},
we get $|(d\varphi^t(x))^{-T}\xi|\to \infty$ as $t\to +\infty$.
To see the second part of~\eqref{e:lapi4}, it suffices to take
an arbitrary sequence $t_j\to +\infty$ such that
$$
\varphi^{t_j}(x)\to x_\infty\in K,\quad
{(d\varphi^{t_j}(x))^{-T}\xi\over |(d\varphi^{t_j}(x))^{-T}\xi|}\to \xi_\infty\in T_{x_\infty}^*\mathcal M
$$
and prove that $\xi_\infty\in E_u^*(x_\infty)$. Clearly $p(x_\infty,\xi_\infty)=0$. Next,
for each $t\geq 0$, we have
$$
(d\varphi^{-t}(x_\infty))^{-T}\xi_\infty=\lim_{j\to \infty}
{(d\varphi^{t_j-t}(x))^{-T}\xi\over |(d\varphi^{t_j}(x))^{-T}\xi|}\in \mathcal C_u^{(\varphi^{-t}(x_\infty))}(\varphi^{-t}(x_\infty)),
$$
which implies that $\xi_\infty\in E_u^*(x_\infty)$ as needed.

Finally, to show that $E_-^*(x)$ depends continuously on $x$, note that the condition~\eqref{e:coid}
is stable under perturbations of $x,\xi$ (recall that the convergence of Lemma~\ref{l:convergence} is uniform in $x$);
on the other hand, similarly to~\eqref{e:cones11} the condition~\eqref{e:coid}
implies that $\xi\notin E_-^*(x)$.
\end{proof}
%%%%%%%%%%%%%%%%%%%%%%%%%%%%%%%%%%%%%%%%%%%%%%%%%%%%%%%%%%%%%%%%%%%%%%%%%%%%%%%%

%%%%%%%%%%%%%%%%%%%%%%%%%%%%%%%%%%%%%%%%%%%%%%%%%%%%%%%%%%%%%%%%%%%%%%%%%%%%%%%%
\begin{figure}
\includegraphics{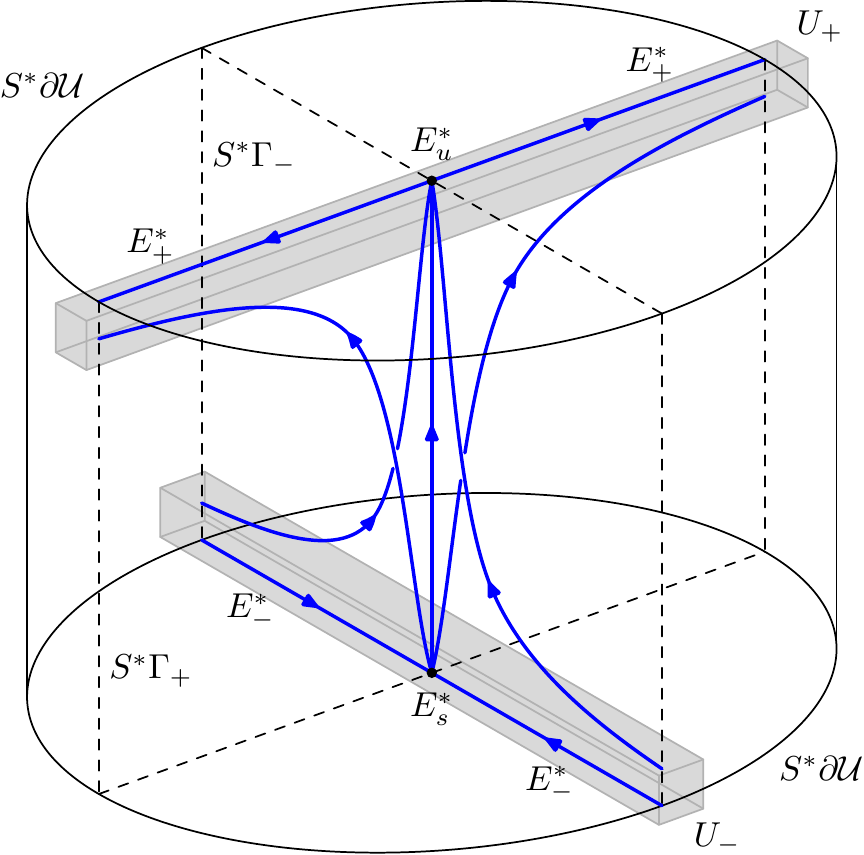}
\caption{A schematic representation of the flow lines $e^{tH_p}$ (thick blue lines)
on $S^*\mathcal U$, which is depicted
by the cylinder. The vertical direction in the picture
corresponds to the fibers of $S^*\mathcal U$.
The two vertical planes are $S^*\Gamma_\pm$, containing the subbundles $E_\pm^*$
(formally speaking, their images under $\kappa$);
the vertical line in the middle is $S^*K$, containing the subbundles $E_u^*,E_s^*$. The shaded
regions are the neighborhoods $U_\pm\supset E_\pm^*$.}
\label{f:funny3d}
\end{figure}
%%%%%%%%%%%%%%%%%%%%%%%%%%%%%%%%%%%%%%%%%%%%%%%%%%%%%%%%%%%%%%%%%%%%%%%%%%%%%%%%

The subbundle $E_+^*$ is a \emph{generalized radial sink} and $E_-^*$ is a \emph{generalized radial source}
in the following sense (this definition is a modification of~\cite[(2.12)]{DyZw}).
%%%%%%%%%%%%%%%%%%%%%%%%%%%%%%%%%%%%%%%%%%%%%%%%%%%%%%%%%%%%%%%%%%%%%%%%%%%%%%%%
\begin{lemm}
  \label{l:global-dynamics}
Let $\kappa:T^*\mathcal M\setminus 0\to S^*\mathcal M$ be the canonical projection, where
$S^*\mathcal M$ is the cosphere bundle over $\mathcal M$. Fix open neighborhoods
$U_\pm\subset S^*\mathcal M$ of $\kappa(E_\pm^*)$ such that 
$\overline{U_\pm}\cap \kappa(E_\mp^*)=\emptyset$ (see Figure~\ref{f:funny3d}). Then
for all $(x,\xi)\in T^*\mathcal M\setminus 0$ such that $p(x,\xi)=0$, $\kappa(x,\xi)\in U_\pm$,
and $x,\varphi^t(x)\in \overline{\mathcal U}$, we have
\begin{equation}
\begin{aligned}
d\big(\kappa(e^{tH_p}(x,\xi)),\kappa(E_\pm^*)\big)\to 0&\quad\text{as }t\to \pm\infty;\\
|(d\varphi^t(x))^{-T}\xi|\geq C^{-1}e^{\tilde\gamma|t|}|\xi|&\quad\text{for }\pm t\geq 0,
\end{aligned}
\end{equation}
uniformly in $(x,\xi)$. Here $d$ denotes any distance function on $S^*\mathcal M$
and $C,\tilde\gamma>0$ are constants independent of $x,\xi$.
\end{lemm}
%%%%%%%%%%%%%%%%%%%%%%%%%%%%%%%%%%%%%%%%%%%%%%%%%%%%%%%%%%%%%%%%%%%%%%%%%%%%%%%%
\begin{proof}
We study the trajectories starting in $U_+$ for $t\geq 0$;
the behavior in $U_-$ for $t\leq 0$ is proved similarly.
It suffices to show that for each sequences $(x_j,\xi_j),(y_j,\eta_j)\in T^*\mathcal M\setminus 0$
and $t_j\to +\infty$ such that $e^{t_jH_p}(x_j,\xi_j)=(y_j,\eta_j)$,
$p(x_j,\xi_j)=0$, $\kappa(x_j,\xi_j)\in U_+$,
and $x_j,y_j\in \overline{\mathcal U}$, we have
\begin{equation}
\label{e:intied}
d\big(\kappa(y_j,\eta_j),\kappa(E_\pm^*)\big)\to 0,\quad
|\eta_j|\geq C^{-1}e^{\tilde\gamma t_j}|\xi_j|.
\end{equation}
By passing to a subsequence, we may assume that $x_j\to x_\infty\in\overline{\mathcal U}$, $y_j\to y_\infty\in\overline{\mathcal U}$.
For each $t\geq 0$, $\varphi^t(x_\infty)=\lim_{j\to\infty} \varphi^t(x_j)\in\mathcal{\overline U}$,
therefore $x_\infty\in\Gamma_-$. Similarly $y_\infty\in\Gamma_+$. We also pass to a subsequence
to make $\xi_j/|\xi_j|\to\xi_\infty\in T^*_{x_\infty}\mathcal M$,
$\eta_j/|\eta_j|\to\eta_\infty\in T^*_{y_\infty}\mathcal M$, with $|\xi_\infty|=|\eta_\infty|=1$.
Since $\kappa(x_j,\xi_j)\in U_+$ and $\overline{U_+}$ does not intersect $\kappa(E_-^*)$, we have
$\xi_\infty\notin E_-^*(x_\infty)$.

For the first part of~\eqref{e:intied}, we need to prove that $\eta_\infty\in E_+^*(y_\infty)$.
Assume the contrary. We will use the proof of Lemma~\ref{l:extended}.
Similarly to~\eqref{e:cones11}, for $t_2>0$ large enough and each $x',y'\in K$ such that
$d_{\tilde g}(\varphi^{t_2}(x_\infty),x'),d_{\tilde g}(\varphi^{-t_2}(y_\infty),y')<\varepsilon$, we have
$$
(d\varphi^{t_2}(x_\infty))^{-T}\xi_\infty\notin\mathcal C^{(x')}_s(\varphi^{t_2}(x_\infty)),\quad
(d\varphi^{-t_2}(y_\infty))^{-T}\eta_\infty\notin\mathcal C^{(y')}_u(\varphi^{-t_2}(y_\infty)).
$$
It follows that for $t_2$ large and fixed, and $j$ large enough depending on $t_2$, we have
\begin{equation}
  \label{e:intieq}
(d\varphi^{t_2}(x_j))^{-T}\xi_j\notin\mathcal C^{(x')}_s(\varphi^{t_2}(x_j)),\quad
(d\varphi^{-t_2}(y_j))^{-T}\eta_j\notin\mathcal C^{(y')}_u(\varphi^{-t_2}(y_j)),
\end{equation}
for each $x',y'\in K$ such that $d_{\tilde g}(\varphi^{t_2}(x_j),x'),d_{\tilde g}(\varphi^{-t_2}(y_j),y')<\varepsilon$.
By Lemma~\ref{l:long}, we can furthermore fix $t_2$ large enough so that for $t_j\geq 2t_2$,
$d_{\tilde g}(\varphi^t(x_j),K)<\varepsilon$ for $t\in [t_2,t_j-t_2]$. Now, by the first statement
of~\eqref{e:intieq} and iterating~\eqref{e:cones1}, for $j$ large enough
and $x'\in K,d_{\tilde g}(x',\varphi^{t_j/2}(x_j))<\varepsilon$, we have
$(d\varphi^{t_j/2}(x_j))^{-T}\xi_j\notin\mathcal C^{(x')}_s(\varphi^{t_j/2}(x_j))$.
Similarly from the second statement of~\eqref{e:intieq} we get
$(d\varphi^{-t_j/2}(y_j))^{-T}\eta_j\notin\mathcal C^{(y')}_u(\varphi^{-t_j/2}(y_j))$.
However, these two vectors are the same, giving a contradiction and implying the first part of~\eqref{e:intied}.

The proof of the second part of~\eqref{e:intied} works in a similar fashion, using~\eqref{e:cones2}.
\end{proof}
%%%%%%%%%%%%%%%%%%%%%%%%%%%%%%%%%%%%%%%%%%%%%%%%%%%%%%%%%%%%%%%%%%%%%%%%%%%%%%%%

To construct the weight function for anisotropic Sobolev spaces,
we need the following adaptation of~\cite[Lemma~2.1]{FaSj} (see also~\cite[Lemma~C.1]{DyZw}).
We consider $e^{tH_p}$ as a flow on the sphere bundle $S^*\mathcal M$, by pulling it back
by the projection $\kappa:T^*\mathcal M\setminus 0\to S^*\mathcal M$.
Consider also the projection $\pi:S^*\mathcal M\to \mathcal M$.
%%%%%%%%%%%%%%%%%%%%%%%%%%%%%%%%%%%%%%%%%%%%%%%%%%%%%%%%%%%%%%%%%%%%%%%%%%%%%%%%
\begin{lemm}
\label{l:functions}
Let $U_\pm\subset S^*\mathcal M$ be the neighborhoods of $\kappa(E_\pm^*)$ introduced in Lemma~\ref{l:global-dynamics}.
Then there exist functions $m_\pm\in C^\infty(S^*\mathcal M)$ such that:
\begin{enumerate}
\item $m_\pm=1$ near $\kappa(E_\pm^*)$ and $0\leq m_\pm\leq 1$ everywhere;
\item $\supp m_\pm\cap \{p=0\}\cap\pi^{-1}(\overline{\mathcal U})\ \subset\ U_\pm$;
\item $\supp m_\pm\subset \pi^{-1}(\Sigma_\pm)$, where $\Sigma_\pm$ is defined in~\eqref{e:Sigma-def};
\item $\pm H_p m_\pm \geq 0$ on $V\cap\pi^{-1}(\mathcal U)$, where
$V$ is a neighborhood of $\{p=0\}$.
\end{enumerate}
\end{lemm}
%%%%%%%%%%%%%%%%%%%%%%%%%%%%%%%%%%%%%%%%%%%%%%%%%%%%%%%%%%%%%%%%%%%%%%%%%%%%%%%%
\begin{proof}
We construct $m_+$; the function $m_-$ is constructed similarly, reversing the direction of propagation.
Let $W\Subset\mathcal U$ be an open neighborhood of $K$.
Fix $m_0\in C_0^\infty(U_+\cap \pi^{-1}(\mathcal U))$ such that $m_0=1$ in a neighborhood of $\kappa(E_+^*)\cap \pi^{-1}(\overline W)$
and $0\leq m_0\leq 1$ everywhere.

We show that for $T>0$ large enough and fixed, the function
$$
m_+(x,\xi)={1\over T}\int_T^{2T} m_0(e^{-tH_p}(x,\xi))\,dt
$$
has the required properties:
\begin{enumerate}
\item Clearly $0\leq m_+\leq 1$ everywhere. Now, take $(x,\xi)\in \kappa(E_+^*)$. Then
$x\in \Gamma_+$. By Lemma~\ref{l:convergence}, $\varphi^{-t}(x)\in W$ for all $t\in [T,2T]$
and $T$ large enough. Since $E_+^*$ is invariant under the flow, we have
$e^{-tH_p}(x,\xi)\in \kappa(E_+^*)$ and thus $m_0(e^{-tH_p}(x,\xi))=1$ for $t\in [T,2T]$,
implying that $m(x,\xi)=1$. Same argument works when $(x,\xi)$ lies in a small neighborhood of $\kappa(E_+^*)$.
\item Assume that $(x,\xi)\in\supp m_+\cap \{p=0\}\cap \pi^{-1}(\overline{\mathcal U})$. Then there exists $t\in [T,2T]$
such that $e^{-tH_p}(x,\xi)\in\supp m_0$. Note that $x,\varphi^{-t}(x)\in\overline{\mathcal U}$ and $e^{-tH_p}(x,\xi)\in U_+$.
Then by Lemma~\ref{l:global-dynamics}, for $T$ large enough and $t\in [T,2T]$, we have
$(x,\xi)\in U_+$ as required.
\item This follows immediately from~\eqref{e:Sigma-def} and the fact that $\supp m_0\subset \pi^{-1}(\mathcal U)$.
\item Assume that $(x,\xi)\in S^*\mathcal M$, $x\in \mathcal U$, and $p(x,\xi)=0$. Then
$$
H_pm_+(x,\xi)={1\over T}\big(m_0(e^{-TH_p}(x,\xi))-m_0(e^{-2TH_p}(x,\xi))\big).
$$
We then need to show that $m_0(e^{-TH_p}(x,\xi))\geq m_0(e^{-2TH_p}(x,\xi))$.
Since $0\leq m_0\leq 1$, we only need to handle the case when
$m_0(e^{-TH_p}(x,\xi))<1$ and $m_0(e^{-2TH_p}(x,\xi))>0$. In particular,
we have $\varphi^{-2T}(x)\in\mathcal U$, and by Lemma~\ref{l:long}, for $T$ large enough,
we have $\varphi^{-T}(x)\in W$. Then $e^{-TH_p}(x,\xi)$ does not lie in some fixed neighborhood
$W_1$ of $\kappa(E_+^*)$, depending only on $m_0$. On the other hand,
$e^{-2TH_p}(x,\xi)\in U_+$ and $\varphi^{-2T}(x),\varphi^{-T}(x)\in\overline{\mathcal U}$.
By Lemma~\ref{l:global-dynamics}, we reach a contradiction
for $T$ large enough. Same reasoning applies if we replace the condition $p(x,\xi)=0$
by $(x,\xi)\in V$ for some neighborhood $V$ of $\{p=0\}$.
\qedhere
\end{enumerate}
\end{proof}
%%%%%%%%%%%%%%%%%%%%%%%%%%%%%%%%%%%%%%%%%%%%%%%%%%%%%%%%%%%%%%%%%%%%%%%%%%%%%%%%

%%%%%%%%%%%%%%%%%%%%%%%%%%%%%%%%%%%%%%%%%%%%%%%%%%%%%%%%%%%%%%%%%%%%%%%%%%%%%%%%
\subsection{Estimates on recurrence}
  \label{s:recur}

We finally give an extension of the recurrence estimates~\cite[Appendix~A]{DyZw}
to our situation, used in \S\ref{s:trace}. Throughout this subsection,
we fix $t_e>0$ and a compact subset $V\subset\mathcal U$. We also
consider the distance function $d_{\tilde g}$ and the parallel transport operators $\tau_{x\to y}$
introduced in the proof of Lemma~\ref{l:extended}, defined for $d_{\tilde g}(x,y)<\varepsilon$, where
$\varepsilon>0$ is a small constant (unrelated to the constant in Lemma~\ref{l:extconv}). We however
ask that $\tau_{x\to y}$ act on the tangent spaces
$T_x\mathcal M\to T_y\mathcal M$ instead of the cotangent spaces.
We start with
%%%%%%%%%%%%%%%%%%%%%%%%%%%%%%%%%%%%%%%%%%%%%%%%%%%%%%%%%%%%%%%%%%%%%%%%%%%%%%%%
\begin{lemm}
  \label{l:recur0}
For each $\varepsilon_1>0$, there exists $\delta_1>0$ such that
$$
d_{\tilde g}(x,\varphi^t(x))<\delta_1,\
t\geq t_e,\
x\in V\ \Longrightarrow\
d_{\tilde g}(x,K)<\varepsilon_1.
$$  
\end{lemm}
%%%%%%%%%%%%%%%%%%%%%%%%%%%%%%%%%%%%%%%%%%%%%%%%%%%%%%%%%%%%%%%%%%%%%%%%%%%%%%%%
\begin{proof}
It suffices to show that
for each sequences $x_j\in V$, $t_j\geq t_0$ such that $x_j\to x_\infty\in\mathcal U$ and
$d(x_j,\varphi^{t_j}(x_j))\to 0$,
we have $x_\infty\in K$. We have $\varphi^{t_j}(x_j)\to x_\infty$.
By passing to a subsequence we may
assume that $t_j\to t_\infty\in(0,\infty]$. If $t_\infty<\infty$, then
$\varphi^{t_\infty}(x_\infty)=x_\infty$
and thus $x_\infty\in K$. Assume now that $t_\infty=\infty$. For each $t\geq 0$ and
$j$ large enough depending on $t$, we have $t_j\geq t$ and
$x_j,\varphi^{t_j}(x_j)\in\mathcal U$; by~\eqref{e:convex2},
$\varphi^t(x_j)\in\mathcal U$ and $\varphi^{t_j-t}(x_j)\in\mathcal U$.
Passing to the limit, we see that $\varphi^t(x_\infty)$ and $\varphi^{-t}(x_\infty)$
lie in $\overline{\mathcal U}$; since $t$ was chosen arbitrarily, we get $x_\infty\in K$.
\end{proof}
%%%%%%%%%%%%%%%%%%%%%%%%%%%%%%%%%%%%%%%%%%%%%%%%%%%%%%%%%%%%%%%%%%%%%%%%%%%%%%%%
Denote by $\pi^\perp:T\mathcal U\to T\mathcal U$ the orthogonal projection
onto the orthogonal complement of $E_0=\mathbb RX$ (with respect to some fixed Riemannian metric).
This operator need not be invariant under $d\varphi^t$ and its image need not be equal to $E_u\oplus E_s$.
However, there exists a constant $C>0$ such that for each $x\in K$ and $v=v_0+v_u+v_s\in T_x\mathcal M$,
$v_0\in E_0(x)$, $v_u\in E_u(x)$, $v_s\in E_s(x)$,
\begin{equation}
  \label{e:samenorm}
C^{-1}(|v_u|+|v_s|)\leq |\pi^\perp(v)|\leq C(|v_u|+|v_s|).
\end{equation}
Moreover, $\pi^\perp(d\varphi^t(x)\cdot v)$ depends only on $\pi^\perp (v)$:
\begin{equation}
  \label{e:samenorm2}
\pi^\perp(d\varphi^t(x)\cdot v)=\pi^\perp(d\varphi^t(x)\cdot\pi^\perp (v))
\end{equation}
The next lemma gives a convexity property
for the absolute value of a vector propagated along the flow.
%%%%%%%%%%%%%%%%%%%%%%%%%%%%%%%%%%%%%%%%%%%%%%%%%%%%%%%%%%%%%%%%%%%%%%%%%%%%%%%%
\begin{lemm}
  \label{l:kindaconvex}
There exists $T_0>0$ such that for each $t\geq T_0$, $\varepsilon_t>0$
small enough depending on $t$, and each
$(x,v)\in T\mathcal U$ with $d_{\tilde g}(x,K)<\varepsilon_t$,
\begin{align}
  \label{e:kindaconvex}
|\pi^\perp(v)|&\leq{|\pi^\perp(d\varphi^t(x)\cdot v)|+|\pi^\perp(d\varphi^{-t}(x)\cdot v)|\over 4}.
\end{align}
\end{lemm}
%%%%%%%%%%%%%%%%%%%%%%%%%%%%%%%%%%%%%%%%%%%%%%%%%%%%%%%%%%%%%%%%%%%%%%%%%%%%%%%%
\begin{proof}
Assume first that $x\in K$. Then $v=v_0+v_u+v_s$, where
$v_0\in E_0(x),v_u\in E_u(x),v_s\in E_s(x)$.
By~\eqref{e:hyper}, there exists a constant $C$ such that for all $t\geq T_0$,
$$
|v_u|\leq Ce^{-\gamma T_0}|d\varphi^t(x)\cdot v_u|,\quad
|v_s|\leq Ce^{-\gamma T_0}|d\varphi^{-t}(x)\cdot v_s|,\quad
$$
Adding these up and using~\eqref{e:samenorm}, we get for some other constant $C$,
$$
|\pi^\perp (v)|\leq Ce^{-\gamma T_0}(|\pi^\perp (d\varphi^{t}(x)\cdot v)|+|\pi^\perp (d\varphi^{-t}(x)\cdot v)|).
$$
It remains to take $T_0$ large enough.
The case of $x$ with $d_{\tilde g}(x,K)<\varepsilon_t$ follows by continuity.
\end{proof}
%%%%%%%%%%%%%%%%%%%%%%%%%%%%%%%%%%%%%%%%%%%%%%%%%%%%%%%%%%%%%%%%%%%%%%%%%%%%%%%%
The following is a generalization of~\cite[Lemma~A.1]{DyZw}:
%%%%%%%%%%%%%%%%%%%%%%%%%%%%%%%%%%%%%%%%%%%%%%%%%%%%%%%%%%%%%%%%%%%%%%%%%%%%%%%%
\begin{lemm}
  \label{l:recur1}
There exist $\delta>0$ and $C$ such that for each $x\in V, t\geq t_e, v\in T_x\mathcal M$
satisfying $d_{\tilde g}(x,\varphi^t(x))<\delta$ and $v\perp X(x)$,
$$
|v|+|w|\leq C|\pi^\perp(w)|,\quad
w:=(d\varphi^t(x)-\tau_{x\to\varphi^t(x)})v.
$$
\end{lemm}
%%%%%%%%%%%%%%%%%%%%%%%%%%%%%%%%%%%%%%%%%%%%%%%%%%%%%%%%%%%%%%%%%%%%%%%%%%%%%%%%
\begin{proof}
It suffices to show that for each sequences
$$
x_j\in V,\quad t_j\geq t_e,\quad 
v_j\perp X(x_j),\quad
w_j:=(d\varphi^{t_j}(x_j)-\tau_{x_j\to \varphi^{t_j}(x_j)})v_j
$$
such that
$$
d_{\tilde g}(x_j,\varphi^{t_j}(x_j))\to 0,\quad
\pi^\perp(w_j)\to 0,
$$
we have $v_j\to 0$ and $w_j\to 0$.
By passing to a subsequence and using Lemma~\ref{l:recur0}, we may assume that
$$
x_j\to x_\infty\in K,\quad
\varphi^{t_j}(x_j)\to x_\infty,\quad
t_j\to t_\infty\in (0,\infty].
$$
Assume first that $t_\infty<\infty$. By passing to a subsequence, we may 
assume that $v_j/|v_j|\to v_\infty\perp X(x_\infty)$. We have
$\varphi^{t_\infty}(x_\infty)=x_\infty$
and $w_j/|v_j|\to d\varphi^{t_\infty}(x_\infty)\cdot v_\infty-v_\infty$.
By~\eqref{e:hyper}, $\pi^\perp(w_j)/|v_j|$ has a nonzero limit;
since $\pi^\perp(w_j)\to 0$, we get $v_j\to 0$ and thus $w_j\to 0$.

We henceforth assume that $t_\infty=\infty$. We first show that $v_j\to 0$. 
Assume the contrary, then by passing to a subsequence and rescaling,
we can make
$$
v_j\to v_\infty\perp X(x_\infty),\quad |v_\infty|=1.
$$
Consider the following two cases:

\noindent\textbf{Case 1}: $v_\infty$ has a nonzero $E_u$ component in the decomposition~\eqref{e:hypersplit}.
By~\eqref{e:hyper}, we have $|\pi^\perp(d\varphi^t(x_\infty)\cdot v_\infty)|\to \infty$ as $t\to +\infty$.
Let $T_0$ be the constant from Lemma~\ref{l:kindaconvex}. Fix $T\geq T_0$ so that
$|\pi^\perp(d\varphi^T(x_\infty)\cdot v_\infty)|>2$.
For $j$ large enough, we have
\begin{equation}
  \label{e:base}
|\pi^\perp (d\varphi^T(x_j)\cdot v_j)|\geq 2|v_j|.
\end{equation}
Moreover, if $\varepsilon_T$ is chosen in Lemma~\ref{l:kindaconvex}, then for $j$ large enough,
\begin{equation}
  \label{e:tight}
d_{\tilde g}(e^{tH_p}(x_j),K)<\varepsilon_T\quad\text{for all }t\in [0, t_j].
\end{equation}
Indeed, for $t\in [T',t_j-T']$ and $T'$ large enough depending on $\varepsilon_T$, this follows from Lemma~\ref{l:long};
for other values of $t$, it follows from continuity and the fact
that both $x_j$ and $\varphi^{t_j}(x_j)$ converge to $x_\infty\in K$.

We have for each $\ell\in\mathbb N_0$ such that $\ell T\leq t_j$,
$$
|\pi^\perp (d\varphi^{(\ell+1) T}(x_j)\cdot v_j)|\geq 2|\pi^\perp (d\varphi^{\ell T}(x_j)\cdot v_j)|.
$$
This is proved by induction on $\ell$; the base $\ell=0$ of the induction is given by~\eqref{e:base}
and the inductive step follows from~\eqref{e:kindaconvex} applied to $v=d\varphi^{\ell T}(x_j)\cdot v_j$, $t=T$. We can modify $T$
a tiny bit depending on $j$ so that $t_j/T$ is an integer; then we obtain
$$
|\pi^\perp (d\varphi^{t_j}(x_j)\cdot v_j)|\geq 2^{t_j/T}|v_j|.
$$
This implies that $|\pi^\perp(w_j)|\to \infty$, a contradiction.

\noindent\textbf{Case 2}: $v_\infty$ has a nonzero $E_s$ component in the decomposition~\eqref{e:hypersplit}.
Since $\pi^\perp(w_j)\to 0$, we have $\pi^\perp(d\varphi^{t_j}(x_j)\cdot v_j)\to v_\infty$. Arguing as in case (i),
with $\varphi^{t_j}(x_j),d\varphi^{t_j}(x_j)\cdot v_j$ replacing $x_j,v_j$, 
and going backwards along the flow, we get
$$
|\pi^\perp (d\varphi^{t_j-(\ell+1)T}(x_j)\cdot v_j)|
\geq 2|\pi^\perp (d\varphi^{t_j-\ell T}(x_j)\cdot v_j)|
$$
which implies
$$
|\pi^\perp(d\varphi^{t_j}(x_j)\cdot v_j)|\leq 2^{-t_j/T}|v_j|.
$$
Then $\pi^\perp(w_j)\to -v_\infty$, a contradiction.

We now show that $w_j\to 0$. Let $T_0$ be the constant from Lemma~\ref{l:kindaconvex}
and fix $T>T_0$; we will modify it a little bit depending on $j$
so that $L:=t_j/T$ is an integer. For large $j$, \eqref{e:tight} is satisfied.
For each $v\in T\mathcal U$, define $\pi_0(v)\in\mathbb R$ by the formula
$$
v=\pi^\perp(v)+\pi_0(v)X.
$$
Since $X$ is invariant under the flow, we have for some constant $C$,
$$
|\pi_0(d\varphi^{(\ell+1)T}(x_j)\cdot v_j)-\pi_0(d\varphi^{\ell T}(x_j)\cdot v_j)|\leq
C|\pi^\perp(d\varphi^{\ell T}(x_j)\cdot v_j)|.
$$
Summing these up and using that $\pi_0(v_j)=0$, we get
$$
|\pi_0(d\varphi^{t_j}(x_j)\cdot v_j)|\leq C\sum_{\ell=0}^{L}|\pi^\perp(d\varphi^{\ell T}(x_j)\cdot v_j)|.
$$
Denote the sum on the right-hand side by $\Sigma$. Using~\eqref{e:kindaconvex} for $v=d\varphi^{\ell T}(x_j)\cdot v_j$
and all $\ell=1,\dots,L-1$, we get
$$
\Sigma\leq |\pi^\perp(v_j)|+|\pi^\perp(d\varphi^{t_j}(x_j)\cdot v_j)|+\Sigma/2.
$$
Since $\pi^\perp(w_j)\to 0$ and $v_j\to 0$, we know that
$|\pi^\perp(v_j)|+|\pi^\perp(d\varphi^{t_j}(x_j)\cdot v_j)|\to 0$
and thus $\Sigma\to 0$. Then $\pi_0(d\varphi^{t_j}(x_j)\cdot v_j)\to 0$,
which implies that $w_j\to 0$, as required.
\end{proof}
%%%%%%%%%%%%%%%%%%%%%%%%%%%%%%%%%%%%%%%%%%%%%%%%%%%%%%%%%%%%%%%%%%%%%%%%%%%%%%%%
Arguing as in~\cite[Appendix~A]{DyZw}, we obtain from Lemma~\ref{l:recur1} the following
analog of~\cite[Lemma~2.1]{DyZw}:
%%%%%%%%%%%%%%%%%%%%%%%%%%%%%%%%%%%%%%%%%%%%%%%%%%%%%%%%%%%%%%%%%%%%%%%%%%%%%%%%
\begin{lemm}
  \label{l:recur2}
Define the following measure on $\mathcal M\times\mathbb R$: $\tilde\mu=\mu\times dt$, where
$\mu$ is some smooth measure on $\mathcal M$. Fix $t_e>0$ and a compact subset $V\subset \mathcal U$.
Then there exist constants $C,L$ such that for each $\varepsilon>0$, $T>t_e$, and $n=\dim\mathcal M$,
$$
\tilde\mu(\{(x,t)\mid t_e\leq t\leq T,\ d(x,\varphi^t(x))<\varepsilon,\ x\in V\})\leq C\varepsilon^n e^{nLT}.
$$
\end{lemm}
%%%%%%%%%%%%%%%%%%%%%%%%%%%%%%%%%%%%%%%%%%%%%%%%%%%%%%%%%%%%%%%%%%%%%%%%%%%%%%%%
Letting $\varepsilon\to 0$, we obtain the following analog of~\cite[Lemma~2.2]{DyZw}:
%%%%%%%%%%%%%%%%%%%%%%%%%%%%%%%%%%%%%%%%%%%%%%%%%%%%%%%%%%%%%%%%%%%%%%%%%%%%%%%%
\begin{lemm}
  \label{l:recur3}
Let $N(T)$ be the number of closed trajectories of $\varphi^t$ on $K$ of period no more
than $T$. Then
$$
N(T)\leq Ce^{(2n-1)LT}.
$$
\end{lemm}
%%%%%%%%%%%%%%%%%%%%%%%%%%%%%%%%%%%%%%%%%%%%%%%%%%%%%%%%%%%%%%%%%%%%%%%%%%%%%%%%

%%%%%%%%%%%%%%%%%%%%%%%%%%%%%%%%%%%%%%%%%%%%%%%%%%%%%%%%%%%%%%%%%%%%%%%%%%%%%%%%
%%%%%%%%%%%%%%%%%%%%%%%%%%%%%%%%%%%%%%%%%%%%%%%%%%%%%%%%%%%%%%%%%%%%%%%%%%%%%%%%
\section{Semiclassical preliminaries}
  \label{s:semi}

In this section, we discuss some general results from microlocal and semiclassical analysis, following
the notation of~\cite[Section~2.3 and Appendix~C]{DyZw}. While some of the facts mentioned
here (such as Lemma~\ref{l:propagation}) are standard, Lemma~\ref{l:ultimate} below
seems to be a new result.

%%%%%%%%%%%%%%%%%%%%%%%%%%%%%%%%%%%%%%%%%%%%%%%%%%%%%%%%%%%%%%%%%%%%%%%%%%%%%%%%
\subsection{Review of semiclassical notation}
  \label{s:notation}

Recall that we are working
on a compact manifold $\mathcal M$ without boundary.
We use the class $\Psi^k(\mathcal M;\mathcal E)$ of pseudodifferential operators of order $k$ acting
on sections of $\mathcal E$. The corresponding symbol class is denoted by~$S^k(\mathcal M)$, see~\cite[(C.1)]{DyZw}.
The principal symbol
$$
\sigma(\mathbf A)\in S^k(\mathcal M;\End(\mathcal E))/S^{k-1}(\mathcal M;\End(\mathcal E))
$$
of $\mathbf A\in \Psi^k(\mathcal M;\mathcal E)$
is in general a section of the endomorphism bundle $\End(\mathcal E)$ pulled back to $T^*\mathcal M$, however
in this paper we mostly work with \emph{principally scalar} operators, whose principal symbols
are products of functions on $T^*\mathcal M$ and the identity homomorphism on $\mathcal E$.
The \emph{wavefront set} $\WF(A)$ is a closed conic subset of $T^*\mathcal M\setminus 0$ which measures the concentration of~$A$
in the phase space, and the \emph{elliptic set} $\Ell(A)$ is an open conic subset of $T^*\mathcal M\setminus 0$
which measures where the principal symbol of $A$ is invertible.

We also use the class of \emph{semiclassical} pseudodifferential operators
$\Psi^k_h(\mathcal M;\mathcal E)$, which depend on a positive parameter $h$ tending to zero.
Quantizing a symbol $a(x,\xi)$ in the $h$-sense is equivalent to quantizing the rescaled
symbol $a(x,h\xi)$ in the nonsemiclassical sense. We use the notion of the semiclassical principal symbol
$$
\sigma_h(A)\in S^k(\mathcal M)/hS^{k-1}(\mathcal M)
$$
of
a principally scalar $A\in\Psi^k_h(\mathcal M;\mathcal E)$.
We also use the \emph{fiber-radially compactified} cotangent bundle $\overline{T}^*\mathcal M$;
the interior of this bundle is diffeomorphic to $T^*\mathcal M$ and the boundary $\partial\overline{T}^*\mathcal M$,
called the \emph{fiber infinity}, is diffeomorphic to the cosphere bundle $S^*\mathcal M$.
The $h$-wavefront set $\WFh(A)$ and the $h$-elliptic set $\Ell_h(A)$ are now subsets of $\overline T^*\mathcal M$.
We use the symbol $\Psi^{\comp}_h$ to denote the class of operators in $\Psi^k_h$ whose wavefront sets
are compactly contained in $T^*\mathcal M$ (that is, do not intersect the fiber infinity).

We use the concept of the wavefront set $\WF(u)\subset T^*\mathcal M\setminus 0$ of any distribution $u\in\mathcal D'(\mathcal M)$.
We also consider wavefront sets $\WF'(B)\subset T^*(\mathcal M\times\mathcal M)\setminus 0$
of operators $B:C^\infty(\mathcal M)\to\mathcal D'(\mathcal M)$, defined as follows:
\begin{equation}
  \label{e:wfprime}
\WF'(B)=\{(x,\xi,y,-\eta)\mid (x,\xi,y,\eta)\in\WF(K_B)\}
\end{equation}
where the Schwartz kernel $K_B\in \mathcal D'(\mathcal M\times\mathcal M)$ is given by the formula
(where we use any smooth density $dy$ on $\mathcal M$)
\begin{equation}
  \label{e:schwartz}
Bf(x	)=\int_{\mathcal M} K_B(x,y) f(y)\,dy,\quad
f\in C^\infty(\mathcal M).
\end{equation}
For distributions $u=u(h)$ and
operators $B=B(h)$ which are $h$-tempered (in the sense
that $\|u(h)\|_{H^{-N}_h}=\mathcal O(h^{-N})$ for some $N$), we consider the semiclassical wavefront sets
$\WFh(u)\subset\overline T^*\mathcal M$, $\WFh'(B)\subset\overline T^*(\mathcal M\times\mathcal M)$.
By taking the union of the wavefront sets of all components, we can extend these notions
to distributions and operators valued in smooth vector bundles.

We will use the following multiplicative property of $h$-wavefront sets away from fiber infinity:
assume that $A(h),B(h): C^\infty(\mathcal M)\to\mathcal D'(\mathcal M)$ are $h$-tempered and
$Q\in\Psi^{\comp}_h(\mathcal M)$. Using~\cite[Lemma~2.3]{DyZw}, we obtain
\begin{equation}
  \label{e:wf-mul}
\begin{gathered}
(x,\xi,z,\zeta)\in \WFh(AQB)\cap T^*(\mathcal M\times\mathcal M)\\\Longrightarrow\ \exists (y,\eta)\in \WFh(Q):
(x,\xi,y,\eta)\in\WFh(A),\ (y,\eta,z,\zeta)\in\WFh(B).
\end{gathered}
\end{equation}

Finally, if $u\in \mathcal D'(\mathcal V)$, where $\mathcal V\subset\mathcal M$ is an open set,
then $\WF(u)\subset T^*\mathcal V\setminus 0$ is defined as the union of all $\WF(\chi u)$ for
$\chi\in C_0^\infty(\mathcal V)$; here $\chi u$ is naturally embedded into $\mathcal D'(\mathcal M)$.
Similarly one can define $\WF(B)\subset T^*(\mathcal U\times\mathcal U)\setminus 0$, where $B:C_0^\infty(\mathcal U)\to \mathcal D'(\mathcal U)$ and
$\mathcal U\subset\mathcal M$ is open, by using~\eqref{e:wfprime} and the previous definition
with $\mathcal V:=\mathcal U\times\mathcal U$.

%%%%%%%%%%%%%%%%%%%%%%%%%%%%%%%%%%%%%%%%%%%%%%%%%%%%%%%%%%%%%%%%%%%%%%%%%%%%%%%%
\subsection{Semiclassical propagation estimates}

We start with several semiclassical estimates which form the basis of our proofs.
To simplify their statements, we say for $p\in C^\infty(T^*\mathcal M)$
that
$$
p\in\Hom^k(T^*\mathcal M;\mathbb R)
$$
if $p$ is real-valued and homogeneous of degree $k$ in $\xi$ for $|\xi|$ large enough.
If $p\in\Hom^1(T^*\mathcal M;\mathbb R)$, then the Hamiltonian field $H_p$ extends to a smooth vector field on $\overline T^*\mathcal M$
which is tangent to $\partial\overline T^*\mathcal M$. For later use in this section, we recall the notation
$$
\Re\mathbf A:={\mathbf A+\mathbf A^*\over 2},\quad
\Im\mathbf A:={\mathbf A-\mathbf A^*\over 2i},
$$
where $\mathbf A$ is an operator $C^\infty(\mathcal M;\mathcal E)\to\mathcal D'(\mathcal M;\mathcal E)$ and we fix
a volume form on $\mathcal M$ and an inner product on $\mathcal E$ to define the adjoint operator $\mathbf A^*$.

First of all, we review the classical Duistermaat--H\"ormander \emph{propagation of singularities},
formulated using the following
\begin{defi}
  \label{d:con}
Assume that $p\in\Hom^1(T^*\mathcal M;\mathbb R)$.
Let $V,W\subset\overline T^*\mathcal M$ be open sets. We say that a point $(x,\xi)\in\overline T^*M$
is \emph{controlled by $V$ inside of $W$}, if there exists $T\geq 0$ such that
$e^{-TH_p}(x,\xi)\in V$ and $e^{-tH_p}(x,\xi)\in W$ for $t\in [0,T]$. Denote by
\begin{equation}
  \label{e:con}
\Con_p(V;W)\subset\overline T^*\mathcal M
\end{equation}
the set of all such points. Note that $\Con_p(V;W)$ is an open subset of $\overline T^*\mathcal M$.
\end{defi}

Propagation of singularities (see for instance~\cite[Proposition~2.5]{DyZw})
is then formulated as follows:
%%%%%%%%%%%%%%%%%%%%%%%%%%%%%%%%%%%%%%%%%%%%%%%%%%%%%%%%%%%%%%%%%%%%%%%%%%%%%%%%
\begin{lemm}
  \label{l:propagation}
Assume that $\mathbf P\in\Psi^1_h(\mathcal M;\mathcal E)$ is principally scalar
and $\sigma_h(\mathbf P)=p-iq$ where%
\footnote{Strictly speaking, this means that $\sigma_h(\mathbf P)=p-iq+\mathcal O(h)_{S^0}$.
In particular, the real part of $\sigma_h(\mathbf P)$ is independent of $h$.}
$p\in\Hom^1(T^*\mathcal M;\mathbb R)$
and $q$ is real-valued.
Let $A,B,B_1\in\Psi^0_h(\mathcal M)$
be such that
$$
q\geq 0\quad\text{near }\WF_h(B_1),\quad
\WFh(A)\subset\Con_p(\Ell_h(B);\Ell_h(B_1)).
$$
Then for each $s,N$ and $\mathbf u\in C^\infty(\mathcal M;\mathcal E)$,
we have
\begin{equation}
  \label{e:propagation}
\|A\mathbf u\|_{H^s_h}\leq C\|B\mathbf u\|_{H^s_h}+Ch^{-1}\|B_1\mathbf P \mathbf u\|_{H^s_h}+\mathcal O(h^\infty)\|\mathbf u\|_{H^{-N}_h}.
\end{equation}
\end{lemm}
%%%%%%%%%%%%%%%%%%%%%%%%%%%%%%%%%%%%%%%%%%%%%%%%%%%%%%%%%%%%%%%%%%%%%%%%%%%%%%%%
In this subsection, we give a more general propagation estimate (Lemma~\ref{l:ultimate})
under the weaker assumption that the trajectories of $e^{-tH_p}$ starting on $\WFh(A)$
either pass through $\Ell_h(B)$ or converge to some closed set $L$, while staying in $\Ell_h(B_1)$.
This follows a long tradition of study of operators with radial invariant sets,
see in particular Guillemin--Schaeffer~\cite{gui-sch}, Melrose~\cite{Me}, Herbst--Skibsted~\cite{he-sk},
and Hassell--Melrose--Vasy~\cite{HMV}.
For the estimate, we need
to additionally restrict the sign of the imaginary part of the subprincipal symbol of $\mathbf P$ on $L$, which
is achieved by the following
%%%%%%%%%%%%%%%%%%%%%%%%%%%%%%%%%%%%%%%%%%%%%%%%%%%%%%%%%%%%%%%%%%%%%%%%%%%%%%%%
\begin{defi}
  \label{d:gtrsimh}
Let $\mathbf P\in\Psi^1_h(\mathcal M;\mathcal E)$ and $L\subset \overline T^*\mathcal M$ be a closed set.
Fix a volume form on $\mathcal M$ and an inner product on the fibers of $\mathcal E$;
this defines an inner product on $L^2(\mathcal M;\mathcal E)$.
Fix also $s\in\mathbb R$.
We say that
\begin{equation}
  \label{e:gtrsimh0}
\Im \mathbf P\lesssim -h\quad\text{on }H^s_h\quad\text{microlocally near }L
\end{equation}
if there exist operators
$$
\begin{gathered}
\mathbf Y_1\in\Psi^s_h(\mathcal M;\mathcal E),\quad \mathbf Y_2\in\Psi^{-s}_h(\mathcal M;\mathcal E),\quad \mathbf Z\in \Psi^0_h(\mathcal M;\mathcal E);\\
\mathbf Y_1\mathbf Y_2=1+\mathcal O(h^\infty)\quad\text{near }L,\quad
L\subset\Ell_h(\mathbf Z),
\end{gathered}
$$
such that for each $N$, $h$ small enough, and each $\mathbf u\in H^{1/2}_h(\mathcal M;\mathcal E)$,
\begin{equation}
  \label{e:gtrsimh}
\Im \langle \mathbf Y_1\mathbf P \mathbf Y_2\mathbf u,\mathbf u\rangle_{L^2}\leq -h\|\mathbf Z\mathbf u\|^2_{L^2}+\mathcal O(h^\infty)\|\mathbf u\|_{H^{-N}_h}^2.
\end{equation}
\end{defi}
%%%%%%%%%%%%%%%%%%%%%%%%%%%%%%%%%%%%%%%%%%%%%%%%%%%%%%%%%%%%%%%%%%%%%%%%%%%%%%%%
\noindent\textbf{Remarks}. (i) The above definition does not actually depend on the choice of
the volume form on $M$ and the metric on the fibers of $\mathcal E$. Indeed, any other choice
yields the inner product $\langle \mathbf u,\mathbf v\rangle'=\langle \mathbf W\mathbf u,\mathbf W\mathbf v\rangle$
for some invertible $\mathbf W\in C^\infty(\mathcal M;\End(\mathcal E))$. Applying~\eqref{e:gtrsimh} for
the inner product $\langle\cdot,\cdot\rangle$ to $\mathbf W\mathbf u$,
we obtain~\eqref{e:gtrsimh} for $\langle\cdot,\cdot\rangle'$ with the operators
$\mathbf Y_1'=\mathbf W^{-1}\mathbf Y_1$,
$\mathbf Y_2'=\mathbf Y_2\mathbf W$,
and $\mathbf Z'=\mathbf W^{-1}\mathbf Z\mathbf W$.

(ii) If $L\cap \partial\overline T^*\mathcal M=\emptyset$, then Definition~\ref{d:gtrsimh} also
does not depend on the value of~$s$. Indeed, for each $B\in\Psi^{\comp}_h(\mathcal M)$ such that
$B=1+\mathcal O(h^\infty)$ microlocally near $L$, we can apply~\eqref{e:gtrsimh} to $B\mathbf u$ to get the same inequality
with the operators $\mathbf Y_1'=B^*\mathbf Y_1$, $\mathbf Y_2'=\mathbf Y_2 B$,
$\mathbf Z'=\mathbf Z B$ which lie in $\Psi^{\comp}_h$, and thus in $\Psi^s_h(\mathcal M)$ for all $s$.

(iii) The presence of the operators $\mathbf Y_1,\mathbf Y_2$ (which is inevitable in the case $s\neq 0$
as there is no canonical elliptic operator in $\Psi^s_h$, unlike the identity operator for $s=0$) makes the definition~\eqref{e:gtrsimh} subtle.
For instance, the sum of two operators satisfying~\eqref{e:gtrsimh0} does not necessarily satisfy the same
condition. Moreover, the real part $\Re \mathbf P$ enters the definition in a nontrivial way. In fact,
the statement~\eqref{e:gtrsimh0} does not change if $\mathbf P$ is replaced by
$$
\mathbf P':=Y\mathbf P Y^{-1}=\mathbf P+[Y,\mathbf P]Y^{-1},\quad
\sigma_h\big(h^{-1}(\mathbf P'-\mathbf P)\big)=i H_{\sigma_h(P)} f,
$$
for any $Y\in\Psi^0_h(\mathcal M)$ with $\sigma_h(Y)=e^f$.
In particular, one can add functions of the form $H_{\Re\sigma_h(P)} f$
to the imaginary part of the subprincipal symbol of $\mathbf P$, which means
that~\eqref{e:gtrsimh} is a really a statement about the ergodic averages of this symbol
along the flow $\exp(tH_{\Re\sigma_h(P)})$. These
subtleties do not play a role in our analysis because we will always enforce~\eqref{e:gtrsimh}
by either adding a large term or taking sufficiently large $|s|$~-- see the following two lemmas.
\smallskip

We will use the following formulation of the sharp G\r arding inequality:
%%%%%%%%%%%%%%%%%%%%%%%%%%%%%%%%%%%%%%%%%%%%%%%%%%%%%%%%%%%%%%%%%%%%%%%%%%%%%%%%
\begin{lemm}
  \label{l:garding}
Assume that $\mathbf P\in\Psi^{2m+1}_h(\mathcal M;\mathcal E)$ is principally scalar,
$A\in\Psi^0_h(\mathcal M)$, and\/%
\footnote{Since $\sigma_h(\mathbf P)\in S^{2m+1}/hS^{2m}$, the following inequality needs to be satisfied
for some representative of this equivalence class.}
 $\Re\sigma_h(\mathbf P)\leq 0$
in a neighborhood $U\subset \overline T^*\mathcal M$ of $\WFh(A)$.
Then there exists a constant $C$ such that for each $N$ and
$\mathbf u\in H^{m+1/2}_h(\mathcal M;\mathcal E)$,
$$
\Re\langle \mathbf P A\mathbf u,A\mathbf u\rangle_{L^2}\leq
Ch\|A\mathbf u\|_{H^m_h}^2+\mathcal O(h^\infty)\|\mathbf u\|_{H^{-N}_h}^2.
$$
\end{lemm}
%%%%%%%%%%%%%%%%%%%%%%%%%%%%%%%%%%%%%%%%%%%%%%%%%%%%%%%%%%%%%%%%%%%%%%%%%%%%%%%%
\begin{proof}
Since $\mathbf P$ is principally scalar, we can write it as a sum of
a scalar operator in $\Psi^{2m+1}_h(\mathcal M)$ and an $h\Psi^{2m}_h(\mathcal M;\mathcal E)$
remainder. Therefore, we may assume that $\mathbf P$ is scalar, which reduces us to
the case when $\mathcal E$ is trivial.

Take $B\in\Psi^0_h(\mathcal M)$ such that $B=1+\mathcal O(h^\infty)$ near $\WFh(A)$,
$\sigma_h(B)\geq 0$ everywhere,
and $\WFh(B)\subset U$. Then $\sigma_h(\Re(\mathbf P B))\leq 0$ everywhere.
By the standard sharp G\r arding inequality~\cite[Theorem~9.11]{e-z},
there exists a constant $C$ such that for $\mathbf u\in H^{m+1/2}_h(\mathcal M)$
$$
\Re\langle \mathbf P B A\mathbf u,A\mathbf u\rangle_{L^2}\leq Ch\|A\mathbf u\|_{H^m_h}^2.
$$
Here we use a partition of unity and coordinate charts to reduce to the case
$\mathcal M=\mathbb R^n$. It remains to note that $A=BA+\mathcal O(h^\infty)$ and thus
$$
\Re\langle\mathbf PA\mathbf u,A\mathbf u\rangle_{L^2}
=\Re\langle\mathbf PBA\mathbf u,A\mathbf u\rangle_{L^2}+
\mathcal O(h^\infty)\|\mathbf u\|_{H^{-N}_h}^2
$$
for each $N$.
\end{proof}
%%%%%%%%%%%%%%%%%%%%%%%%%%%%%%%%%%%%%%%%%%%%%%%%%%%%%%%%%%%%%%%%%%%%%%%%%%%%%%%%

We now provide several situations in which~\eqref{e:gtrsimh0} is satisfied:
%%%%%%%%%%%%%%%%%%%%%%%%%%%%%%%%%%%%%%%%%%%%%%%%%%%%%%%%%%%%%%%%%%%%%%%%%%%%%%%%
\begin{lemm}
  \label{l:gs1}
Let $L\subset T^*\mathcal M$ be a closed subset, $\mathbf P\in\Psi^1_h(\mathcal M;\mathcal E)$
be principally scalar, $Q\in\Psi^0_h(\mathcal M)$, and
\begin{equation}
  \label{e:imim}
\Im\sigma_h(\mathbf P)\leq 0\quad\text{near }L,\quad
\Re\sigma_h(Q)>0\quad\text{on }L.
\end{equation}
Then $\Im(\mathbf P-iQ)\lesssim -h$ on $H^s_h$ near $L$, for all $s$.
\end{lemm}
%%%%%%%%%%%%%%%%%%%%%%%%%%%%%%%%%%%%%%%%%%%%%%%%%%%%%%%%%%%%%%%%%%%%%%%%%%%%%%%%
\begin{proof}
Take $Y_1\in\Psi^s_h(\mathcal M)$, $Y_2\in\Psi^{-s}_h(\mathcal M)$ such that
$Y_1Y_2=1+\mathcal O(h^\infty)$ near $L$. Take also $Z\in\Psi^0_h(\mathcal M)$
such that near $L$, $Z=1+\mathcal O(h^\infty)$ and near $\WFh(Z)$,
$$
Y_1Y_2=1+\mathcal O(h^\infty),\quad
\Im\sigma_h(\mathbf P)\leq 0,\quad
\Re\sigma_h(Q)\geq \varepsilon>0.
$$
Then $\Im\sigma_h(Y_1\mathbf PY_2)\leq 0$ and $\Re\sigma_h(Y_1 Q Y_2)\geq \varepsilon$ near $\WFh(Z)$; therefore,
by Lemma~\ref{l:garding}, there exists a constant $C$ such that
for each $N$ and each $\mathbf u\in H^{1/2}_h(\mathcal M;\mathcal E)$,
$$
\begin{aligned}
\Im\langle Y_1\mathbf P Y_2 Z\mathbf u,Z\mathbf u\rangle_{L^2}
&\leq Ch\|Z\mathbf u\|^2_{L^2}+
\mathcal O(h^\infty)\|\mathbf u\|_{H^{-N}_h}^2;\\
\Re\langle Y_1 QY_2 Z\mathbf u,Z\mathbf u\rangle_{L^2}
&\geq \varepsilon\|Z\mathbf u\|_{L^2}^2
-Ch\|Z\mathbf u\|_{H^{-1/2}_h}^2-\mathcal O(h^\infty)\|u\|_{H^{-N}_h}^2.
\end{aligned}
$$
This implies that for $h$ small enough,
$$
\Im\langle Y_1(\mathbf P-iQ)Y_2 Z\mathbf u,Z\mathbf u\rangle_{L^2}\leq (2Ch-\varepsilon)\|Z\mathbf u\|_{L^2}^2
+\mathcal O(h^\infty)\|u\|_{H^{-N}_h}.
$$
Therefore, \eqref{e:gtrsimh0} holds for small $h$
with $\mathbf Y_1:=Z^*Y_1$, $\mathbf Y_2:=Y_2 Z$, and $\mathbf Z:=Z$.
\end{proof}
%%%%%%%%%%%%%%%%%%%%%%%%%%%%%%%%%%%%%%%%%%%%%%%%%%%%%%%%%%%%%%%%%%%%%%%%%%%%%%%%

%%%%%%%%%%%%%%%%%%%%%%%%%%%%%%%%%%%%%%%%%%%%%%%%%%%%%%%%%%%%%%%%%%%%%%%%%%%%%%%%
\begin{lemm}
  \label{l:gs2}
Let $L,\mathbf P$ satisfy the assumptions of Lemma~\ref{l:gs1} and
additionally $p:=\Re\sigma_h(\mathbf P)\in\Hom^1(T^*\mathcal M;\mathbb R)$.
Assume next that $L\subset \partial\overline T^*\mathcal M$
and $L$ is invariant under~$e^{tH_p}$. Fix a metric $|\cdot|$ on the fibers of $T^*\mathcal M$. Then:

1. Assume that there exist $c,\gamma>0$ such that
\begin{equation}
  \label{e:rads1}
{|e^{tH_p}(x,\xi)|\over |\xi|}\geq ce^{\gamma|t|}\quad\text{for }(x,\xi)\in L,\
t\leq 0.
\end{equation}
(Note that the left-hand side of~\eqref{e:rads1}
extends to a smooth function on $\overline T^*\mathcal M$.)
Then there exists $s_0$ such that for all $s>s_0$,
$\Im\mathbf P\lesssim -h$ near $L$ on $H^s_h$.

2. Assume that there exist $c,\gamma>0$ such that
\begin{equation}
  \label{e:rads2}
{|e^{tH_p}(x,\xi)|\over |\xi|}\geq ce^{\gamma |t|}\quad\text{for }(x,\xi)\in L,\
t\geq 0.
\end{equation}
Then there exists $s_0$ such that for all $s<s_0$,
$\Im\mathbf P\lesssim -h$ near $L$ on $H^s_h$.
\end{lemm}
%%%%%%%%%%%%%%%%%%%%%%%%%%%%%%%%%%%%%%%%%%%%%%%%%%%%%%%%%%%%%%%%%%%%%%%%%%%%%%%%
\begin{proof}
1. We first find $f\in \Hom^1(T^*\mathcal M;\mathbb R)$ such that
$\langle\xi\rangle^{-1} f>0$ on $\overline T^*\mathcal M$ and 
$$
{H_p f\over f}<-{\gamma\over 2}<0\quad\text{on }L.
$$
For that, we fix $f_0\in\Hom^1(T^*\mathcal M;\mathbb R)$ such that $\langle\xi\rangle^{-1}f_0>0$
on $\overline T^*\mathcal M$. Then for $T$ large enough, \eqref{e:rads1} implies that
$$
{f_0\circ e^{-TH_p}\over f_0}>e^{\gamma T/2}\quad\text{on }L.
$$
Using that $\log(f_0\circ e^{-tH_p})-\log f_0\in\Hom^0(T^*\mathcal M;\mathbb R)$, 
we then define $f$ by
$$
\log f={1\over T}\int_0^T\log(f_0\circ e^{-tH_p})\,dt,\quad
{H_p f\over f}=-{1\over T}\log\Big({f_0\circ e^{-TH_p}\over f_0}\Big)<-{\gamma\over 2}\quad\text{on }L.
$$
Having constructed $f$, we take $Y_1\in\Psi^s_h(\mathcal M),Y_2\in\Psi^{-s}_h(\mathcal M),Z\in\Psi^0_h(\mathcal M)$ such that
$$
\begin{gathered}
Z=1+\mathcal O(h^\infty)\quad\text{near }L,\quad
Y_1Y_2=1+\mathcal O(h^\infty)\quad\text{near }\WFh(Z);\\
\Im\sigma_h(\mathbf P)\leq 0,\quad
\sigma_h(Y_1)=f^s,\quad\text{and}\quad
{H_p f\over f}< -{\gamma\over 2}
\quad\text{near }\WFh(Z).
\end{gathered}
$$
Then we have microlocally near $\WFh(Z)$,
$$
\begin{gathered}
Y_1\mathbf PY_2=\mathbf P+[Y_1,\mathbf P]Y_2,\quad
[Y_1,\mathbf P]Y_2\in h\Psi^0_h(\mathcal M;\mathcal E),\quad
\Im\sigma_h(h^{-1}[Y_1,\mathbf P]Y_2)=s\,{H_pf\over f}.
\end{gathered}
$$
Similarly to the proof of Lemma~\ref{l:gs1}, by applying the sharp G\r arding inequality twice we get
for some constant $C$ independent of $s>0$,
all $N$, and all $\mathbf u\in H^{1/2}_h(\mathcal M;\mathcal E)$
$$
\begin{gathered}
\Im\langle Y_1\mathbf P Y_2 Z\mathbf u,Z\mathbf u\rangle_{L^2}=\Im\langle \mathbf P Z\mathbf u,Z\mathbf
u\rangle_{L^2}+\Im\langle [Y_1,\mathbf P]Y_2 Z\mathbf u,Z\mathbf u\rangle_{L^2}
+\mathcal O(h^\infty)\|\mathbf u\|_{H^{-N}_h}^2\\
\leq \Big(C-{s\gamma\over 2}\Big)h\|Z\mathbf u\|_{L^2}^2+\mathcal O(h^\infty)\|u\|_{H^{-N}_h}^2.
\end{gathered}
$$
It remains to choose $s$ large enough so that ${s\gamma\over 2}-C\geq 1$; then~\eqref{e:gtrsimh0} holds
with $\mathbf Y_1:=Z^*Y_1$, $\mathbf Y_2:=Y_2 Z$, and $\mathbf Z:=Z$.

2. We argue similarly to part~1. First of all, we construct
$f\in \Hom^1(T^*\mathcal M;\mathbb R)$ such that
$\langle\xi\rangle^{-1} f>0$ on $\overline T^*\mathcal M$ and 
$$
{H_p f\over f}>{\gamma\over 2}>0\quad\text{on }L.
$$
This is done as in part~1, reversing the direction of the flow. We next argue as before, replacing
$C-{s\gamma\over 2}$ by $C+{s\gamma\over 2}$ and choosing $s<0$ large enough in absolute value.
\end{proof}
%%%%%%%%%%%%%%%%%%%%%%%%%%%%%%%%%%%%%%%%%%%%%%%%%%%%%%%%%%%%%%%%%%%%%%%%%%%%%%%%
We now formulate the main propagation estimate; see Figure~\ref{f:ultimate}.
%%%%%%%%%%%%%%%%%%%%%%%%%%%%%%%%%%%%%%%%%%%%%%%%%%%%%%%%%%%%%%%%%%%%%%%%%%%%%%%%
\begin{lemm}
  \label{l:ultimate}
Assume that $\mathbf P\in\Psi^1_h(\mathcal M;\mathcal E)$ is principally scalar with
$\sigma_h(\mathbf P)=p-iq$, where $p,q$ are real-valued and $p\in\Hom^1(T^*\mathcal M;\mathbb R)$.
Let $L\subset\overline T^*\mathcal M$ be compact and invariant under $e^{tH_p}$.
Assume that $A,B,B_1\in\Psi^0_h(\mathcal M)$ and $s\in\mathbb R$ are such that
$$
\begin{gathered}
\WFh(A)\subset\Ell_h(B_1),\quad
L\subset\Ell_h(A),\quad
L\cap\WFh(B)=\emptyset,\\
q\geq 0\quad\text{near }\WFh(B_1),\quad
\Im \mathbf P\lesssim -h\quad\text{on }H^s_h\quad\text{near }L.
\end{gathered}
$$
Consider the closed subset set of $\overline T^*\mathcal M$ (see~\eqref{e:con})
$$
\Omega:=\{\langle\xi\rangle^{-1}p=0\}\setminus\Con_p(\Ell_h(B);\Ell_h(B_1))
$$
and assume that uniformly in $(x,\xi)\in\Omega\cap \WFh(A)$,
\begin{equation}
  \label{e:ultcon}
e^{tH_p}(x,\xi)\to L\quad\text{as }t\to-\infty;\quad
e^{tH_p}(x,\xi)\in \Ell_h(B_1)\quad\text{for }t\leq 0.
\end{equation}
Then for each $N$, for $h$ small enough, and for each $\mathbf u\in C^\infty(\mathcal M;\mathcal E)$,
\begin{equation}
  \label{e:ultimate}
\|A\mathbf u\|_{H^s_h}\leq C\|B\mathbf u\|_{H^s_h}+Ch^{-1}\|B_1\mathbf P\mathbf u\|_{H^s_h}
+\mathcal O(h^\infty)\|\mathbf u\|_{H^{-N}_h}.
\end{equation}
\end{lemm}
%%%%%%%%%%%%%%%%%%%%%%%%%%%%%%%%%%%%%%%%%%%%%%%%%%%%%%%%%%%%%%%%%%%%%%%%%%%%%%%%
\begin{figure}
\includegraphics{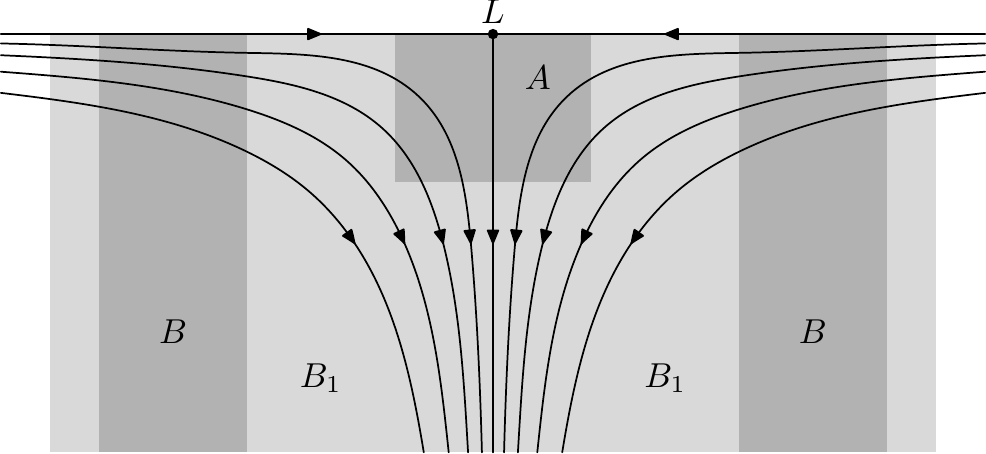}
\caption{An illustration of Lemma~\ref{l:ultimate} in the case
where $L$ lies inside the fiber infinity, the latter depicted by the horizontal line
at the top. The lighter shaded region
is the wavefront set of $B_1$, while the darker shaded regions are the wavefront
sets of $A$ and $B$. Several trajectories of the flow are displayed;
$\Omega$ is the vertical trajectory converging to $L$.}
\label{f:ultimate}
\end{figure}
%%%%%%%%%%%%%%%%%%%%%%%%%%%%%%%%%%%%%%%%%%%%%%%%%%%%%%%%%%%%%%%%%%%%%%%%%%%%%%%%
\noindent\textbf{Remarks}.
(i) The condition $\mathbf u\in C^\infty(\mathcal M)$ can be relaxed as follows:
let $m<s$ and $\Im\mathbf P\lesssim -h$ on $H^{s'}_h$ near $L$ for all $s'\in [m,s]$,
and the symbol of the corresponding operators $\mathbf Z$ is invertible
on $L$ uniformly in $s'$.
Then the conditions $A\mathbf u\in H^m_h$, $B\mathbf u\in H^s_h$,
$B_1\mathbf Pu\in H^s_h$ imply
that $A\mathbf u\in H^s_h$, and~\eqref{e:ultimate} holds. The proof works by 
improving the Sobolev regularity of $\mathbf u$ in small steps $\delta>0$ (depending
on the operators in~\eqref{e:gtrsimh}) by an approximation argument similar
to the one in the proofs of~\cite[Propositions~2.3--2.4]{vasy}.
For our purposes, it suffices to show~\eqref{e:ultimate}
for $\mathbf u\in C^\infty$, so we avoid this approximation argument.

(ii) Lemma~\ref{l:ultimate} implies several other semiclassical estimates:
\begin{itemize}
\item propagation of singularities (Lemma~\ref{l:propagation}), by taking $L=\emptyset$;
\item radial points estimate (see~\cite[Proposition~9]{Me} and~\cite[Proposition~2.6]{DyZw}), by taking $L$ to be a radial source,
$B=0$, $\Omega=\{\langle\xi\rangle^{-1}p=0\}$, and using part~1 of Lemma~\ref{l:gs2};
\item dual radial points estimate (see~\cite[Proposition~10]{Me} and~\cite[Proposition~2.7]{DyZw}), by taking $L$ to be a radial sink,
$B$ microlocalized inside a punctured neighborhood of $L$, 
$\Omega\cap\WFh(A)=L\cap\{\langle\xi\rangle^{-1}p=0\}$, and using part~2 of Lemma~\ref{l:gs2}.
\end{itemize}
The first implication is circular, since the proof uses
propagation of singularities.

The proof of Lemma~\ref{l:ultimate} relies on the construction
of a special escape function:
%%%%%%%%%%%%%%%%%%%%%%%%%%%%%%%%%%%%%%%%%%%%%%%%%%%%%%%%%%%%%%%%%%%%%%%%%%%%%%%%
\begin{lemm}
  \label{l:ultimatesc}
Under the assumptions of Lemma~\ref{l:ultimate}, let $U\subset\overline T^*\mathcal M$
be an open neighborhood of $L$. Then there exists a function $\chi\in C^\infty(\overline T^*\mathcal M;[0,1])$ such that:
\begin{enumerate}
\item $\supp\chi\subset U$;
\item $\chi=1$ near $L$;
\item $H_p\chi\leq 0$ in some neighborhood of $\Omega$.
\end{enumerate}
\end{lemm}
%%%%%%%%%%%%%%%%%%%%%%%%%%%%%%%%%%%%%%%%%%%%%%%%%%%%%%%%%%%%%%%%%%%%%%%%%%%%%%%%
\begin{figure}
\includegraphics{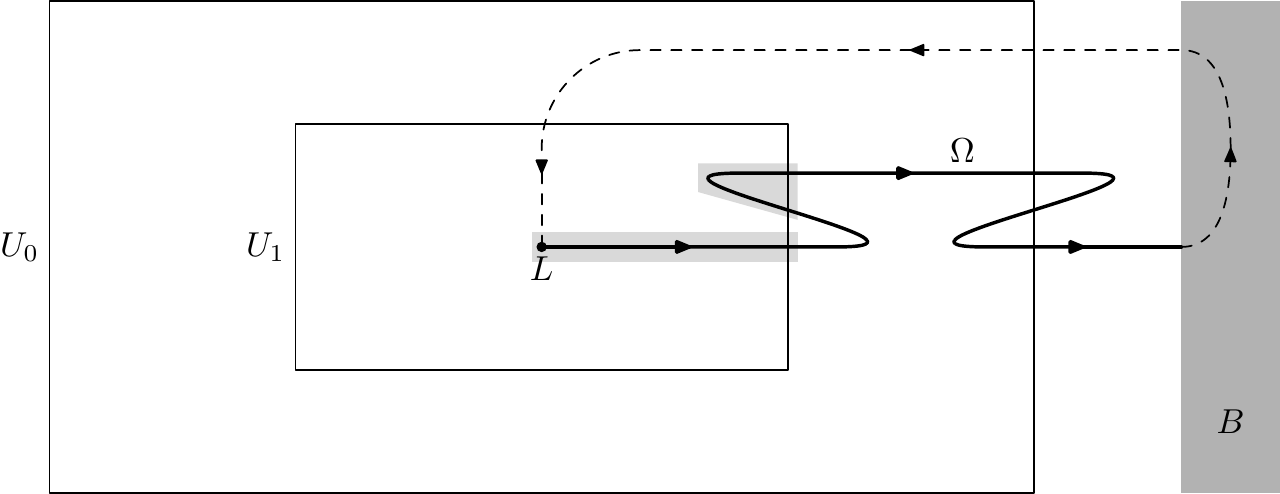}
\caption{An illustration of the proof of Lemma~\ref{l:ultimatesc}. The darker
shaded region is $\WFh(B)$ and the lighter shaded region is
$\supp\chi_1$. A possible trajectory of $H_p$ is shown; $\Omega$ is the undashed
part of this trajectory.}
\label{f:escape}
\end{figure}
%%%%%%%%%%%%%%%%%%%%%%%%%%%%%%%%%%%%%%%%%%%%%%%%%%%%%%%%%%%%%%%%%%%%%%%%%%%%%%%%
\begin{proof}
We take open neighborhoods (see Figure~\ref{f:escape})
$$
U_1\ \subset\ U_0\ \subset\ \Ell_h(A)\setminus\WFh(B)
$$
of $L$ such that
\begin{equation}
  \label{e:nbhds}
e^{-tH_p}(\Omega\cap\overline{U_0})\ \subset\ U\quad\text{and}\quad
e^{-tH_p}(\Omega\cap\overline{U_1})\ \subset\ U_0\quad\text{for all }t\geq 0.
\end{equation}
The first equation in~\eqref{e:nbhds} follows from~\eqref{e:ultcon} for
$t$ large enough independently of $U_0$;
for bounded $t$, it suffices to use the fact that $L$ is invariant under the flow
and take $U_0$ small enough. The set $U_1$ is constructed in the same way.

By~\eqref{e:ultcon}, there exists $T>0$ such that
\begin{equation}
  \label{e:falibi0}
e^{-TH_p}(\Omega\cap\overline{U_0})\ \subset\ U_1.
\end{equation}
Take a function $\chi_1\in C_0^\infty(U_0;[0,1])$ such that $\chi_1=1$ near $(\Omega\cap\overline{U_1})\cup L$
and
\begin{equation}
  \label{e:falibi1}
e^{-tH_p}(\supp\chi_1)\ \subset\ U_0\quad\text{for all }t\in [0,T].
\end{equation}
The existence of such function follows from the second equation in~\eqref{e:nbhds}
and the invariance of $L$ under the flow.

We have
\begin{equation}
  \label{e:libifa}
e^{tH_p}(\Omega)\cap \supp\chi_1\ \subset\ \Omega\quad\text{for all }t\in [0,T].
\end{equation}
Indeed, assume that $(x,\xi)\in\supp\chi_1$
and $e^{-tH_p}(x,\xi)\in\Omega$, but $(x,\xi)\notin\Omega$. Since
the set $\{\langle\xi\rangle^{-1}p=0\}$
is invariant under the flow and contains $e^{-tH_p}(x,\xi)$, we have
$(x,\xi)\in\Con_p(\Ell_h(B);\Ell_h(B_1))$. By Definition~\ref{d:con},
there exists $T'\geq 0$ such that $e^{-T'H_p}(x,\xi)\in\Ell_h(B)$
and $e^{-t'H_p}(x,\xi)\in\Ell_h(B_1)$ for all $t'\in [0,T']$.
By~\eqref{e:falibi1}, we have
$e^{-sH_p}(x,\xi)\notin\Ell_h(B)$ for $s\in [0,T]$, which
implies that $T'>T\geq t$. Then $e^{-tH_p}(x,\xi)\in\Con_p(\Ell_h(B);\Ell_h(B_1))$,
which contradicts the fact that $e^{-tH_p}(x,\xi)\in\Omega$.

Combining~\eqref{e:falibi0} and~\eqref{e:libifa}, we get
\begin{equation}
  \label{e:libifb}
\Omega\cap e^{-TH_p}(\supp\chi_1)\ \subset\ \Omega\cap e^{-TH_p}(\Omega\cap\supp\chi_1)\ \subset\ \Omega\cap U_1.
\end{equation}
Combining~\eqref{e:libifa} with the first part of~\eqref{e:nbhds}, we get
\begin{equation}
  \label{e:libifc}
\Omega\cap e^{-tH_p}(\supp\chi_1)\ \subset\ e^{-tH_p}(\Omega\cap\supp\chi_1)\ \subset U\quad\text{for all }t\in [0,T].
\end{equation}
It follows from~\eqref{e:libifb} that for $(x,\xi)$ near $\Omega\cap\supp(\chi_1\circ e^{TH_p})$,
we have $\chi_1(x,\xi)=1$.
Since $0\leq \chi_1\leq 1$, we have for $(x,\xi)$ in some neighborhood of $\Omega$,
\begin{equation}
  \label{e:falibi}
\chi_1(x,\xi)\geq \chi_1(e^{TH_p}(x,\xi)).
\end{equation}
Put
$$
\widetilde\chi={1\over T}\int_0^T \chi_1\circ e^{tH_p}\,dt,\quad
H_p\widetilde\chi={\chi_1\circ e^{TH_p}-\chi_1\over T},
$$
then $H_p\widetilde\chi\leq 0$ in some neighborhood of $\Omega$. By~\eqref{e:libifc}, $\Omega\cap\supp\widetilde\chi\subset U$. Since
$\chi_1=1$ near~$L$, we also have $\widetilde\chi=1$ near $L$. It remains
to put $\chi:=\chi_2\widetilde\chi$, where $\chi_2\in C_0^\infty(U;[0,1])$
satisfies $\chi_2=1$ near $(\Omega\cap\supp\widetilde\chi)\cup L$.
\end{proof}
%%%%%%%%%%%%%%%%%%%%%%%%%%%%%%%%%%%%%%%%%%%%%%%%%%%%%%%%%%%%%%%%%%%%%%%%%%%%%%%%
We now give
%%%%%%%%%%%%%%%%%%%%%%%%%%%%%%%%%%%%%%%%%%%%%%%%%%%%%%%%%%%%%%%%%%%%%%%%%%%%%%%%
\begin{proof}[Proof of Lemma~\ref{l:ultimate}]
We start with the estimate
\begin{equation}
  \label{e:ullie1}
\|\mathbf A_1\mathbf u\|_{H^s_h}\leq C\|B\mathbf u\|_{H^s_h}
+Ch^{-1}\|B_1\mathbf P\mathbf u\|_{H^s_h}
+\mathcal O(h^\infty)\|\mathbf u\|_{H^{-N}_h}
\end{equation}
valid for all $\mathbf A_1\in\Psi^0_h(\mathcal M;\mathcal E)$ such that
$\WFh(\mathbf A_1)\subset\Ell_h(B_1)\setminus\Omega$.
Indeed, we have
$$
\WFh(\mathbf A_1)\ \subset\ (\Ell_h(B_1)\cap\Ell_h(\mathbf P))\cup \Con_p(\Ell_h(B);\Ell_h(B_1)),
$$
therefore by a partition of unity we may reduce to the situation when $\WFh(\mathbf A_1)$
is contained either inside $\Ell_h(B_1)\setminus\{\langle\xi\rangle^{-1}p=0\}$
or inside $\Con_p(\Ell_h(B);\Ell_h(B_1))$. The first case is handled by the
elliptic estimate~\cite[Proposition~2.4]{DyZw} and the second one,
by propagation of singularities (Lemma~\ref{l:propagation}).

Similarly we have the estimate
\begin{equation}
  \label{e:ullie2}
\|A\mathbf u\|_{H^s_h}\leq C\|B\mathbf u\|_{H^s_h}
+C\|\mathbf A_2\mathbf u\|_{H^s_h}+Ch^{-1}\|B_1\mathbf P\mathbf u\|_{H^s_h}
+\mathcal O(h^\infty)\|\mathbf u\|_{H^{-N}_h}
\end{equation}
valid for all $\mathbf A_2\in\Psi^0_h(\mathcal M;\mathcal E)$ such that
$L\subset\Ell_h(\mathbf A_2)$, where we use the following corollary of~\eqref{e:ultcon}:
$$
\WFh(A)\ \subset\ (\Ell_h(B_1)\cap\Ell_h(\mathbf P))\cup\Con_p(\Ell_h(B);\Ell_h(B_1))
\cup\Con_p(\Ell_h(\mathbf A_2);\Ell_h(B_1)).
$$
Next, using Definition~\ref{d:gtrsimh}, choose
$\mathbf Y_1\in\Psi^s_h(\mathcal M;\mathcal E)$,
$\mathbf Y_2\in\Psi^{-s}_h(\mathcal M;\mathcal E)$,
and $\mathbf Z\in\Psi^0_h(\mathcal M;\mathcal E)$ such that
$$
\mathbf Y_1\mathbf Y_2=1+\mathcal O(h^\infty)\quad\text{near }\overline U,\quad
\overline U\subset\Ell_h(\mathbf Z),
$$
for some neighborhood $U\subset\Ell_h(A)\cap\Ell_h(B_1)$ of $L$, and for each $\mathbf v\in C^\infty(\mathcal M;\mathcal E)$,
\begin{equation}
  \label{e:ultii}
\Im\langle \mathbf P'\mathbf v,\mathbf v\rangle_{L^2}\leq -h\|\mathbf Z\mathbf v\|_{L^2}^2+\mathcal O(h^\infty)\|\mathbf v\|^2_{H^{-N}_h},
\end{equation}
where $\mathbf P':=\mathbf Y_1\mathbf P\mathbf Y_2\in\Psi^1_h(\mathcal M;\mathcal E)$.
Note that $\sigma_h(\mathbf P')=\sigma_h(\mathbf P)$ on $U$.

We now claim that it suffices to show that there exist operators
$$
A_1,A_2,B_2\in\Psi^0_h(\mathcal M),\quad
\WFh(A_1)\subset U\setminus\Omega,\quad
L\subset \Ell_h(A_2),\quad
\WFh(B_2)\subset U,
$$
such that for each $\mathbf v\in C^\infty(\mathcal M;\mathcal E)$,
\begin{equation}
  \label{e:tiner1}
\|A_2 \mathbf v\|_{L^2}\leq C\|A_1\mathbf v\|_{L^2}+Ch^{-1}\|B_2\mathbf P'\mathbf v\|_{L^2}
+Ch^{1/2}\|B_2\mathbf v\|_{H^{-1/2}_h}+\mathcal O(h^\infty)\|\mathbf v\|_{H^{-N}_h}.
\end{equation}
Indeed, applying~\eqref{e:tiner1} to $\mathbf v:=\mathbf Y_1\mathbf u$ and assuming
that $\WFh(A_2)\subset U$, we get
\begin{equation}
  \label{e:tinerz}
\|\mathbf A_2\mathbf u\|_{H^s_h}
\leq C\|\mathbf A_1\mathbf u\|_{H^s_h}
+Ch^{-1}\|B_1\mathbf P\mathbf u\|_{H^s_h}
+Ch^{1/2}\|A \mathbf u\|_{H^{s-1/2}_h}
+\mathcal O(h^\infty)\|\mathbf u\|_{H^{-N}_h}.
\end{equation}
where $\mathbf A_j:=\mathbf Y_2 A_j\mathbf Y_1$ satisfy
$\WFh(\mathbf A_1)\subset U\setminus\Omega$, $L\subset\Ell_h(\mathbf A_2)$. Here we used the
fact that $\WFh(B_2)\subset\Ell_h(A)$ and the elliptic estimate
to bound $\|B_2\mathbf v\|_{H^{-1/2}_h}$ in terms of
$\|A\mathbf u\|_{H^{s-1/2}_h}$.
To obtain the required estimate~\eqref{e:ultimate}, it remains to combine this
with~\eqref{e:ullie1} and~\eqref{e:ullie2}, and take $h$ small enough to eliminate
the $Ch^{1/2}\|A\mathbf u\|_{H^{s-1/2}_h}$ remainder.

We now prove~\eqref{e:tiner1} using a positive commutator argument.
Let $\chi$ be the function constructed in Lemma~\ref{l:ultimatesc}.
Fix $F\in \Psi^0_h(\mathcal M)$ such that
$$
\sigma_h(F)=\chi,\quad
\WFh(F)\subset U.
$$
Then
$$
\sigma_h(ih^{-1}[\Re\mathbf P',F^*F])=2\chi H_p\chi\leq 0\quad\text{near }\Omega.
$$
Therefore
\begin{equation}
[\Re\mathbf P',F^*F]=-ih (\mathbf G_1+\mathbf G_2),
\end{equation}
where $\mathbf G_j\in \Psi^0_h(\mathcal M;\mathcal E)$ are self-adjoint and principally
scalar, and
$$
\WFh(\mathbf G_j)\subset U;\quad
\sigma_h(\mathbf G_1)\leq 0;\quad
\WFh(\mathbf G_2)\cap\Omega=\emptyset.
$$
For each $\mathbf v\in C^\infty(\mathcal M;\mathcal E)$, we have
\begin{equation}
  \label{e:pc-1}
\Im\langle \mathbf P'\mathbf v,F^*F\mathbf v\rangle_{L^2}
={i\over 2}\langle [\Re\mathbf P',F^*F]\mathbf v,\mathbf v\rangle_{L^2}
+\Re\langle F^*F(\Im\mathbf P')\mathbf v,\mathbf v\rangle_{L^2}.
\end{equation}
Take $A_1,B_2\in\Psi^0_h(\mathcal M)$ such that
$$
\begin{gathered}
\WFh(A_1)\ \subset\ U\setminus\Omega,\quad
A_1=1+\mathcal O(h^\infty)\quad\text{near }
\WFh(\mathbf G_2);\\
\WFh(B_2)\ \subset\ U,\quad
B_2=1+\mathcal O(h^\infty)\quad\text{near }
\WFh(F)\cup\WFh(\mathbf G_1)\cup\WFh(\mathbf G_2).
\end{gathered}
$$
By the sharp G\r arding inequality (Lemma~\ref{l:garding}) applied to $\mathbf G_1$
and the elliptic estimate applied to $\mathbf G_2$, the product
${i\over 2}\langle [\Re\mathbf P',F^*F]\mathbf v,\mathbf v\rangle_{L^2}$ is equal to
\begin{equation}
  \label{e:pc-2}
\begin{aligned}
&{h\over 2}\big(\langle \mathbf G_1B_2\mathbf v,B_2\mathbf v\rangle_{L^2}
+\langle \mathbf G_2A_1\mathbf v,A_1\mathbf v\rangle_{L^2}\big)+\mathcal O(h^\infty)\|\mathbf v\|_{H^{-N}_h}^2
\\&\leq
Ch^2\|B_2\mathbf v\|_{H^{-1/2}_h}^2+
Ch\|A_1\mathbf v\|_{L^2}^2
+\mathcal O(h^\infty)\|\mathbf v\|_{H^{-N}_h}^2.
\end{aligned}
\end{equation}
Next, $h^{-1}F^*[F,\Im\mathbf P']\in \Psi^0_h$
and its principal symbol is imaginary valued, therefore
$\Re(F^*[F,\Im\mathbf P'])\in h^2\Psi^{-1}_h$. It then follows from~\eqref{e:ultii} that 
\begin{equation}
  \label{e:pc-3}
\begin{gathered}
\Re\langle F^*F(\Im\mathbf P')\mathbf v,\mathbf v\rangle_{L^2}
=\Im\langle \mathbf P'F\mathbf v,F\mathbf v\rangle_{L^2}
+\langle\Re(F^*[F,\Im\mathbf P'])\mathbf v,\mathbf v\rangle_{L^2}\\
\leq -h\|\mathbf ZF\mathbf v\|_{L^2}^2
+Ch^2\|B_2\mathbf v\|_{H^{-1/2}_h}^2
+\mathcal O(h^\infty)\|\mathbf v\|_{H^{-N}_h}^2.
\end{gathered}
\end{equation}
Since $\mathbf Z$ is elliptic on $U\supset \WFh(F)$, we have
$$
\|F\mathbf v\|_{L^2}^2\leq C\|\mathbf ZF\mathbf v\|_{L^2}^2
+\mathcal O(h^\infty)\|\mathbf v\|_{H^{-N}_h}^2. 
$$
Combining this with~\eqref{e:pc-1}--\eqref{e:pc-3}, we get
$$
\begin{gathered}
C^{-1}h\|F\mathbf v\|_{L^2}^2-Ch\|A_1\mathbf v\|_{L^2}^2-Ch^2\|B_2\mathbf v\|_{H^{-1/2}_h}^2
\\\leq -\Im\langle \mathbf P'\mathbf v,F^*F\mathbf v\rangle+\mathcal O(h^\infty)\|\mathbf v\|_{H^{-N}_h}^2
\\\leq C\|B_2\mathbf P'\mathbf v\|_{L^2}\cdot \|F\mathbf v\|_{L^2}+\mathcal O(h^\infty)\|\mathbf v\|_{H^{-N}_h}^2,
\end{gathered}
$$
which implies~\eqref{e:tiner1} with $A_2:=F$, finishing the proof.
\end{proof}
%%%%%%%%%%%%%%%%%%%%%%%%%%%%%%%%%%%%%%%%%%%%%%%%%%%%%%%%%%%%%%%%%%%%%%%%%%%%%%%%

%%%%%%%%%%%%%%%%%%%%%%%%%%%%%%%%%%%%%%%%%%%%%%%%%%%%%%%%%%%%%%%%%%%%%%%%%%%%%%%%
%%%%%%%%%%%%%%%%%%%%%%%%%%%%%%%%%%%%%%%%%%%%%%%%%%%%%%%%%%%%%%%%%%%%%%%%%%%%%%%%
\section{Properties of the resolvent}
  \label{s:res}

In this section, we prove Theorems~\ref{t:mer} and~\ref{t:res-states}, and show microlocalization
statements for the resolvent that form the basis of the proof of Theorem~\ref{t:zeta} in the next section.
We follow in part the argument of~\cite{DyZw}, based on the strategy
of~\cite{FaSj}.

%%%%%%%%%%%%%%%%%%%%%%%%%%%%%%%%%%%%%%%%%%%%%%%%%%%%%%%%%%%%%%%%%%%%%%%%%%%%%%%%
\subsection{Auxiliary resolvent}
\label{s:absorbed}

In this section, we introduce an auxiliary resolvent depending on the semiclassical parameter $h>0$.
Recall the function $\rho$, the constant $\varepsilon$, and the vector field $X_1$ used in Lemma~\ref{l:extconv},
and let $p$ be defined in~\eqref{e:symbol}.

%%%%%%%%%%%%%%%%%%%%%%%%%%%%%%%%%%%%%%%%%%%%%%%%%%%%%%%%%%%%%%%%%%%%%%%%%%%%%%%%
\smallsection{Anisotropic spaces}
We first construct the \emph{anisotropic Sobolev spaces}
on which the auxiliary resolvent will be defined.
The order function of these spaces
is given by the following (see Figure~\ref{f:absorbers}(a))
%%%%%%%%%%%%%%%%%%%%%%%%%%%%%%%%%%%%%%%%%%%%%%%%%%%%%%%%%%%%%%%%%%%%%%%%%%%%%%%%
\begin{lemm}
  \label{l:order}
There exists $m\in C^\infty(S^*\mathcal M;\mathbb R)$ such that, with
$H_p$ pulled back to $S^*\mathcal M$ by the projection
$\kappa:T^*\mathcal M\setminus 0\to S^*\mathcal M$,
and $E_\pm^*$ defined in Lemma~\ref{l:extended},
\begin{enumerate}
\item $m=1$ in a neighborhood of $\kappa(E_-^*)\supset\kappa(E_s^*)$;
\item $m=-1$ in a neighborhood of $\kappa(E_+^*)\supset\kappa(E_u^*)$;
\item $H_pm\leq 0$ in a neighborhood of $\{p=0\}$;
\item $\supp m\subset \{\rho>-2\varepsilon\}$ and
$\supp m\cap \{\rho=-\varepsilon\}\cap \{X_1\rho=0\}=\emptyset$.
\end{enumerate}  
\end{lemm}
%%%%%%%%%%%%%%%%%%%%%%%%%%%%%%%%%%%%%%%%%%%%%%%%%%%%%%%%%%%%%%%%%%%%%%%%%%%%%%%%
\begin{proof}
Let $m_\pm$ be the functions constructed in Lemma~\ref{l:functions}, then
\begin{enumerate}
\item $m_-=1$ and $m_+=0$ in a neighborhood of $\kappa(E_-^*)$;
\item $m_+=1$ and $m_-=0$ in a neighborhood of $\kappa(E_+^*)$;
\item $\pm H_p m_\pm\geq 0$ on $V\setminus \pi^{-1}(V_\pm)$, where $V$ is a neighborhood of $\{p=0\}$ and
$V_\pm:=\pi(\supp m_\pm)\setminus\mathcal U\subset \Sigma_\pm\setminus\mathcal U$ are compact. Here $\pi:S^*\mathcal M\to\mathcal M$
is the projection map and $\Sigma_\pm$ are defined in~\eqref{e:Sigma-def}.
\end{enumerate}
Next, take the functions $\chi_\pm$ constructed in Lemma~\ref{l:marron3}
(with the sets $V_\pm$ defined in (3) above).
We have $\pm H_p(\chi_\pm\circ\pi)\geq 0$ everywhere, $\chi_\pm=0$ near $K$, and
$\pm H_p(\chi_\pm\circ\pi)>\delta>0$ on $\pi^{-1}(V_\pm)$. Then for a large enough constant
$R>0$, the function
$$
m:=m_--m_++R(\chi_-\circ\pi-\chi_+\circ\pi)
$$
satisfies conditions~(1)--(3). Condition~(4) follows immediately
from the fact that $\supp m_\pm\subset\pi^{-1}(\Sigma_\pm)\subset \{\rho>-\varepsilon\}$
and $\supp\chi_\pm\subset\{\rho>-2\varepsilon\}\cap \{\pm X_1\rho<0\}$.
\end{proof}
%%%%%%%%%%%%%%%%%%%%%%%%%%%%%%%%%%%%%%%%%%%%%%%%%%%%%%%%%%%%%%%%%%%%%%%%%%%%%%%%
We now consider $m$ as a homogeneous function of degree 0 on $T^*\mathcal M$ and
define the weight $\widetilde m\in C^\infty(T^*\mathcal M)$ by
\begin{equation}
  \label{e:tilde-m}
\widetilde m(x,\xi)=(1-\chi_m(x,\xi))m(x,\xi)\log|\xi|,
\end{equation}
where $\chi_m\in C_0^\infty(T^*\mathcal M;[0,1])$ is equal to 1 near the zero section
and supported in $\{|\xi|<1\}$.
%%%%%%%%%%%%%%%%%%%%%%%%%%%%%%%%%%%%%%%%%%%%%%%%%%%%%%%%%%%%%%%%%%%%%%%%%%%%%%%%
\begin{figure}
\includegraphics{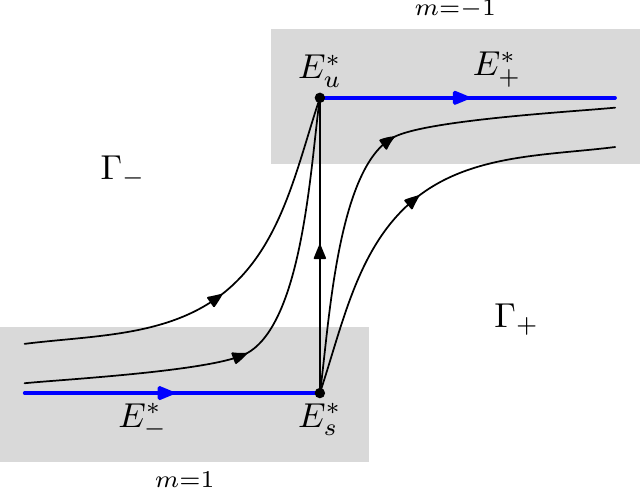}
\qquad
\includegraphics{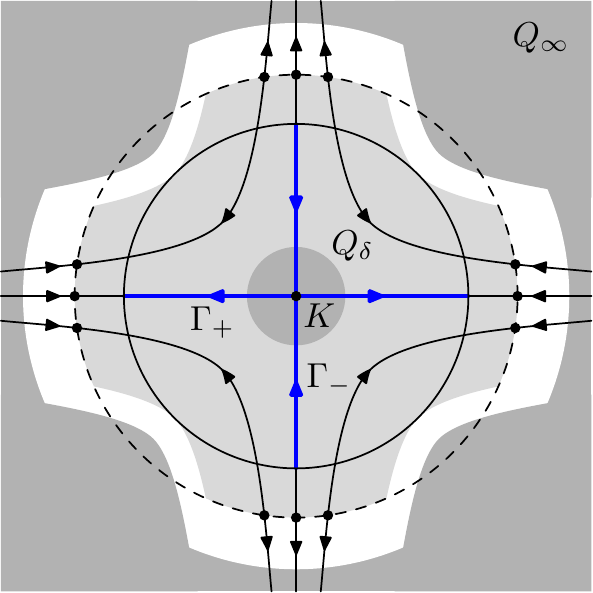}
\hbox to\hsize{\hss (a)\hss\hss (b)\hss}
\caption{(a) The flow $e^{tH_p}$ on $S^*\mathcal M\cap \{p=0\}$. The left half of the figure represents the sphere bundle over
$\Gamma_-$, the right half represents $\Gamma_+$, and the vertical midline represents $K$.
(To obtain this from Figure~\ref{f:funny3d}, glue the rear halves of the two vertical planes.)
The thick blue lines
are $E_\pm^*$ and the shaded boxes represent the regions where $m=\pm 1$.
(b) The flow $\varphi^t$ near $\mathcal U$ (pictured by the solid circle). The dashed circle
is $\{\rho=-\varepsilon\}$, consisting of fixed points; $q_1$ is supported near this circle. The lighter shaded region denotes the set $\Sigma$
from~\eqref{e:Sigma-def}, while
the darker shaded regions denote the supports of $Q_\infty$ and $Q_\delta$ (the latter microlocalized
near zero frequency).
}
\label{f:absorbers}
\end{figure}
%%%%%%%%%%%%%%%%%%%%%%%%%%%%%%%%%%%%%%%%%%%%%%%%%%%%%%%%%%%%%%%%%%%%%%%%%%%%%%%%
Take an operator
\begin{equation}
  \label{e:Goat}
G=G(h)\in\bigcap_{\delta>0}\Psi^\delta_h(\mathcal M),\quad
\sigma_h(G)=\widetilde m,\quad
\WFh(G)\cap \{\xi=0\}=\emptyset.
\end{equation}
We moreover require that
\begin{equation}
  \label{e:Chevre}
\WFh(G)\ \subset\ \{\rho>-2\varepsilon\},\quad
\WFh(G)\cap\{\rho=-\varepsilon\}\cap \{X_1\rho=0\}=\emptyset.
\end{equation}
For each $r\in\mathbb R$, we define the \emph{anisotropic Sobolev space} $\mathcal H^r_h$ as follows:
\begin{equation}
  \label{e:anisop}
\mathcal H^r_h:=\exp(-rG(h))(L^2(\mathcal M;\mathcal E)),\quad
\|\mathbf u\|_{\mathcal H^r_h}:=\|\exp(rG(h)) \mathbf u\|_{L^2(\mathcal M;\mathcal E)}.
\end{equation}
As explained for instance in~\cite[Sections~3.1 and~3.3]{DyZw}, we have
\begin{equation}
  \label{e:anisotc}
H^{C_m|r|}_h(\mathcal M;\mathcal E)\subset\mathcal H^r_h\subset H^{-C_m|r|}_h(\mathcal M;\mathcal E)
\end{equation}
where $H^r_h$ stands for the standard semiclassical Sobolev space~\cite[Section~14.2.4]{e-z}
and $C_m=\sup_{S^*\mathcal M}|m|$; we can take $C_m=1$ for distributions supported inside $\mathcal U$.
Moreover, the norms of $\mathcal H^r_h$ for different $h$ are all equivalent with constants depending on $h$.

Since $m=1$ near $\kappa(E_-^*)$, we have $\widetilde m=\log|\xi|$ near $\kappa(E_-^*)\subset\partial\overline T^*\mathcal M$, where
\begin{equation}
  \label{e:kappa}
\kappa:T^*\mathcal M\setminus 0\ \to\ S^*\mathcal M=\partial\overline T^*\mathcal M
\end{equation}
is the projection map. It follows from~\cite[Theorems~8.6 and~8.10]{e-z} that
$\mathcal H^r_h$ is microlocally equivalent to the standard semiclassical Sobolev space
$H^r_h$ near $\kappa(E_-^*)$ in the following sense:
for each $A\in\Psi^0_h(\mathcal M)$ such that $\WFh(A)$ is contained in a small neighborhood
of $\kappa(E_-^*)$ and all $\mathbf u\in C^\infty(\mathcal M;\mathcal E)$, we have
\begin{equation}
  \label{e:aniseq-1}
\|A\mathbf u\|_{\mathcal H^r_h}\leq C\|\mathbf u\|_{H^r_h},\quad
\|A\mathbf u\|_{H^r_h}\leq C\|\mathbf u\|_{\mathcal H^r_h}.
\end{equation}
Since $m=-1$ near $\kappa(E_+^*)$, we similarly have for each
$A\in\Psi^0_h(\mathcal M)$ with $\WFh(A)$ contained in a small neighborhood
of $\kappa(E_+^*)$,
\begin{equation}
  \label{e:aniseq-2}
\|A\mathbf u\|_{\mathcal H^r_h}\leq C\|\mathbf u\|_{H^{-r}_h},\quad
\|A\mathbf u\|_{H^{-r}_h}\leq C\|\mathbf u\|_{\mathcal H^r_h}.
\end{equation}

%%%%%%%%%%%%%%%%%%%%%%%%%%%%%%%%%%%%%%%%%%%%%%%%%%%%%%%%%%%%%%%%%%%%%%%%%%%%%%%%
\smallsection{Complex absorbing operators}
Take small $\delta>0$
and choose
\begin{equation}
  \label{e:absorbers-def}
Q_\infty\in\Psi^1_h(\mathcal M),\quad
Q_\delta\in \Psi^{\comp}_h(\mathcal M),\quad
q_1\in C^\infty(\mathcal M)
\end{equation}
such that
$\sigma_h(Q_\infty),\sigma_h(Q_\delta),q_1\geq 0$ everywhere and
\begin{enumerate}
\item $\indic_{\Sigma'}Q_\infty=Q_\infty\indic_{\Sigma'}=0$, where
$\Sigma'$ is a neighborhood of $\overline{\Sigma}$ and $\Sigma$ is defined
in~\eqref{e:Sigma-def};
\item $\{\rho\leq -2\varepsilon\}\cup\big(\{\rho=-\varepsilon\}\cap\{X_1\rho=0\}\big)\ \subset\ \Ell_h(Q_\infty)$;
\item $Q_\delta=\chi_\delta Q_\delta\chi_\delta$ for some $\chi_\delta\in C^\infty(\mathcal M)$ supported in a $\delta$-neighborhood of $K$,
$\WFh(Q_\delta)\ \subset\ \{|\xi|<\delta\}$, and $\{x\in K,\ \xi=0\}\ \subset\ \Ell_h(Q_\delta)$;
\item $\supp q_1\cap \overline{\mathcal U}=\emptyset$ and $q_1>0$ on $\{\rho=-\varepsilon\}$;
\item $\WFh(G)\cap (\WFh(Q_\delta)\cup\WFh(Q_\infty))=\emptyset$.
\end{enumerate}
The existence of $Q_\infty$ is guaranteed by Lemma~\ref{l:marron} and the fact that $\Sigma\subset \{\rho>-\varepsilon\}$. Condition~(5) can be satisfied by~\eqref{e:Goat} and~\eqref{e:Chevre}.
See Figure~\ref{f:absorbers}(b).

The use of the absorbing operator $Q_\delta$ goes back to~\cite{FaSj};
we will follow closely the later argument of~\cite{DyZw}. By contrast,
the operator $Q_\infty$ is something specific to open systems; such complex absorbing operators
have been previously used in scattering theory, see for instance~\cite{stefanov,NoZw,vasy}.
The absorbing potential $q_1$ guarantees invertibility
on $\{\rho=-\varepsilon\}$; making $Q_\infty$ elliptic there would destroy
the propagation of support property, meaning that Lemma~\ref{l:samething} below
would no longer be true.

%%%%%%%%%%%%%%%%%%%%%%%%%%%%%%%%%%%%%%%%%%%%%%%%%%%%%%%%%%%%%%%%%%%%%%%%%%%%%%%%
\smallsection{Existence of the auxiliary resolvent}
Introduce the modified operator
\begin{equation}
  \label{e:P-delta}
\mathbf P_\delta=\mathbf P_\delta(h)={h\over i}\mathbf X-i(Q_\infty+q_1+Q_\delta)\in\Psi^1_h(\mathcal M;\mathcal E)
\end{equation}
which acts $\mathcal D^r_h\to\mathcal H^r_h$, where
\begin{equation}
  \label{e:anisopd}
\mathcal D^r_h:=\{\mathbf u\in \mathcal H^r_h\mid \mathbf P_\delta\mathbf u\in\mathcal H^r_h\},\quad
\|\mathbf u\|_{\mathcal D^r_h}:=\|\mathbf u\|_{\mathcal H^r_h}+\|\mathbf P_\delta \mathbf u\|_{\mathcal H^r_h}.
\end{equation}
Note that, with $p$ defined in~\eqref{e:symbol},
$$
\sigma_h(\mathbf P_\delta)=p-i(\sigma_h(Q_\infty)+\sigma_h(Q_\delta)+q_1).
$$
The main result of this subsection is the following
%%%%%%%%%%%%%%%%%%%%%%%%%%%%%%%%%%%%%%%%%%%%%%%%%%%%%%%%%%%%%%%%%%%%%%%%%%%%%%%%
\begin{lemm}
  \label{l:keymic}
Take $C_1,C_2>0$. Then there exists $r_0=r_0(C_1)\geq 0$ such that
for all $r\geq r_0$ and $0<h<h_0(C_1,C_2,r)$,
the inverse
\begin{equation}
  \label{e:keymic}
\mathbf R_\delta(z)=(\mathbf P_\delta-z)^{-1}:\mathcal H^r_h\to\mathcal D^r_h,\quad
z\in [-C_2h,C_2h]+i[-C_1h,1]
\end{equation}
exists and satisfies the bound
\begin{equation}
  \label{e:keymicb}
\|\mathbf R_\delta(z)\|_{\mathcal H^r_h\to\mathcal H^r_h}\leq Ch^{-1}.
\end{equation}
Furthermore, the $h$-wavefront set of $\mathbf R_\delta(z)$ satisfies
\begin{equation}
  \label{e:WF-absorbed}
\WFh'(\mathbf R_\delta(z))\cap T^*(\mathcal U\times \mathcal U)\ \subset\ \Delta(T^*\mathcal U)\cup\Upsilon_+,
\end{equation}
where $\Delta(T^*\mathcal U)$ is the diagonal of $T^*\mathcal U$ and $\Upsilon_+$ is the positive flow-out
of $e^{tH_p}$ on $\{p=0\}$ inside $\pi^{-1}(\mathcal U)$ (here $\pi:\overline T^*\mathcal M\to \mathcal M$
is the projection map)
$$
\Upsilon_+=\{(e^{tH_p}(y,\eta),y,\eta)\mid t\geq 0,\ p(y,\eta)=0,\ y\in\mathcal U,\ \varphi^t(y)\in\mathcal U\}.
$$
\end{lemm}
%%%%%%%%%%%%%%%%%%%%%%%%%%%%%%%%%%%%%%%%%%%%%%%%%%%%%%%%%%%%%%%%%%%%%%%%%%%%%%%%
\noindent\textbf{Remarks}. (i)
The proof of Lemma~\ref{l:keymic} can be summarized as follows:
\begin{itemize}
\item the anisotropic spaces
give invertibility at the projections of the sets $E_u^*,E_s^*$ to fiber infinity
$\partial\overline T^*\mathcal M$;
\item together, $q_1$ and $Q_\infty$ give invertibility on $\{\rho\leq -\varepsilon\}$;
\item the operator
$Q_\delta$ gives invertibility on the set $\{(x,0)\mid x\in K\}$; and
\item invertibility elsewhere is obtained by propagation of singularities.
\end{itemize}

(ii) One can specify the value of $r_0$ more precisely. Indeed, the condition
$r\geq r_0$ is only needed to ensure that~\eqref{e:sub1}, \eqref{e:sub2} hold.
Examining the proof of Lemma~\ref{l:gs2}, we see immediately that
we can take for some large fixed constant $\widetilde C>0$,
\begin{equation}
  \label{e:improv-0}
r_0=\widetilde C(1+C_1).
\end{equation}
Moreover, if $h^{-1}\Im z$ is large enough and positive, then we can take $r_0=0$.

(iii)
If additionally $\mathbf X^*=-\mathbf X$ near $K$ with respect to some
smooth measure on $\mathcal U$ and some inner product on the fibers of $\mathcal E$
(e.g. when $\mathcal E=\mathbb R$, $\mathbf X=X$,
and $X$ admits a smooth invariant measure), then we can take for some $\widetilde C>0$,
any $r_0$ with
\begin{equation}
  \label{e:improv-1}
r_0>\widetilde C C_1.
\end{equation}
Furthermore, replacing $\gamma/2$ in the proof of Lemma~\ref{l:gs2}
by a constant arbitrarily close to $\gamma$, we can put $\widetilde C:=\gamma^{-1}$, where $\gamma$ is the minimal expansion rate
appearing in~\eqref{e:hyper}.
%%%%%%%%%%%%%%%%%%%%%%%%%%%%%%%%%%%%%%%%%%%%%%%%%%%%%%%%%%%%%%%%%%%%%%%%%%%%%%%%
\begin{proof}
We use the strategy of the proof of~\cite[Proposition~3.4]{DyZw}.
 One could similarly adapt the construction of~\cite[Section~3]{FaSj},
however the method of~\cite{DyZw} is more convenient for the wavefront set statements, needed for Theorem~\ref{t:zeta}.

To reduce $\mathcal H^r_h$ estimates to $L^2$ estimates, we use the conjugated operator
(see~\cite[Section~3.3]{DyZw} for details)
\begin{equation}
  \label{e:p-conj}
\mathbf P_{\delta,r}=e^{rG}\mathbf P_\delta e^{-rG}=\mathbf P_\delta+r[G,\mathbf P_\delta]+\mathcal O(h^2)_{\Psi^{-1+}_h}.
\end{equation}

Since $\WFh(G)\cap (\WFh(Q_\delta)\cup\WFh(Q_\infty))=\emptyset$,
$\sigma_h(Q_\infty),\sigma_h(Q_\delta),q_1$ are real-valued, and $q_1$ is a pseudodifferential
operator of order 0, we get
\begin{equation}
  \label{e:p-conj-1}
\Re\sigma_h(\mathbf P_{\delta,r})=p+\mathcal O(h)_{S^0}.
\end{equation}
Since $H_pm\leq 0$ near $\{p=0\}$, it follows from~\eqref{e:tilde-m} that
$H_p\widetilde m\leq 0$ modulo $S^0$ near $\{\langle\xi\rangle^{-1}p=0\}$. Together with the fact
that $\sigma_h(Q_\infty),\sigma_h(Q_\delta),q_1\geq 0$, this implies
\begin{equation}
  \label{e:p-conj-2}
\Im\sigma_h(\mathbf P_{\delta,r})\leq 0\quad\text{near }\{\langle\xi\rangle^{-1}p=0\}.
\end{equation}
Now, by~\eqref{e:hyperdual}, $L=\kappa(E_s^*)$ satisfies~\eqref{e:rads1}, where
$\kappa$ is defined in~\eqref{e:kappa}.
By part~1 of Lemma~\ref{l:gs2}, there exists $r_0:=r_0(C_1)$ such that
\begin{equation}
  \label{e:sub1}
\Im(\mathbf P_\delta-z)\lesssim -h\quad\text{near }\kappa(E_s^*)\quad\text{on }H^r_h,\quad
r\geq r_0,\quad
\Im z\geq -C_1h;
\end{equation}
here we use Definition~\ref{d:gtrsimh}. Similarly, $L=\kappa(E_u^*)$ satisfies~\eqref{e:rads2}.
By part~2 of Lemma~\ref{l:gs2},
\begin{equation}
  \label{e:sub2}
\Im(\mathbf P_\delta-z)\lesssim -h\quad\text{near }\kappa(E_u^*)\quad\text{on }H^{-r}_h,\quad
r\geq r_0,\quad
\Im z\geq -C_1h.
\end{equation}
Since~\eqref{e:p-conj-2} is true when $q_1$ is removed from $\mathbf P_\delta$,
and $\Re\sigma_h(e^{rG}q_1e^{-rG})=q_1(x)>0$ on $\{\rho=-\varepsilon\}\subset\{\langle\xi\rangle^{-1}p=0\}\subset\overline T^*\mathcal M$, we have
by Lemma~\ref{l:gs1},
\begin{equation}
  \label{e:subq}
\Im(\mathbf P_{\delta,r}-z)\lesssim -h\quad\text{near }\{\rho=-\varepsilon\}\quad\text{on }L^2,\quad
\Im z\geq -C_1h.
\end{equation}
Finally, since $\{\rho=-\varepsilon\}\cap \{X_1\rho=0\}\subset \Ell_h(Q_\infty)$, there exist compact sets
(see Figure~\ref{f:LPM})
\begin{equation}
  \label{e:LPM}
L_\pm\ \subset\ \{\rho=-\varepsilon\}\cap \{\pm X_1\rho<0\}\ \subset\ \mathcal M
\end{equation}
such that on $\overline T^*\mathcal M$, and with $L_\pm^\circ$ denoting the interior of $L_\pm$
inside $\{\rho=-\varepsilon\}$,
\begin{equation}
  \label{e:LPM2}
\{\rho=-\varepsilon\}\cap \{\pm X_1\rho\leq 0\}\ \subset\ \pi^{-1}(L_\pm^\circ)\cup\Ell_h(Q_\infty).
\end{equation}
%%%%%%%%%%%%%%%%%%%%%%%%%%%%%%%%%%%%%%%%%%%%%%%%%%%%%%%%%%%%%%%%%%%%%%%%%%%%%%%%
\begin{figure}
\includegraphics{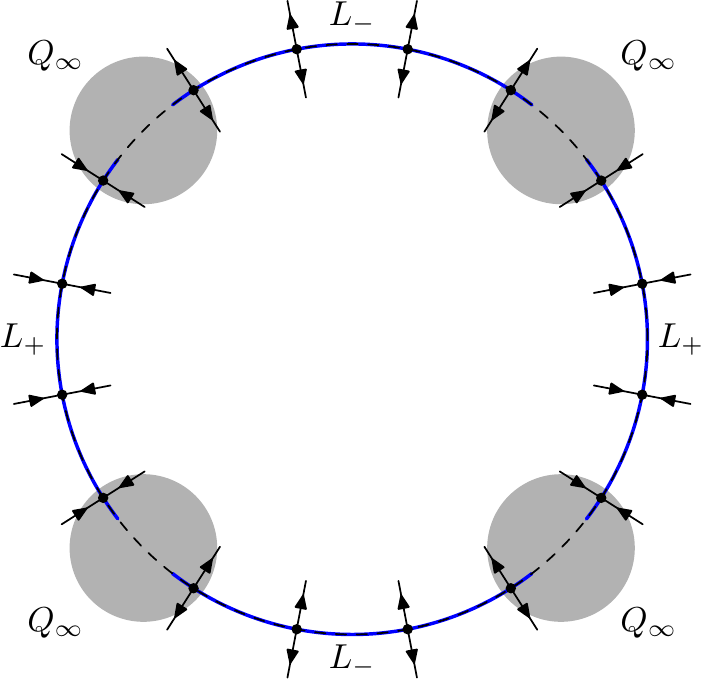}
\caption{The sets $L_\pm$ (thick blue lines) and the elliptic set of $Q_\infty$ (shaded)
near $\{\rho=-\varepsilon\}$ (dashed circle). The arrows depict the flow $\varphi^t$.}
\label{f:LPM}
\end{figure}
%%%%%%%%%%%%%%%%%%%%%%%%%%%%%%%%%%%%%%%%%%%%%%%%%%%%%%%%%%%%%%%%%%%%%%%%%%%%%%%%
The proof of the lemma is based on the following bound similar to~\cite[(3.10)]{DyZw}:
\begin{equation}
  \label{e:direct-bound}
\|\mathbf u\|_{\mathcal H^r_h}\leq Ch^{-1}\|\mathbf f\|_{\mathcal H^r_h},\quad
\mathbf u\in \mathcal D^r_h,\quad
\mathbf f=(\mathbf P_\delta-z)\mathbf u.
\end{equation}
By~\cite[Lemma~A.1]{FaSj} applied to $e^{rG}\mathbf u$ and the operator $\mathbf P_{\delta,r}\in \Psi^1_h$, for
each fixed $h$ and each $\mathbf u\in\mathcal D^r_h$ there exists a sequence
$\mathbf u_j\in C^\infty(\mathcal M;\mathcal E)$ such that
$\mathbf u_j\to\mathbf u$ in $\mathcal H^r_h$ and
$(\mathbf P_\delta-z)\mathbf u_j\to\mathbf f$ in $\mathcal H^r_h$.
Therefore, it suffices to prove~\eqref{e:direct-bound} for the case
$\mathbf u\in C^\infty(\mathcal M;\mathcal E)$.

We now use semiclassical estimates to obtain bounds on $A\mathbf u$, where
$A\in\Psi^0_h(\mathcal M)$ falls into one of the following cases.
We will typically arrive to a propagation estimate of the form
\begin{equation}
  \label{e:case-0}
\|A\mathbf u\|_{\mathcal H^r_h}\leq C\|B \mathbf u\|_{\mathcal H^r_h}
+Ch^{-1}\|B_1\mathbf f\|_{\mathcal H^r_h}+\mathcal O(h^\infty)\|\mathbf u\|_{\mathcal H^r_h},
\end{equation}
for some choice of operators $B,B_1\in\Psi^0_h(\mathcal M)$. The term $B\mathbf u$
will be controlled by previously considered cases and we keep track of the wavefront set of $B_1$
to show~\eqref{e:WF-absorbed}.

\noindent\textbf{Case 1}: $\WFh(A)\cap\{\langle\xi\rangle^{-1}p=0\}\ \subset\ \Ell_h(Q_\infty)\cup\Ell_h(Q_\delta)$. Then
$\mathbf P_{\delta}-z$, and thus $\mathbf P_{\delta,r}-z$, is elliptic on $\WFh(A)$.
Similarly to~\cite[Proposition~3.4, Case~1]{DyZw}, we find for some $B_1\in\Psi^0_h(\mathcal M)$
microlocalized in a small neighborhood of $\WFh(A)$,
\begin{equation}
  \label{e:case-1}
\|A\mathbf u\|_{\mathcal H^r_h}\leq C\|B_1 \mathbf f\|_{\mathcal H^r_h}+\mathcal O(h^\infty)\|\mathbf u\|_{\mathcal H^r_h}.
\end{equation}
Note that the $\mathcal H^r_h$ bound on the operator $\mathbf P_\delta-z$
is equivalent to the $L^2$ bound on the operator $\mathbf P_{\delta,r}-z$;
we will use this fact in the next cases.

\noindent\textbf{Case 2}: $\WFh(A)$ is contained in a small neighborhood of $\pi^{-1}(L_-)\subset \overline T^*\mathcal M$,
where $L_-$ is defined in~\eqref{e:LPM} and $\pi:\overline T^*\mathcal M\to\mathcal M$ is the projection map;
moreover, $\pi^{-1}(L_-)\subset\Ell_h(A)$.

For each $(x,\xi)\in\WFh(A)$, $\varphi^t(x)$ uniformly converges to
$\{\rho=-\varepsilon\}\cap \{X_1\rho>0\}$ as $t\to -\infty$. Here we used that
$L_-\subset\{X_1\rho>0\}$,
$\varphi^t(x)=e^{tX}$, $X=\psi(\rho)X_1$, and $\sgn\psi(\rho)=\sgn(\rho+\varepsilon)$
(see Lemma~\ref{l:extconv} and Figures~\ref{f:LPM} and~\ref{f:case2}).

Take $B\in\Psi^0_h$ such that
$$
\{\rho=-\varepsilon\}\cap \{X_1\rho\geq 0\}\setminus \pi^{-1}(L_-^\circ)\ \subset\ \Ell_h(B),\quad
\WFh(B)\subset \Ell_h(Q_\infty).
$$
We apply Lemma~\ref{l:ultimate}, with $\mathbf P=\mathbf P_{\delta,r}-z$,
$L=\pi^{-1}(L_-)$, $s=0$, and $B_1$ elliptic in a sufficiently large neighborhood
of $L$ depending on $\WFh(A)$. All assumptions
of this lemma are satisfied, except for the condition $L\cap\WFh(B)=\emptyset$.
Indeed, $L$ is invariant under $e^{TH_p}$
since $L_-$ consists of fixed points of $X$ and thus is invariant under $\varphi_t$.
The condition~\eqref{e:subq} implies that $\Im\mathbf P\lesssim -h$ on $L^2$
near $L$. Moreover, for each $(x,\xi)\in\Omega\cap\WFh(A)$, the point
$\lim_{t\to -\infty}\varphi^t(x)$
lies in $L_-^\circ$.

Finally, the condition $L\cap\WFh(B)=\emptyset$ can be waived as it is only used in Lemma~\ref{l:ultimatesc}
and we can instead construct
the required function $\chi$ directly. In fact, 
using the coordinates~\eqref{e:funkco}, we see that
there exists
$\chi=\chi(x)\in C^\infty(\mathcal M;[0,1])$ supported in an arbitrarily small neighborhood of $L_-$
such that $\chi=1$ near $L_-$ and $H_p\chi\leq 0$ everywhere.

Now, the estimate~\eqref{e:ultimate} gives~\eqref{e:case-0}
for some $B_1\in\Psi^0_h(\mathcal M)$ microlocalized in a small neighborhood of $\pi^{-1}(L_-)$.
The term $B\mathbf u$ is controlled by Case~1.

%%%%%%%%%%%%%%%%%%%%%%%%%%%%%%%%%%%%%%%%%%%%%%%%%%%%%%%%%%%%%%%%%%%%%%%%%%%%%%%%
\begin{figure}
\includegraphics{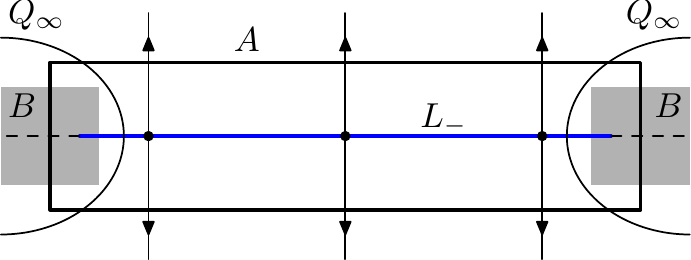}
\qquad
\includegraphics{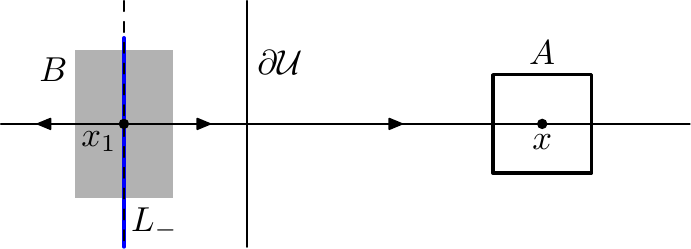}
\hbox to\hsize{\hss Case 2\hss\quad\hss Case 3\hss}
\caption{An illustration of Cases 2 and 3, with the flow lines
of $\varphi^t$ drawn. The solid blue lines
are $L_-$, the dashed lines containing them are $\{\rho=-\varepsilon\}$,
and the semicircles denote $\Ell_h(Q_\infty)$.
In this and the following figures,
the elliptic sets of $B$ are shaded and $\WFh(A)$ are pictured by the rectangles.
}
\label{f:case2}
\end{figure}
%%%%%%%%%%%%%%%%%%%%%%%%%%%%%%%%%%%%%%%%%%%%%%%%%%%%%%%%%%%%%%%%%%%%%%%%%%%%%%%%

\noindent\textbf{Case 3}: $\WFh(A)$ is contained in a small neighborhood
of some $(x,\xi)\in \{\langle\xi\rangle^{-1}p=0\}$, where
$x\in\overline{\mathcal U}\setminus\Gamma_+$. By~\eqref{e:gpm}, there exists
$T>0$ such that $\varphi^{-T}(x)\notin\overline{\mathcal U}$.
Similarly to the proof of Lemma~\ref{l:marron3}, we use
part~2 of Lemma~\ref{l:marron0} to see that
$X_1\rho(\varphi^{-T}(x))>0$. We apply part~2 of Lemma~\ref{l:marron0}
(with $[\alpha,\beta]=[-\varepsilon,0]$)
again to see that there exists $T'>0$ such that
$x_1:=e^{-T'X_1}(\varphi^{-T}(x))\in \{\rho=-\varepsilon\}\cap\{X_1\rho>0\}$
and $e^{-tX_1}(\varphi^{-T}(x))\in \{-\varepsilon<\rho<0\}$ for all $t\in [0,T']$.
Since $X=\psi(\rho)X_1$, it follows that (see Figure~\ref{f:case2})
$$
\varphi^{-t}(x)\ \to\ x_1\in \{\rho=-\varepsilon\}\cap \{X_1\rho>0\}\quad\text{as }t\to +\infty.
$$
By~\eqref{e:LPM2}, there exists $B\in\Psi^0_h$ such that $\pi^{-1}(x_1)\subset\Ell_h(B)$
and $B\mathbf u$ is controlled either by Case~1 (if $\pi^{-1}(x_1)\subset\Ell_h(Q_\infty)$)
or by Case~2 (if $x_1\in L_-$). By propagation of singularities (Lemma~\ref{l:propagation})
applied to $\mathbf P_{\delta,r}$, the estimate~\eqref{e:case-0} holds for some
$B_1\in\Psi^0_h$ microlocalized in a small neighborhood of
$\{e^{-tH_p}(x,\xi)\mid t\geq 0\}$.

\noindent\textbf{Case 4}: $\WFh(A)$ is contained in a small neighborhood $U$ of $\kappa(E_s^*)$, where
$\kappa$ is defined in~\eqref{e:kappa}; moreover, $\kappa(E_s^*)\subset\Ell_h(A)$.
Take $B,B_1\in\Psi^0_h$ such that for some arbitrarily small fixed open sets $V\supset\kappa(E_-^*)\cap\{\rho=\varepsilon\}$ and $W\supset\kappa(E_-^*)$
$$
\begin{aligned}
\kappa(E_-^*)\cap \{\rho=\varepsilon\}\ \subset\ \Ell_h(B),&\quad
\WFh(B)\ \subset\ V;\\
\kappa(E_-^*)\ \subset\ \Ell_h(B_1),&\quad
\WFh(B_1)\ \subset\ W;
\end{aligned}
$$
see Figure~\ref{f:case4}. We also assume that~\eqref{e:aniseq-1} holds for the operators $A,B,B_1$.

We claim that for some choice of $U$ depending on $B,B_1$,
\begin{equation}
  \label{e:containment}
U\cap \{\langle\xi\rangle^{-1}p=0\}\setminus \pi^{-1}(\Gamma_+)\ \subset\ \Con_p(\Ell_h(B);\Ell_h(B_1)),
\end{equation}
see Definition~\ref{d:con} for the notation on the right-hand side. To see~\eqref{e:containment},
we first note that by Lemma~\ref{l:global-dynamics}, there exists $T'\geq 0$
(depending on $B,B_1$, but not on $U$, as long as $U$ lies inside a fixed small neighborhood of $\kappa(E_s^*)$)
such that
for each $T\geq T'$ and each
\begin{equation}
  \label{e:lalacond}
(x,\xi)\in U\cap \{\langle\xi\rangle^{-1}p=0\},\quad\varphi^{-T}(x)\in\{\rho=\varepsilon\},
\end{equation}
we have
\begin{equation}
  \label{e:lalaconc}
e^{-TH_p}(x,\xi)\in \Ell_h(B);\quad
e^{-tH_p}(x,\xi)\in \Ell_h(B_1)\quad\text{for all }t\in [T', T].
\end{equation}
Since $\kappa(E_s^*)$ is invariant under the flow and lies inside $\Ell_h(B_1)\setminus \{\rho=\varepsilon\}$,
we can make sure that~\eqref{e:lalacond} never holds for $T\in [0,T')$ and
\eqref{e:lalaconc} holds for all $t\in [0,T]$, as long as
$U$ is chosen small enough depending on $B,B_1,T'$.
Now, for each $(x,\xi)\in U\cap\{\langle\xi\rangle^{-1}p=0\}\setminus\pi^{-1}(\Gamma_+)$, there
exists $T\geq 0$ such that~\eqref{e:lalacond} holds. Then
\eqref{e:lalaconc} implies that $(x,\xi)\in\Con_p(\Ell_h(B);\Ell_h(B_1))$,
which proves~\eqref{e:containment}.

%%%%%%%%%%%%%%%%%%%%%%%%%%%%%%%%%%%%%%%%%%%%%%%%%%%%%%%%%%%%%%%%%%%%%%%%%%%%%%%%
\begin{figure}
\includegraphics{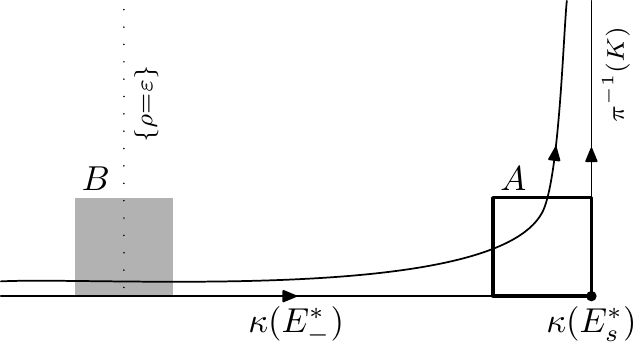}
\qquad
\includegraphics{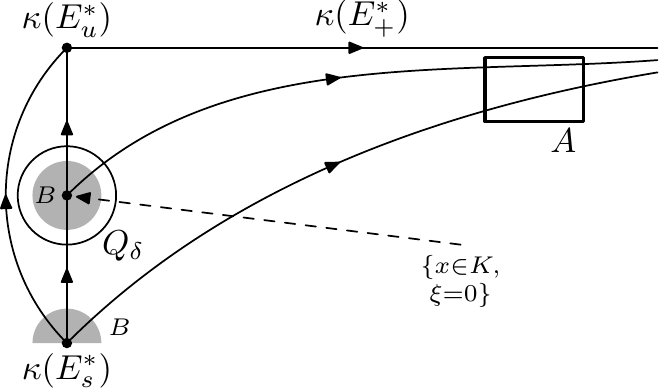}
\hbox to\hsize{\hss Case 4\hss\quad\hss Case 5\hss}
\caption{An illustration of Cases 4 and 5, with the flow lines
of $e^{tH_p}$ drawn on $\overline T^*_{\Gamma_-}\mathcal M$ (Case~4)
and $\overline T^*_{\Gamma_+}\mathcal M$ (Case~5).
The right (Case~4) and left (Case~5) side of the pictures is $\pi^{-1}(K)$.}
\label{f:case4}
\end{figure}
%%%%%%%%%%%%%%%%%%%%%%%%%%%%%%%%%%%%%%%%%%%%%%%%%%%%%%%%%%%%%%%%%%%%%%%%%%%%%%%%

We now apply Lemma~\ref{l:ultimate}, with $\mathbf P=\mathbf P_\delta-z$, $L=\kappa(E_s^*)$, $s=r$.
To verify that $\Im \mathbf P\lesssim -h$ near $L$ on $H^r_h$, we use~\eqref{e:sub1}.
By~\eqref{e:containment}, we have $\Omega\cap \WFh(A)\subset \pi^{-1}(\Gamma_+)$. By Lemmas~\ref{l:convergence}
and~\ref{l:global-dynamics},
we have $e^{-tH_p}(x,\xi)\to L$ as $t\to +\infty$ uniformly in $(x,\xi)\in\Omega\cap\WFh(A)$.
Finally, by~\eqref{e:aniseq-1}, the space $H^r_h$ can be replaced by $\mathcal H^r_h$ in the estimate.

We see that~\eqref{e:ultimate} gives the estimate~\eqref{e:case-0}.
By Lemma~\ref{l:gpmclosed},
$K\cap \{\rho=\varepsilon\}=\emptyset$ for $\varepsilon$ small enough; therefore
 we can choose $V$ so that $\pi(\WFh(B))\subset\mathcal U\setminus\Gamma_+$.
Then the term $B\mathbf u$ is controlled by Case~3.

\noindent\textbf{Case 5}: $\WFh(A)$ is contained in a small neighborhood
of some $(x,\xi)\in \{\langle\xi\rangle^{-1}p=0\}$, where
$x\in \Gamma_+$ and $(x,\xi)\notin\kappa(E_+^*)$. If $\xi\notin \overline{E_+^*(x)}$, then
by part~4 of Lemma~\ref{l:extended},
we have $e^{-tH_p}(x,\xi)\to \kappa(E_s^*)$ as $t\to +\infty$.
Otherwise $\xi\in E_+^*(x)$ does not lie on the fiber infinity; by part~3
of Lemma~\ref{l:extended}, we have $e^{-tH_p}(x,\xi)\to \{x\in K,\ \xi=0\}$
as $t\to +\infty$.

Similarly to Case~3, we use propagation of singularities
to obtain the estimate~\eqref{e:case-0}, where $B_1$ is microlocalized in a small neighborhood
of $\{e^{-tH_p}(x,\xi)\mid t\geq 0\}$ and
$\WFh(B)$ lies either in a small neighborhood of $\kappa(E_s^*)$ or
in a small neighborhood of $\{x\in K,\ \xi=0\}$. In the first case,
$B\mathbf u$ is controlled by Case~4; in the second case,
$\WFh(B)\subset\Ell_h(Q_\delta)$ and $B\mathbf u$ is controlled by Case~1.
See Figure~\ref{f:case4}.

\noindent\textbf{Case 6}: $\WFh(A)$ is contained in a small neighborhood $U$ of
$\kappa(E_u^*)$; moreover, $\kappa(E_u^*)\subset\Ell_h(A)$.
Take $B\in \Psi^0_h$ such that (see Figure~\ref{f:case6})
$$
\kappa(E_s^*)\cup \{x\in K,\ \xi=0\}\cup(\{\rho=\varepsilon\}\cap\pi^{-1}(\Gamma_-))\ \subset\ \Ell_h(B)
$$
and $\WFh(B)$ lies in a small neighborhood of the above set.
Let $B_1\in\Psi^0_h$ satisfy $\{\langle\xi\rangle^{-1}p=0\}\subset\Ell_h(B_1)$
and~\eqref{e:p-conj-2} hold near $\WFh(B_1)$. We claim that for $U$ small enough,
\begin{equation}
  \label{e:containment2}
U\cap \{\langle\xi\rangle^{-1}p=0\}\setminus\kappa(E_+^*)\ \subset\ \Con_p(\Ell_h(B);\Ell_h(B_1)).
\end{equation}
To show~\eqref{e:containment2}, take $(x,\xi)\in U\cap \{\langle\xi\rangle^{-1}p=0\}\setminus\kappa(E_+^*)$.
If $x\in\Gamma_+$, then by the analysis of Case~5, $(x,\xi)\in\Con_p(\Ell_h(B);\Ell_h(B_1))$.
If $x\notin\Gamma_+$, then there exists $T>0$ such that
$\varphi^{-T}(x)\in \{\rho=\varepsilon\}$; we claim that $\varphi^{-T}(x)\in \Ell_h(B)$.
Indeed, otherwise $\varphi^{-T}(x)$ does not lie in some fixed closed subset of $\overline{\mathcal U}$
which does not intersect $\Gamma_-$, which implies that $\varphi^{t-T}(x)\notin\pi(U)$
for $\pi(U)$ a small enough neighborhood of $K$ and all $t\geq 0$; putting $t:=T$, we get a contradiction.

We now apply Lemma~\ref{l:ultimate}
with $\mathbf P=\mathbf P_\delta-z$, $L=\kappa(E_u^*)$, $s=-r$.
To see that $\Im\mathbf P\lesssim -h$ near $L$ on $H^{-r}_h$, we use~\eqref{e:sub2}.
By~\eqref{e:containment2}, $\Omega\cap \WFh(A)\subset \kappa(E_+^*)$.
Then by Lemma~\ref{l:convergence}
and the invariance of $E_+^*$ under the flow, $e^{-tH_p}(x,\xi)\to\kappa(E_+^*)\cap \pi^{-1}(K)=
\kappa(E_u^*)$ as $t\to +\infty$ uniformly in $(x,\xi)\in\Omega\cap\WFh(A)$.

By~\eqref{e:ultimate}, we obtain~\eqref{e:case-0}.
The space $H^{-r}_h$ can be replaced in~\eqref{e:ultimate} by $\mathcal H^r_h$;
indeed, \eqref{e:tinerz} still holds by~\eqref{e:aniseq-2}
and~\eqref{e:ullie1}, \eqref{e:ullie2} follow by propagation
of singularities for the conjugated operator $\mathbf P_{\delta,r}-z$.
The term $B\mathbf u$ is controlled by Cases~1, 3, and~4,
corresponding to the parts of $B$ lying near
$\{x\in K,\ \xi=0\}$, $\{\rho=\varepsilon\}\cap \pi^{-1}(\Gamma_-)$,
and $\kappa(E_s^*)$ respectively.

%%%%%%%%%%%%%%%%%%%%%%%%%%%%%%%%%%%%%%%%%%%%%%%%%%%%%%%%%%%%%%%%%%%%%%%%%%%%%%%%
\begin{figure}
\includegraphics{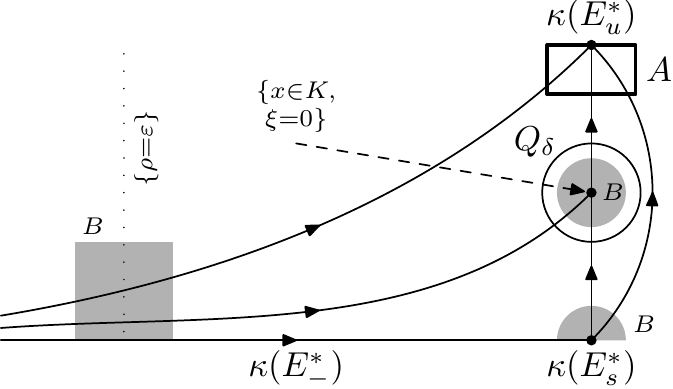}
\qquad
\includegraphics{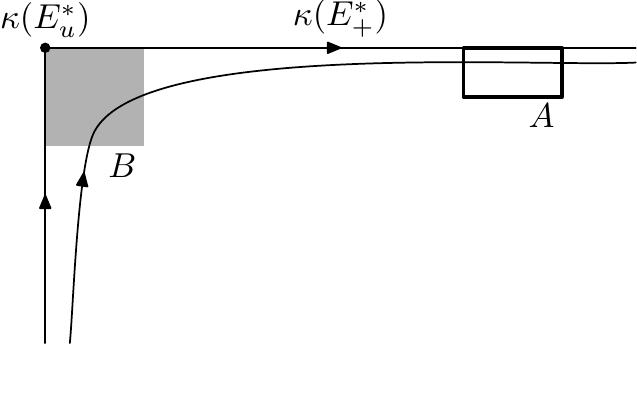}
\hbox to\hsize{\hss Case 6\hss\quad\hss Case 7\hss}
\caption{An illustration of Cases 6 and 7, with the flow lines
of $e^{tH_p}$ drawn on $\overline T^*_{\Gamma_-}\mathcal M$ (Case~6)
and $\overline T^*_{\Gamma_+}\mathcal M$ (Case~7).}
\label{f:case6}
\end{figure}
%%%%%%%%%%%%%%%%%%%%%%%%%%%%%%%%%%%%%%%%%%%%%%%%%%%%%%%%%%%%%%%%%%%%%%%%%%%%%%%%

\noindent\textbf{Case 7}: $\WFh(A)$ is contained in a small neighborhood
of some $(x,\xi)\in \kappa(E_+^*)$. Take $B\in\Psi^0_h$ which is microlocalized
in a small neighborhood of $\kappa(E_u^*)$ and $\kappa(E_u^*)\subset\Ell_h(B)$.
Let $B_1$ be as in Case~6.
By Lemma~\ref{l:convergence} and the invariance of $E_+^*$, we see that
$\WFh(A)\subset\Con_p(\Ell_h(B);\Ell_h(B_1))$. Similarly to Case~3,
propagation of singularities gives~\eqref{e:case-0}. The term $B\mathbf u$
is controlled by Case~6. See Figure~\ref{f:case6}.

\noindent\textbf{Case 8}: $\WFh(A)$ is contained in a small neighborhood
of some $(x,\xi)\in \{\langle\xi\rangle^{-1}p=0\}$
and $(x,\xi)\notin\Ell_h(Q_\infty)\cup\pi^{-1}(\overline{\mathcal U}\cup L_-\cup L_+)$.
Here $L_\pm$ are defined in~\eqref{e:LPM}.
Then
$x\in \{-2\varepsilon<\rho<0\}\setminus\{\rho=-\varepsilon\}$. By part~1 of Lemma~\ref{l:marron0}
(with $[\alpha,\beta]=[-2\varepsilon,\varepsilon]$ or $[\alpha,\beta]=[-\varepsilon,0]$), 
and since $X=\psi(\rho)X_1$, $\sgn\psi(\rho)=\sgn(\rho+\varepsilon)$, we see that
one of the following holds (see Figure~\ref{f:case9})
\begin{enumerate}
\item there exists $T\geq 0$ such that $x_1:=\varphi^{-T}(x)\in \partial\overline{\mathcal U}$, or
\item there exists $T\geq 0$ such that $x_1:=\varphi^{-T}(x)\in \{\rho=-2\varepsilon\}$, or
\item there exists $x_1\in \{\rho=-\varepsilon\}\cap \{X_1\rho\geq 0\}$ such that
$\varphi^{-t}(x)\to x_1$ as $t\to +\infty$.
\end{enumerate}
Take $B\in\Psi^0_h$ such that $\pi^{-1}(x_1)\in\Ell_h(B)$, but
$\pi(\WFh(B))$ lies in a small neighborhood of $x_1$. Let $B_1$ be as in Case~6.
Similarly to Case~3, by propagation of singularities we get~\eqref{e:case-0}. The term $B\mathbf u$
can be estimated in each of the situations above as follows:
\begin{enumerate}
\item by Cases~1, 3, 5, and~7;
\item by Case~1, since $\pi^{-1}(x_1)\subset\Ell_h(Q_\infty)$;
\item by Case~2 if $x_1\in L_-$, and by Case~1 otherwise (as then $\pi^{-1}(x_1)\subset\Ell_h(Q_\infty)$).
\end{enumerate}

\noindent\textbf{Case 9}: $\WFh(A)$ is contained in a small neighborhood
of $\pi^{-1}(L_+)$ and
$\pi^{-1}(L_+)\subset\Ell_h(A)$, where $L_+$ is defined in~\eqref{e:LPM}.
We in particular require that
$$
\WFh(A)\subset \textstyle\{-{3\over 2}\varepsilon<\rho<-{1\over 2}\varepsilon\}\cap \{X_1\rho<0\}.
$$
Take $B\in\Psi^0_h$ such that (see Figure~\ref{f:case9})
$$
  \label{e:zzBreq}
\begin{gathered}
\WFh(B)\ \subset\ \big(\{-2\varepsilon<\rho<0\}\cap\{\rho\neq-\varepsilon\}\big)
\cup\Ell_h(Q_\infty),\\
\pi^{-1}\big(\{\rho=-\varepsilon\}\cap\{X_1\rho\leq 0\}\setminus L_+^\circ\big)\ \subset\ \Ell_h(B),\\
\textstyle\pi^{-1}\big(\{\rho=-{3\over 2}\varepsilon\}\cup \{\rho=-{1\over 2}\varepsilon\}\big)\ \subset\ \Ell_h(B).
\end{gathered}
$$
We apply Lemma~\ref{l:ultimate}, with $\mathbf P=\mathbf P_{\delta,r}-z$,
$L=\pi^{-1}(L_+)$, $s=0$, and $B_1$ chosen as in Case~6.
To verify that $\Im\mathbf P\lesssim -h$ near $L$ on $L^2$, we use~\eqref{e:subq}.
The condition $L\cap\WFh(B)=\emptyset$ does not hold, but similarly to Case~2 it can be waived
by taking a function $\chi\in C^\infty(\mathcal M;[0,1])$ which is supported in a small enough neighborhood
of $L_+$, but $\chi=1$ near $L_+$. As follows from the next paragraph, this function
satisfies the conclusions of Lemma~\ref{l:ultimatesc}, in fact $H_p\chi=0$ near $\Omega$.

To finish verifying the assumptions of Lemma~\ref{l:ultimate}, note that $\WFh(A)\cap\Omega\subset\pi^{-1}(L_+^\circ)$.
Indeed, let $(x,\xi)\in\WFh(A)$. If $x\notin\{\rho=-\varepsilon\}$, then by part~1 of Lemma~\ref{l:marron0}
(with $[\alpha,\beta]=[-{3\over 2}\varepsilon,-\varepsilon]$ or $[\alpha,\beta]=[-\varepsilon,-{1\over 2}\varepsilon]$),
either $\varphi^{-T}(x)\in \{\rho=-{3\over 2}\varepsilon\}\cup\{\rho=-{1\over 2}\varepsilon\}$
for some $T\geq 0$, or $\varphi^{-t}(x)\to x_1\in \{\rho=-\varepsilon\}$ as $t\to +\infty$.
The latter option is impossible if $\pi(\WFh(A))$ is sufficiently close to $L_+\subset\{X_1\rho<0\}$,
and the former option gives $(x,\xi)\notin\Omega$. If $x\in \{\rho=-\varepsilon\}$, then
we also have $x\in \{X_1\rho<0\}$. Therefore, either $x\in L_+^\circ$ or $(x,\xi)\in\Ell_h(B)$; in the latter
case, $(x,\xi)\notin\Omega$.

Now, the estimate~\eqref{e:ultimate} gives~\eqref{e:case-0}. The term $B\mathbf u$ can be estimated
by Cases~1 and~8.

%%%%%%%%%%%%%%%%%%%%%%%%%%%%%%%%%%%%%%%%%%%%%%%%%%%%%%%%%%%%%%%%%%%%%%%%%%%%%%%%
\begin{figure}
\includegraphics{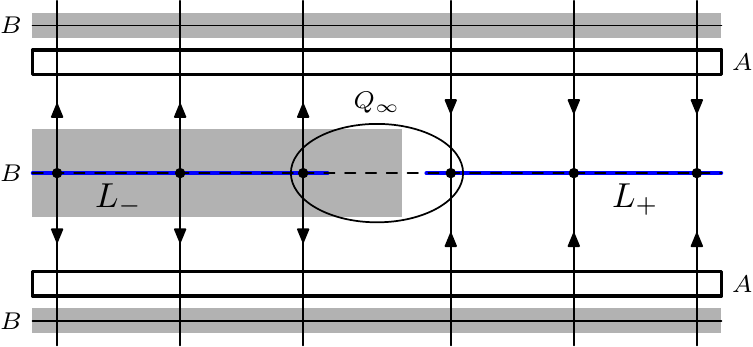}
\quad
\includegraphics{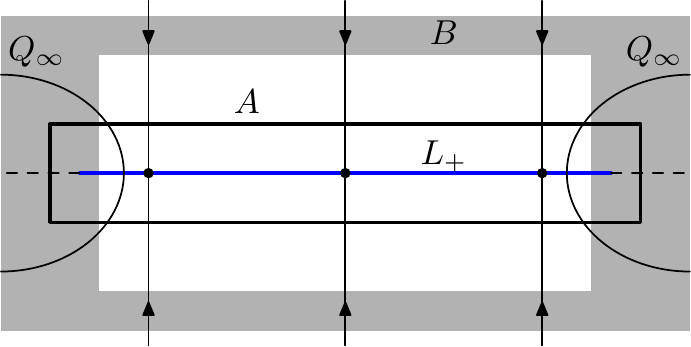}
\hbox to\hsize{\hss Case 8\hss\qquad\hss Case 9\hss}
\caption{An illustration of Cases~8 and~9, with the flow lines
of $\varphi^t$ drawn. The solid blue lines
are $L_\pm$ and the dashed lines containing them are $\{\rho=-\varepsilon\}$;
the solid lines on the top and bottom of Case~8 are $\partial\mathcal U$
and $\{\rho=-2\varepsilon\}$.
The (semi)circles denote $\Ell_h(Q_\infty)$.}
  \label{f:case9}
\end{figure}
%%%%%%%%%%%%%%%%%%%%%%%%%%%%%%%%%%%%%%%%%%%%%%%%%%%%%%%%%%%%%%%%%%%%%%%%%%%%%%%%

\smallskip

Combining the above cases and using a pseudodifferential partition
of unity, we get the estimate~\eqref{e:direct-bound}. More precisely,
if $A\in\Psi^0_h$ and $\WFh(A)$ lies in a small neighborhood of $\pi^{-1}(\overline{\mathcal U})$,
then $A\mathbf u$ is estimated by a combination of Cases~1, 3, 5, and~7.
If $\pi^{-1}(\WFh(A))\cap\overline{\mathcal U}=\emptyset$, then $A\mathbf u$ is estimated by
a combination of Cases~1, 2, 8, and~9.

Reversing the direction of propagation
(replacing $X_1$ by $-X_1$, $\mathbf X$ by $-\mathbf X$, $m$ by $-m$, and switching $E_s^*$ with $E_u^*$, $E_+^*$ with $E_-^*$, and
$L_+$ with $L_-$), we repeat the above reasoning to get the adjoint estimate
similar to~\cite[(3.17)]{DyZw}
\begin{equation}
  \label{e:dual-bound}
\|\mathbf u\|_{\mathcal H^{-r}_h}\leq Ch^{-1}\|\mathbf f\|_{\mathcal H^{-r}_h},\quad
\mathbf u\in \mathcal H^{-r}_h,\quad
\mathbf f:=(\mathbf P_\delta^*-\bar z)\mathbf u\in\mathcal H^{-r}_h.
\end{equation}
Note that $\mathcal H^{-r}_h$ is dual to $\mathcal H^r_h$ with respect to the $L^2$ pairing.
The functional analytic argument
given at the end of the proof of~\cite[Proposition~3.4]{DyZw} shows that together, \eqref{e:direct-bound}
and~\eqref{e:dual-bound} imply invertibility
of $\mathbf P_\delta-z:\mathcal D^r_h\to\mathcal H^r_h$ and the bound~\eqref{e:keymicb}.

It remains to verify the wavefront set condition~\eqref{e:WF-absorbed}. By~\cite[Lemma~2.3]{DyZw},
it suffices to show that for each $(x,\xi),(y,\eta)\in T^*\mathcal U$ such that
$(y,\eta)\neq (x,\xi)$ and either $p(x,\xi)\neq 0$ or
$(x,\xi)\neq e^{tH_p}(y,\eta)$ for all $t\geq 0$, there exist $A\in\Psi^{\comp}_h(\mathcal M)$
and $B_2\in\Psi^0_h(\mathcal M)$ such that
$$
(x,\xi)\in\Ell_h(A),\ (y,\eta)\notin\WFh(B_2),
$$
and for each $\mathbf u\in\mathcal D^r_h$ with $\mathbf f:=(\mathbf P_\delta-z)\mathbf u$,
\begin{equation}
  \label{e:wf-bound}
\|A\mathbf u\|_{\mathcal H^r_h}\leq Ch^{-1}\|B_2\mathbf f\|_{\mathcal H^r_h}
+\mathcal O(h^\infty)\|\mathbf u\|_{\mathcal H^r_h}.
\end{equation}
As remarked
after~\eqref{e:direct-bound}, an approximation argument reduces us to the case
$\mathbf u\in C^\infty$. Then~\eqref{e:wf-bound} follows by a combination of Cases~1, 3, and~5.
Here we use that the operator $B_1$ from Case~2 is microlocalized in a small neighborhood
of $\pi^{-1}(L_-)\subset\pi^{-1}(\mathcal M\setminus\mathcal U)$ and
the same operator from Case~4 is microlocalized in a small neighborhood of $\kappa(E_-^*)\subset\partial\overline T^*\mathcal M$;
thus their wavefront sets do not contain $(y,\eta)$.
\end{proof}
%%%%%%%%%%%%%%%%%%%%%%%%%%%%%%%%%%%%%%%%%%%%%%%%%%%%%%%%%%%%%%%%%%%%%%%%%%%%%%%%

%%%%%%%%%%%%%%%%%%%%%%%%%%%%%%%%%%%%%%%%%%%%%%%%%%%%%%%%%%%%%%%%%%%%%%%%%%%%%%%%
\subsection{Proofs of Theorems~\ref{t:mer} and~\ref{t:res-states}}
  \label{s:proofs}

In this section, we show the meromorphic continuation of the resolvent $\mathbf R(\lambda)$
defined in~\eqref{e:res}. We start with the following corollary of Lemma~\ref{l:keymic}:
%%%%%%%%%%%%%%%%%%%%%%%%%%%%%%%%%%%%%%%%%%%%%%%%%%%%%%%%%%%%%%%%%%%%%%%%%%%%%%%%
\begin{lemm}
  \label{l:meromorphic}
Let $Q_\infty\in\Psi^1_h(\mathcal M),Q_\delta\in \Psi^{\comp}_h(\mathcal M),q_1\in C^\infty(\mathcal M)$ be
introduced in~\eqref{e:absorbers-def},
and $\mathcal H^r_h,\mathcal D^r_h$ be given by~\eqref{e:anisop}, \eqref{e:anisopd}.
Fix $C_1,C_2>0$ and $r>r_0=r_0(C_1)$. Define
\begin{equation}
  \label{e:p-0}
\mathbf P_0=\mathbf P_0(h):={h\over i}\mathbf X-i(Q_\infty+q_1):\mathcal D^r_h\to\mathcal H^r_h.
\end{equation}
Then for $0<h<h_0(C_1,C_2,r)$,

1. $\mathbf P_0-z:\mathcal D^r_h\to\mathcal H^r_h$ is a Fredholm operator of index zero
for $z\in [-C_2h,C_2h]+i[-C_1h,1]$.

2. The inverse
\begin{equation}
  \label{e:r-0}
\mathbf R_0(z):=(\mathbf P_0-z)^{-1}:\mathcal H^r_h\to\mathcal D^r_h,\quad
z\in [-C_2h,C_2h]+i[-C_1h,1]
\end{equation}
is a meromorphic family of operators with poles of finite rank.
\end{lemm}
%%%%%%%%%%%%%%%%%%%%%%%%%%%%%%%%%%%%%%%%%%%%%%%%%%%%%%%%%%%%%%%%%%%%%%%%%%%%%%%%
\begin{proof}
1. Take $z\in [-C_2h,C_2h]+i[-C_1h,1]$. By Lemma~\ref{l:keymic}, $\mathbf P_\delta-z:\mathcal D^r_h\to\mathcal H^r_h$
is invertible, where $\mathbf P_\delta$ is defined in~\eqref{e:P-delta}. We write
$$
\mathbf P_\delta-z=\mathbf P_0-z-iQ_\delta.
$$
Now, $Q_\delta$ is compactly microlocalized (that is, $\WFh(Q_\delta)\Subset T^*\mathcal M$)
so it is smoothing; that is, $Q_\delta$ is bounded
$H^{-N}(\mathcal M)\to H^{N}(\mathcal M)$ for all $N$. By Rellich's Theorem (using the fact that $\mathcal M$ is compact
and $e^{\pm rG}$ are pseudodifferential operators), we see that
$Q_\delta$ is a compact operator $\mathcal H^r_h\to\mathcal H^r_h$
and thus $\mathcal D^r_h\to\mathcal H^r_h$. It follows
that $\mathbf P_0-z:\mathcal D^r_h\to\mathcal H^r_h$ is a Fredholm operator of index zero.

2. The meromorphy of $\mathbf R_0(z)$ follows by analytic Fredholm theory~\cite[Proposition~D.4]{e-z},
as long as $\mathbf P_0-z$ is known to be invertible for at least one value of $z$.
We take $z=i$; it suffices to prove the estimate
\begin{equation}
  \label{e:upperhlf}
\|\mathbf u\|_{\mathcal H^r_h}\leq Ch^{-1}\|\mathbf f\|_{\mathcal H^r_h},\quad
\mathbf u\in\mathcal D^r_h,\quad
\mathbf f=(\mathbf P_0-i)\mathbf u.
\end{equation}
Similarly to~\eqref{e:p-conj}, let $\mathbf P_{0,r}:=e^{rG}\mathbf P_0e^{-rG}\in\Psi^1_h(\mathcal M;\mathcal E)$.
Note that $\Im\sigma_h(\mathbf P_{0,r})\leq 0$ near $\{\langle\xi\rangle^{-1}p=0\}$ 
and $\Re\sigma_h(\mathbf P_{0,r})=p$ similarly to~\eqref{e:p-conj-1},
\eqref{e:p-conj-2}. By~\eqref{e:anisop} and the approximation argument following~\eqref{e:direct-bound}, we reduce
\eqref{e:upperhlf} to
\begin{equation}
  \label{e:upperhlf2}
\|\mathbf v\|_{L^2}\leq Ch^{-1}\|\mathbf g\|_{L^2},\quad
\mathbf v\in C^\infty(\mathcal M;\mathcal E),\quad
\mathbf g=(\mathbf P_{0,r}-i)\mathbf v.
\end{equation}
We now apply Lemma~\ref{l:ultimate}, with $\mathbf P=\mathbf P_{0,r}-i$, $L=\{\langle\xi\rangle^{-1}p=0\}$, $s=0$, and $B=0$.
Note that $\Im\mathbf P\lesssim -h$ on $L^2$ near $L$ by Lemma~\ref{l:gs1}, with $Q:=1$.
By~\eqref{e:ultimate}, we get for some $A,B_1\in\Psi^0$ such that $\{\langle\xi\rangle^{-1}p=0\}\subset\Ell_h(A)$
and $B_1$ is microlocalized in a neighborhood of $\{\langle\xi\rangle^{-1}p=0\}$,
$$
\|A\mathbf v\|_{L^2}\leq Ch^{-1}\|B_1\mathbf g\|_{L^2}+\mathcal O(h^\infty)\|\mathbf v\|_{L^2}.
$$
Combining this with the elliptic estimate~\eqref{e:case-1} valid for $\WFh(A)\cap \{\langle\xi\rangle^{-1}p=0\}=\emptyset$,
we get~\eqref{e:upperhlf2}.
\end{proof}
%%%%%%%%%%%%%%%%%%%%%%%%%%%%%%%%%%%%%%%%%%%%%%%%%%%%%%%%%%%%%%%%%%%%%%%%%%%%%%%%
The operator $\mathbf R_0(z)$ depends on the choice of $Q_\infty,q_1$
(and thus on $h$).
It is independent of the choice of $r$, but proving this would require a separate argument.
However, the restriction of this operator to $\mathcal U$ is independent of $Q_\infty,q_1,r$.
This is a byproduct of the following
%%%%%%%%%%%%%%%%%%%%%%%%%%%%%%%%%%%%%%%%%%%%%%%%%%%%%%%%%%%%%%%%%%%%%%%%%%%%%%%%
\begin{lemm}
  \label{l:samething}
In the notation of Lemma~\ref{l:meromorphic},
let $\lambda\in [-C_1,h^{-1}]+i[-C_2,C_2]$ and put $z:=ih\lambda$. Assume also
that $\Re\lambda>C_0$, where $C_0$ is defined in~\eqref{e:C0-def}. Then
\begin{equation}
  \label{e:samething}
\mathbf R(\lambda)\mathbf f=-ih\mathbf R_0(z)\mathbf f|_{\mathcal U}\quad\text{for all}\quad\mathbf f\in C_0^\infty(\mathcal U;\mathcal E).
\end{equation}
\end{lemm}
%%%%%%%%%%%%%%%%%%%%%%%%%%%%%%%%%%%%%%%%%%%%%%%%%%%%%%%%%%%%%%%%%%%%%%%%%%%%%%%%
\begin{proof}
By analyticity and since $h$ can be chosen arbitrarily small, it suffices to prove~\eqref{e:samething} in the case
$\Re\lambda>C_3$, where $C_3>C_0$ is a large enough constant depending on $r$, but not on $h$.
As discussed after~\eqref{e:anisop}, the anisotropic Sobolev space $\mathcal H^r_h$
contains the standard Sobolev space $H^N_h(\mathcal M;\mathcal E)$, for $N$ large enough
depending on $r$.

We consider an extension of $\mathbf X$ to $\mathcal M$ such that~\eqref{e:X-bundles} holds on $\mathcal M$.
Note that~\eqref{e:X-bundles} holds also for the operator
$\mathbf X+h^{-1}q_1$, since $q_1\in C^\infty(\mathcal M)$ is a multiplication operator. We claim that for some $C_3$
depending on $N$, but not on $h$,
\begin{equation}
  \label{e:mozilla}
\|e^{-t(\mathbf X+h^{-1}q_1)}\mathbf f\|_{H^N_h}\leq C(h)e^{C_3 t},\quad
t\geq 0.
\end{equation}
This follows by writing the transfer operator $e^{-t(\mathbf X+h^{-1}q_1)}$ in the form similar to~\eqref{e:Xpot}
using a local trivialization of the bundle $\mathcal E$, with $V$ now a matrix.
Here we use the fact that each derivative of $\varphi^{-t}$ is bounded exponentially in $t$.
The term $h^{-1}q_1$ does not change the value of $C_3$, as $q_1\geq 0$ everywhere and $t\geq 0$,
see~\eqref{e:Xpot}.

Now, for $\Re\lambda>C_3$ and $z:=ih\lambda$ not a pole of $\mathbf R_0$, consider the function
$$
\mathbf v:=\int_0^\infty e^{-t(\mathbf X+h^{-1}q_1+\lambda)}\mathbf f\,dt\ \in\ H^N_h(\mathcal M;\mathcal E)\ \subset\ \mathcal H^r_h.
$$
Since $\supp\mathbf f\subset\mathcal U$, $\supp q_1\cap \overline{\mathcal U}=\emptyset$,
and $\mathcal U$ is convex, it follows that (see~\eqref{e:Xpot})
\begin{equation}
  \label{e:same1}
\mathbf R(\lambda)\mathbf f=\mathbf v\quad\text{on }\mathcal U.
\end{equation}
We also have $\supp \mathbf v\subset\overline{\Sigma}$, where $\Sigma$ is defined in~\eqref{e:Sigma-def}.
(In fact, \eqref{e:X-support} implies that
$\supp\mathbf v\subset\overline{\Sigma_+}$.) Therefore, $Q_\infty\mathbf v=0$.
It follows that
$$
(\mathbf P_0-z)\mathbf v=-ih\mathbf f,\quad
\mathbf v\in \mathcal D^r_h.
$$
Since $\mathbf P_0-z$ is invertible $\mathcal D^r_h\to\mathcal H^r_h$, we have
\begin{equation}
  \label{e:same2}
-ih\mathbf R_0(z)\mathbf f=\mathbf v.
\end{equation}
Combining~\eqref{e:same1} and~\eqref{e:same2}, we get~\eqref{e:samething}.
\end{proof}
%%%%%%%%%%%%%%%%%%%%%%%%%%%%%%%%%%%%%%%%%%%%%%%%%%%%%%%%%%%%%%%%%%%%%%%%%%%%%%%%

%%%%%%%%%%%%%%%%%%%%%%%%%%%%%%%%%%%%%%%%%%%%%%%%%%%%%%%%%%%%%%%%%%%%%%%%%%%%%%%%
\begin{proof}[Proof of Theorem~\ref{t:mer}]
By Lemmas~\ref{l:meromorphic} and~\ref{l:samething},
the operator $-ih\indic_{\mathcal U}\mathbf R_0(ih\lambda)\indic_{\mathcal U}$ gives the meromorphic
continuation of $\mathbf R(\lambda)$ in the region
$[-C_1,h^{-1}]+i[-C_2,C_2]$ for $h$ small enough. Since $C_1,C_2$ can be chosen arbitrarily
and $h$ can be arbitrarily small, we obtain the continuation to the entire complex plane.
\end{proof}
%%%%%%%%%%%%%%%%%%%%%%%%%%%%%%%%%%%%%%%%%%%%%%%%%%%%%%%%%%%%%%%%%%%%%%%%%%%%%%%%
Note that for all $\lambda\in\mathbb C$,
\begin{equation}
  \label{e:quation}
(\mathbf X+\lambda)\mathbf R(\lambda)=\mathbf R(\lambda)(\mathbf X+\lambda)=1\ :\ C_0^\infty(\mathcal U)\to \mathcal D'(\mathcal U).
\end{equation}
Indeed, by analytic continuation it suffices to consider the case $\Re\lambda>C_0$; in this case,
\eqref{e:quation} follows from~\eqref{e:res-upper}.

The following microlocalization statement is used in the proofs of Theorems~\ref{t:res-states} and~\ref{t:zeta}.
See~\eqref{e:wfprime} and~\eqref{e:schwartz} for the notation used below.
%%%%%%%%%%%%%%%%%%%%%%%%%%%%%%%%%%%%%%%%%%%%%%%%%%%%%%%%%%%%%%%%%%%%%%%%%%%%%%%%
\begin{lemm}
  \label{l:microl}
Let $\lambda_0\in\mathbb C$. Then the expansion~\eqref{e:xpansion} holds
for $\mathbf R_H(\lambda):C_0^\infty(\mathcal U;\mathcal E)\to\mathcal D'(\mathcal U;\mathcal E)$ holomorphic
near $\lambda_0$ and a finite rank operator
$\Pi=\Pi_{\lambda_0}:C_0^\infty(\mathcal U;\mathcal E)\to \mathcal D'(\mathcal U;\mathcal E)$.
Moreover, $\supp K_\Pi\subset \Gamma_+\times\Gamma_-$ and
\begin{equation}
  \label{e:wavefront}
\WF'(\mathbf R_H(\lambda))\ \subset\ \Delta(T^*\mathcal U)\cup\Upsilon_+\cup(E_+^*\times E_-^*),\quad
\WF'(\Pi)\ \subset\ E_+^*\times E_-^*,
\end{equation}
where $\Delta(T^*\mathcal U)$ is the diagonal, $E_\pm^*\subset T^*\overline{\mathcal U}$ are defined in Lemma~\ref{l:extended},
and
$$
\Upsilon_+=\{(e^{tH_p}(y,\eta),y,\eta)\mid t\geq 0,\ p(y,\eta)=0,\ y\in\mathcal U,\ \varphi^t(y)\in\mathcal U\}.
$$
\end{lemm}
%%%%%%%%%%%%%%%%%%%%%%%%%%%%%%%%%%%%%%%%%%%%%%%%%%%%%%%%%%%%%%%%%%%%%%%%%%%%%%%%
\begin{proof}
We argue similarly to the proof of~\cite[Proposition~3.3]{DyZw}. By
Theorem~\ref{t:mer},
\begin{equation}
  \label{e:merex-0}
\mathbf R(\lambda)=\mathbf R_H(\lambda)+\sum_{j=1}^{J(\lambda_0)} {\mathbf A_j\over (\lambda-\lambda_0)^j}
\end{equation}
where $\mathbf R_H(\lambda):C_0^\infty(\mathcal U;\mathcal E)\to \mathcal D'(\mathcal U;\mathcal E)$ is holomorphic near $\lambda_0$
and $\mathbf A_j:C_0^\infty(\mathcal U;\mathcal E)\to\mathcal D'(\mathcal U;\mathcal E)$ are finite rank operators.
Plugging this expansion into~\eqref{e:quation}, we get
\begin{equation}
  \label{e:merex-1}
\mathbf A_{j+1}=-(\mathbf X+\lambda_0)\mathbf A_j,\ 1\leq j<J(\lambda_0);\quad
(\mathbf X+\lambda_0)\mathbf A_{J(\lambda_0)}=0.
\end{equation}
The expansion~\eqref{e:xpansion} follows from here by putting $\Pi:=\mathbf A_1$.

If $\psi_1,\psi_2\in C_0^\infty(\mathcal U)$ satisfy $\supp\psi_1\cap \Gamma_-=\supp\psi_2\cap\Gamma_+=\emptyset$,
then Lemma~\ref{l:outgoing2} shows that $\mathbf R(\lambda)\psi_1,\psi_2\mathbf R(\lambda)$ are holomorphic
for all $\lambda\in\mathbb C$. Therefore, $\Pi\psi_1=\psi_2\Pi=0$; this implies that
$\supp K_\Pi\subset \Gamma_+\times\Gamma_-$.

We finally prove~\eqref{e:wavefront}. We start by writing the following identity
relating the auxiliary resolvents defined by~\eqref{e:keymic} and~\eqref{e:r-0}
(we put $z:=ih\lambda$):
$$
\mathbf R_0(z)=\mathbf R_\delta(z)-i\mathbf R_\delta(z)Q_\delta\mathbf R_\delta(z)
-\mathbf R_\delta(z)Q_\delta \mathbf R_0(z)Q_\delta\mathbf R_\delta(z).
$$
Since $Q_\delta$ is supported inside $\mathcal U$, by~\eqref{e:samething} this gives
\begin{equation}
  \label{e:identi2}
\mathbf R(\lambda)=-ih\indic_{\mathcal U}\big(\mathbf R_\delta(z)-i\mathbf R_\delta(z)Q_\delta \mathbf R_\delta(z)\big)\indic_{\mathcal U}
-\indic_{\mathcal U}\mathbf R_\delta(z)Q_\delta\mathbf R(\lambda)Q_\delta\mathbf R_\delta(z)\indic_{\mathcal U}.
\end{equation}
We analyse each of the terms on the right-hand side separately. By~\eqref{e:WF-absorbed}, we have
\begin{equation}
  \label{e:kicka}
\WFh'(\mathbf R_\delta(z))\cap T^*(\mathcal U\times\mathcal U)\ \subset\ \Delta(T^*\mathcal U)\cup\Upsilon_+.
\end{equation}
By~\eqref{e:wf-mul}, and since $\WFh(Q_\delta)\subset T^*\mathcal U$, we get
$$
\WFh'(\mathbf R_\delta(z)Q_\delta\mathbf R_\delta(z))\cap T^*(\mathcal U\times\mathcal U)\ \subset\
\Delta(T^*\mathcal U)\cup\Upsilon_+. 
$$
To handle the third term in~\eqref{e:identi2}, note that for each family of operators
$\mathbf T(\lambda):C_0^\infty(\mathcal U;\mathcal E)\to \mathcal D'(\mathcal U;\mathcal E)$
which is holomorphic in $\lambda$ and independent of $h$, we have by~\eqref{e:wf-mul}
$$
\begin{gathered}
\WFh'(\mathbf R_\delta(z)Q_\delta \mathbf T(\lambda) Q_\delta\mathbf R_\delta(z))\cap T^*(\mathcal U\times\mathcal U)\ \subset\
\Theta_\delta^+\times\Theta_\delta^-,\\
\Theta_\delta^\pm=T^*\mathcal U\cap \bigcup_{\pm t\geq 0} e^{tH_p}(\WFh(Q_\delta)).
\end{gathered}
$$
Plugging the expansion~\eqref{e:xpansion} into the third term in~\eqref{e:identi2} and using
that the terms in this expansion are $h$-independent and $\mathbf R(\lambda)$
does not depend on $\delta$, we get
$$
\begin{aligned}
\WFh'(\mathbf R_H(\lambda))\cap T^*(\mathcal U\times\mathcal U)\ &\subset\ \Delta(T^*\mathcal U)\cup\Upsilon_+\cup\Big(\bigcap_{\delta}\Theta_\delta^+\times\bigcap_\delta\Theta_\delta^-\Big),\\
\WFh'(\Pi)\cap T^*(\mathcal U\times\mathcal U)\ &\subset\ \bigcap_{\delta}\Theta_\delta^+\times\bigcap_\delta\Theta_\delta^-.
\end{aligned}
$$
Since $\mathbf R_H(\lambda)$ is independent of $h$,
by~\cite[(2.6)]{DyZw} we have
$$
\WF'(\mathbf R_H(\lambda))=\WFh'(\mathbf R_H(\lambda))\cap (T^*(\mathcal U\times\mathcal U)\setminus 0),
$$
and same is true for $\Pi$. To show~\eqref{e:wavefront}, it remains to prove that
$$
\bigcap_\delta \Theta_\delta^\pm\subset E_\pm^*.
$$
Take $(y,\eta)\in\bigcap_\delta\Theta_\delta^\pm$. By taking a sequence of $\delta$ converging to 0,
we see that there exists a sequence $t_j$ such that
$e^{\mp t_jH_p}(y,\eta)\to \{x\in K,\ \xi=0\}$.
If $y\notin\Gamma_\pm$, then the trajectory $\{\varphi^{\mp t}(y)\mid t\geq 0\}$ never passes through
some neighborhood of $K$; therefore, we have $y\in\Gamma_\pm$. Since $p$ is preserved along the trajectories
of $e^{tH_p}$, we have $p(y,\eta)=0$. Finally, if
$\eta\notin E_\pm^*(y)$, then by part~4 of Lemma~\ref{l:extended}
the trajectory $\{e^{\mp tH_p}(y,\eta)\mid t\geq 0\}$ never passes through some neighborhood
of the zero section and we have a contradiction. It follows that $(y,\eta)\in E_\pm^*$ as required.
\end{proof}
%%%%%%%%%%%%%%%%%%%%%%%%%%%%%%%%%%%%%%%%%%%%%%%%%%%%%%%%%%%%%%%%%%%%%%%%%%%%%%%%
For the proof of Theorem~\ref{t:res-states}, we also need
%%%%%%%%%%%%%%%%%%%%%%%%%%%%%%%%%%%%%%%%%%%%%%%%%%%%%%%%%%%%%%%%%%%%%%%%%%%%%%%%
\begin{lemm}
  \label{l:returns}
Assume that $\mathbf u\in \mathcal D'(\mathcal U;\mathcal E)$ satisfies
\begin{equation}
  \label{e:wfcond}
\supp\mathbf u\ \subset\ \Gamma_+,\quad
\WF(\mathbf u)\ \subset\ E_+^*.
\end{equation}
For some $\lambda\in\mathbb C$, put~$\mathbf f:=(\mathbf X+\lambda)\mathbf u$.
Take $\chi,\chi'\in C_0^\infty(\mathcal U)$ satisfying~\eqref{e:convfun} and $\chi=1$ near $K$. Then:

1. For $r>0$ large enough, $\chi'\mathbf u$ and $\chi'\mathbf f$ lie in the space $\mathcal H^r_h$ from~\eqref{e:anisop}.
By Lemmas~\ref{l:meromorphic} and~\ref{l:samething}, this makes it possible to define
$\mathbf R(\lambda)\chi' \mathbf f\in \mathcal D'(\mathcal U;\mathcal E)$.

2. We have $\chi\mathbf u=\chi \mathbf R(\lambda)\chi'\mathbf f$.
\end{lemm}
%%%%%%%%%%%%%%%%%%%%%%%%%%%%%%%%%%%%%%%%%%%%%%%%%%%%%%%%%%%%%%%%%%%%%%%%%%%%%%%%
\begin{proof}
1. Take $r>0$ large enough so that $\chi'\mathbf u\in H^{-r}_h(\mathcal M;\mathcal E)$.
By Lemma~\ref{l:order}, the order function $m$ is equal to $-1$ near $E_+^*\supset\WF(\chi'\mathbf u)$.
Let $G=G(h)$ be the operator defined in~\eqref{e:Goat}. Then $e^{rG(h)}$ is a nonsemiclassical
pseudodifferential operator of order $-r$ microlocally near $\WF(\chi'\mathbf u)$. It follows
that $e^{rG(h)}\chi'\mathbf u\in L^2$ and thus $\chi'\mathbf u\in \mathcal H^r_h$.
Since $\mathbf f$ satisfies~\eqref{e:wfcond} as well, we similarly have
$\chi'\mathbf f\in \mathcal H^r_h$.

2. Since $\chi'\mathbf u\in \mathcal H^r_h$ and $\mathbf X\chi'\mathbf u\in \mathcal H^r_h$
for $r$ large enough, we have by Lemmas~\ref{l:meromorphic} and~\ref{l:samething},
$$
\chi\mathbf u=\chi\chi'\mathbf u=\chi\mathbf R(\lambda)(\mathbf X+\lambda)\chi'\mathbf u.
$$
Now, by~\eqref{e:X-bundles}
$$
(\mathbf X+\lambda)\chi'\mathbf u=(X\chi')\mathbf u+\chi'\mathbf f.
$$
Take $x\in \Gamma_+\cap \supp(X\chi')$. By Lemma~\ref{l:convergence}, there exists
$t'>0$ such that $\varphi^{-t'}(x)\in \supp\chi$. Then by~\eqref{e:convfun},
$\varphi^t(x)\notin\supp\chi$ for all $t\geq 0$. In particular,
$x\notin\Gamma_-$. Since $\supp\mathbf u\subset\Gamma_+$,
there exists $\psi_1\in C_0^\infty(\mathcal M)$ such that
$(X\chi')\mathbf u=\psi_1(X\chi')\mathbf u$ and $\supp\psi_1\cap\Gamma_-=\emptyset$.
Then $\mathbf R(\lambda)\psi_1$ is given by~\eqref{e:holodeck}. It follows from~\eqref{e:X-support} that
$$
\chi\mathbf R(\lambda)(X\chi')\mathbf u=0
$$
and thus $\chi\mathbf u=\chi\mathbf R(\lambda)\chi'\mathbf f$ as needed.
\end{proof}
%%%%%%%%%%%%%%%%%%%%%%%%%%%%%%%%%%%%%%%%%%%%%%%%%%%%%%%%%%%%%%%%%%%%%%%%%%%%%%%%

%%%%%%%%%%%%%%%%%%%%%%%%%%%%%%%%%%%%%%%%%%%%%%%%%%%%%%%%%%%%%%%%%%%%%%%%%%%%%%%%
\begin{proof}[Proof of Theorem~\ref{t:res-states}]
The expansion~\eqref{e:xpansion} and the properties~\eqref{e:piprop1} have already been
established in Lemma~\ref{l:microl}. Therefore, it remains to prove~\eqref{e:piprop2}.
The property~$\mathbf X\Pi=\Pi\mathbf X$ follows from~\eqref{e:xpansion} and~\eqref{e:quation}.
By~\eqref{e:piprop1} and~\eqref{e:merex-1}, we know that
$$
\Ran\Pi\ \subset\ \Res^{(J(\lambda_0))}_{\mathbf X}(\lambda_0),
$$
therefore it remains to prove that for each $N$,
\begin{equation}
  \label{e:repeats}
\mathbf u\in \Res^{(N)}_{\mathbf X}(\lambda_0)\ \Longrightarrow\
\mathbf u=\Pi\mathbf u.
\end{equation}
Take $\chi\in C_0^\infty(\mathcal U)$ such that $\chi=1$ near $K$ and let $\chi'\in C_0^\infty(\mathcal U)$
be constructed in Lemma~\ref{l:convfun}. We claim that for each $j=0,\dots,N$,
\begin{equation}
  \label{e:repeats1}
\chi\mathbf R(\lambda)\chi'(\mathbf X+\lambda_0)^j\mathbf u=\sum_{k=j}^{N-1}{(-1)^{k-j}\chi (\mathbf X+\lambda_0)^k\mathbf u\over (\lambda-\lambda_0)^{k-j+1}}.
\end{equation}
We argue by induction on $j=N,\dots, 0$. For $j=N$, we have $(\mathbf X+\lambda_0)^j\mathbf u=0$ and~\eqref{e:repeats1} is trivial.
Now, assume that~\eqref{e:repeats1} is true for $j+1$. Using the identity
$$
(\mathbf X+\lambda_0)^j \mathbf u={(\mathbf X+\lambda)(\mathbf X+\lambda_0)^j\mathbf u-(\mathbf X+\lambda_0)^{j+1}\mathbf u\over \lambda-\lambda_0}
$$
and Lemma~\ref{l:returns} for the first term on the right-hand side, we obtain~\eqref{e:repeats1} for $j$, finishing
its proof.

Now, take $j=0$ in~\eqref{e:repeats1} and use~\eqref{e:xpansion}. Equating the terms next
to $(\lambda-\lambda_0)^{-1}$, we obtain $\chi\mathbf u=\chi\Pi\chi'\mathbf u$.
Moreover, $\Pi\chi'\mathbf u=\Pi\mathbf u$ since $\supp K_\Pi\subset\Gamma_+\times\Gamma_-$,
$\supp\mathbf u\subset\Gamma_+$, and $\chi'=1$ near $K$.
Since $\chi$ could be chosen arbitrarily, this gives~\eqref{e:repeats}.
\end{proof}
%%%%%%%%%%%%%%%%%%%%%%%%%%%%%%%%%%%%%%%%%%%%%%%%%%%%%%%%%%%%%%%%%%%%%%%%%%%%%%%%

%%%%%%%%%%%%%%%%%%%%%%%%%%%%%%%%%%%%%%%%%%%%%%%%%%%%%%%%%%%%%%%%%%%%%%%%%%%%%%%%
%%%%%%%%%%%%%%%%%%%%%%%%%%%%%%%%%%%%%%%%%%%%%%%%%%%%%%%%%%%%%%%%%%%%%%%%%%%%%%%%
\section{Dynamical traces and zeta functions}
  \label{s:trace-zeta}

In this section we prove Theorem~\ref{t:zeta}. More generally, we prove
in Theorem~\ref{t:trace} below that the dynamical trace $F_{\mathbf X}(\lambda)$
associated to $\mathbf X$ is equal to the flat trace of a certain operator
featuring the resolvent $\mathbf R(\lambda)$; this flat
trace gives the meromorphic extension of $F_{\mathbf X}(\lambda)$. The key
ingredient of the proof is the wavefront set condition~\eqref{e:wavefront}
on the meromorphic extension of the resolvent. We follow the strategy of~\cite{DyZw}
and refer the reader to that paper for the parts of the proof that remain unchanged
in our more general case.

%%%%%%%%%%%%%%%%%%%%%%%%%%%%%%%%%%%%%%%%%%%%%%%%%%%%%%%%%%%%%%%%%%%%%%%%%%%%%%%%
\subsection{Meromorphic extension of traces}  
  \label{s:trace}
  
We first show how to express Pollicott--Ruelle resonances of $\mathbf X$
as the poles of a certain trace expression featuring closed geodesics.
To write down this expression, we need to introduce some notation.
Define the vector bundle $\mathcal E_0$ over $\overline{\mathcal U}$ by
\begin{equation}
  \label{e:E-0}
\mathcal E_0(x)=\{\eta\in T^*_x\mathcal M\mid \langle X(x),\eta\rangle=0\},\quad
x\in\overline{\mathcal U}.
\end{equation}
Assume that $x,\varphi^t(x)\in \mathcal U$ for some $t$. Define the \emph{linearized Poincar\'e map}
$$
\mathcal P_{x,t}:\mathcal E_0(x)\to\mathcal E_0(\varphi^t(x)),\quad
\mathcal P_{x,t}=(d\varphi^t(x))^{-T}|_{\mathcal E_0(x)}.
$$
Here $(d\varphi^t(x))^{-T}$ is the inverse transpose of $d\varphi^t(x)$
as in~\eqref{e:hammertime}.
Next, the parallel transport
$$
\alpha_{x,t}:\mathcal E(x)\to \mathcal E(\varphi^t(x))
$$
is defined as follows: for each $\mathbf u\in C^\infty(\mathcal M;\mathcal E)$,
we put $\alpha_{x,t}(\mathbf u(x))=e^{-t\mathbf X}\mathbf u(\varphi^t(x))$.
This definition only depends on the value of $\mathbf u$ at $x$;
indeed, \eqref{e:X-support} shows that
if $\mathbf u(x)=0$, then $e^{-t\mathbf X}\mathbf u(\varphi^t(x))=0$ as well
(by writing $\mathbf u$ as a sum of expressions of the form $f\mathbf v$,
where $f\in C^\infty(\mathcal M)$ vanish at $x$). 

Now, assume that $\gamma(t)=\varphi^t(x_0)$ is a closed trajectory, that is
$\gamma(T)=\gamma(0)$ for some $T>0$. (We call $T$ the period of $\gamma$, and regard the same $\gamma$
with two different values of $T$ as two different closed trajectories. The minimal
positive $T^\sharp$ such that $\gamma(T^\sharp)=\gamma(0)$ is called the \emph{primitive period}.) Assume also that $x_0\in \mathcal U$; this implies
immediately that $\gamma$ lies inside $K$.
The operators $\alpha_{\gamma(t),T}:\mathcal E(\gamma(t))\to \mathcal E(\gamma(t))$,
as well as $\mathcal P_{\gamma(t),T}:\mathcal E_0(\gamma(t))\to\mathcal E_0(\gamma(t))$,
are conjugate to each other for different $t$, therefore the trace
and the determinant
\begin{equation}
  \label{e:fancy}
\tr\alpha_{\gamma}:=\tr\alpha_{\gamma(t),T},\quad
\det(I-\mathcal P_{\gamma}):=\det(I-\mathcal P_{\gamma(t),T})
\end{equation}
do not depend on $t$. Note that by~\eqref{e:hyperdual},
\begin{equation}
  \label{e:nondeg}
\det(I-\mathcal P_{\gamma})\neq 0.
\end{equation}

The main result of this subsection, and the key ingredient for showing meromorphic continuation of dynamical zeta
functions, is
%%%%%%%%%%%%%%%%%%%%%%%%%%%%%%%%%%%%%%%%%%%%%%%%%%%%%%%%%%%%%%%%%%%%%%%%%%%%%%%%
\begin{theo}
  \label{t:trace}
Define for $\Re\lambda\gg 1$,
\begin{equation}
  \label{e:tracesum}
F_{\mathbf X}(\lambda):=\sum_{\gamma} {e^{-\lambda T_\gamma}\,T_\gamma^\sharp\,\tr\alpha_{\gamma} \over |\det(I-\mathcal P_{\gamma})|}
\end{equation}
where the sum is over all closed trajectories $\gamma$ inside $K$, $T_\gamma>0$
is the period of $\gamma$, and $T_\gamma^\sharp$ is the primitive period.
Then $F(\lambda)$ extends meromorphically to $\lambda\in\mathbb C$.
The poles of $F(\lambda)$ are the Pollicott--Ruelle resonances of $\mathbf X$ and
the residue at a pole $\lambda_0$ is equal to the
rank of $\Pi_{\lambda_0}$ (see Theorem~\ref{t:res-states}).
\end{theo}
%%%%%%%%%%%%%%%%%%%%%%%%%%%%%%%%%%%%%%%%%%%%%%%%%%%%%%%%%%%%%%%%%%%%%%%%%%%%%%%%
\noindent\textbf{Remark}.
The sum~\eqref{e:tracesum} converges for large $\Re\lambda$, since
$|\det(I-\mathcal P_\gamma)|$ is bounded away from zero,
$\alpha_\gamma$ grows at most exponentially in $T_\gamma$,
and the number of closed trajectories grows at most exponentially by Lemma~\ref{l:recur3}.
%%%%%%%%%%%%%%%%%%%%%%%%%%%%%%%%%%%%%%%%%%%%%%%%%%%%%%%%%%%%%%%%%%%%%%%%%%%%%%%%
\begin{proof}
We use the concept of the flat trace of an operator
$\mathbf A:C^\infty(\mathcal M;\mathcal U)\to \mathcal D'(\mathcal M;\mathcal U)$
satisfying the condition
\begin{equation}
  \label{e:trcond}
\WF'(\mathbf A)\cap \Delta(T^*\mathcal M\setminus 0)=\emptyset,\quad
\Delta(T^*\mathcal M\setminus 0)=\{(x,\xi,x,\xi)\mid (x,\xi)\in T^*\mathcal M\setminus 0\}.
\end{equation}
The flat trace is defined as the integral of the restriction
of the Schwartz kernel $K_{\mathbf A}\in \mathcal D'(\mathcal M\times\mathcal M;\End(\mathcal E))$ to the diagonal:
$$
\tr^\flat \mathbf A:=\int_{\mathcal M} \tr_{\End(\mathcal E)} K_{\mathbf A}(x,x)\,dx
$$
and $\tr_{\End(\mathcal E)}K_{\mathbf A}(x,x)$ is a well-defined distribution on $\mathcal M$
due to~\eqref{e:trcond}~-- see~\cite[\S2.4]{DyZw}.

The starting point of the proof is the Atiyah--Bott--Guillemin trace formula~\cite{guillemin}
\begin{equation}
  \label{e:guillemin}
\tr^\flat \int_0^\infty \varphi(t) \chi e^{-t\mathbf X}\chi \,dt=\sum_{\gamma}{\varphi(T_\gamma)T_\gamma^\sharp\,\tr\alpha_{\gamma}
\over |\det(I-\mathcal P_{\gamma})|},\quad
\varphi\in C_0^\infty(0,\infty),
\end{equation}
where $\chi\in C_0^\infty(\mathcal U)$ is any function such that $\chi=1$ near $K$.
Note that the Schwartz kernel
$\chi(x)K_{e^{-t\mathbf X}}(x,y)\chi(y)$ is a smooth function times the delta
function of the submanifold $\{y=\varphi^{-t}(x)\}\subset \mathbb R_t\times\mathcal M_x\times\mathcal M_y$, so the wavefront set
of this kernel is contained in the conormal bundle to this surface~\cite[Example~8.2.5]{ho1}
$$
\{y=\varphi^{-t}(x),\ \xi=-(d\varphi^{-t}(x))^T\cdot\eta,\ \tau=\langle X(y),\eta\rangle,\ \eta\neq 0,\ x,y\in\supp\chi\}.
$$
Here $\tau$ is the momentum dual to $t$.
If $\mathbf A_\varphi$ is the integral on the left-hand side of~\eqref{e:guillemin},
then its Schwartz kernel is the pushforward of $\varphi(t)\chi(x)K_{e^{-t\mathbf X}}(x,y)\chi(y)$
under the map $(t,x,y)\mapsto (x,y)$, therefore~\cite[Example~8.2.5 and Theorem~8.2.13]{ho1}
$$
\WF'(\mathbf A_\varphi)\subset\{x=\varphi^t(y),\ \xi=\mathcal P_{y,t}\cdot\eta,\ \eta\in\mathcal E_0(y),\
t\in\supp\varphi,\
x,y\in\supp\chi\}.
$$
By~\eqref{e:nondeg}, $\mathbf A_\varphi$ satisfies~\eqref{e:trcond} and thus the left-hand side
of~\eqref{e:guillemin} is well-defined. See~\cite[Appendix~B]{DyZw} for a detailed
proof of~\eqref{e:guillemin}, which generalizes directly to our situation.
Note that the Poincar\'e map defined in~\cite[(B.1)]{DyZw} is the transpose of the one
used in this paper, which does not change the determinant~\eqref{e:nondeg}.

As in~\cite[\S4]{DyZw}, using~\eqref{e:guillemin}, Lemma~\ref{l:recur2}, and the fact that
the right-hand side is well-defined by the wavefront set condition (see below), we get for some $C_1>0$,
\begin{equation}
  \label{e:same-trace}
F_{\mathbf X}(\lambda)=\tr^\flat\big(\chi e^{-t_0(\mathbf X+\lambda)}\mathbf R(\lambda)\chi \big),\quad
\Re\lambda>C_1
\end{equation}
where $t_0>0$ is small enough so that $t_0 <T_\gamma$ for all $\gamma$.
We also make $t_0$ small enough so that $\varphi^{-t_0}(\supp\chi)\subset\mathcal U$;
then $\chi e^{-t_0(\mathbf X+\lambda)}\mathbf R(\lambda)\chi$ is a well-defined compactly supported
operator on $\mathcal U$.

Note that by~\eqref{e:X-support} and by~\cite[Example~8.2.5]{ho1},
the wavefront set of the operator $e^{-t_0\mathbf X}$ is contained in the graph
of $e^{t_0H_p}$. Then by~\eqref{e:wavefront} and multiplicativity of wavefront sets~\cite[Theorem~8.2.14]{ho1},
we have for each $\lambda\in\mathbb C$ which is not a resonance,
$$
\begin{gathered}
\WF'(\chi e^{-t_0(\mathbf X+\lambda)}\mathbf R(\lambda)\chi)\ \subset\
\{(e^{t_0H_p}(y,\eta),y,\eta)\mid (y,\eta)\in T^*\mathcal U\setminus 0\}\ \cup\ (E_+^*\times E_-^*)\\
\cup\ \{(e^{tH_p}(y,\eta),y,\eta)\mid (y,\eta)\in T^*\mathcal U\setminus 0,\ \eta\in\mathcal E_0(y),\ t\geq t_0\},
\end{gathered}
$$
and a similar statement is true for the regular and the singular parts of this operator
when $\lambda$ is a resonance~-- see Lemma~\ref{l:microl}.
It follows from~\eqref{e:hammertime} and~\eqref{e:nondeg} that the operator $\chi e^{-t_0(\mathbf X+\lambda)}\mathbf R(\lambda)\chi$
satisfies~\eqref{e:trcond}; therefore, the right-hand side of~\eqref{e:same-trace}
is defined as a meromorphic function of $\lambda\in\mathbb C$, and its poles are the resonances of $\mathbf X$.

It remains to show that for each resonance $\lambda_0$, the meromorphic continuation
of $F_{\mathbf X}(\lambda)$ has a simple pole at $\lambda_0$ with residue equal to the rank of $\Pi_{\lambda_0}$.
By~\eqref{e:same-trace} and recalling the expansion~\eqref{e:xpansion}, it suffices to show that
$$
\tr^\flat\sum_{j=1}^{J(\lambda_0)}(-1)^{j-1}{\chi e^{-t_0(\mathbf X+\lambda)}(\mathbf X+\lambda_0)^{j-1}\Pi_{\lambda_0}\chi\over (\lambda-\lambda_0)^j}
={\rank\Pi_{\lambda_0}\over \lambda-\lambda_0}+\Hol(\lambda)
$$
where $\Hol(\lambda)$ stands for a function which is holomorphic near $\lambda_0$.
Expanding $e^{-t_0(\mathbf X+\lambda)}$ at $\lambda=\lambda_0$, we see that it
is enough to prove that
\begin{equation}
  \label{e:nilpotent-traces}
\begin{aligned}
\tr^\flat (\chi e^{-t_0(\mathbf X+\lambda_0)}\Pi_{\lambda_0}\chi)&=\rank\Pi_{\lambda_0};\\
\tr^\flat (\chi e^{-t_0(\mathbf X+\lambda_0)}(\mathbf X+\lambda_0)^j\Pi_{\lambda_0}\chi)&=0,\quad
j\geq 1.
\end{aligned}
\end{equation}
By~\eqref{e:piprop1},
each operator on the left-hand side can be written as a finite sum $\sum_\ell \mathbf u_\ell\otimes \mathbf v_\ell$,
where $\otimes$ denotes the Hilbert tensor product and
$$
\begin{aligned}
\mathbf u_\ell\in\mathcal D'(\mathcal U;\mathcal E),&\quad
\supp\mathbf u_\ell\subset \Gamma_+,\quad
\WF(\mathbf u_\ell)\subset E_+^*;\\
\mathbf v_\ell\in\mathcal D'(\mathcal U;\mathcal E^*\otimes |\Omega|^1),&\quad
\supp \mathbf v_\ell\subset\Gamma_-,\quad
\WF(\mathbf v_\ell)\subset E_-^*.
\end{aligned}
$$
Since $\Gamma_+\cap\Gamma_-=K$ is compactly contained in $\mathcal U$ and
$E_+^*\cap E_-^*\cap (T^*\mathcal U\setminus 0)=\emptyset$, by~\cite[Theorem~8.2.13]{ho1}
we can define the inner product $\langle \mathbf u_\ell,\mathbf v_{\ell'}\rangle$ for each $\ell,\ell'$.
This implies that the operators of the form $\sum_\ell \mathbf u_\ell\otimes \mathbf v_\ell$ can be multiplied
(and form an algebra), they satisfy~\eqref{e:trcond},
and their flat traces are given by $\sum_\ell\langle \mathbf u_\ell,\mathbf v_\ell\rangle$.

Recall the spaces $\Res^{(k)}:=\Res^{(k)}_{\mathbf X}(\lambda_0)$ defined in~\eqref{e:res-spaces}.
By Theorem~\ref{t:res-states}, and since $\chi=1$ near $K$, the operators
\begin{equation}
  \label{e:nilpotent-operators}
\chi(e^{-t_0(\mathbf X+\lambda_0)}\Pi_{\lambda_0}-\Pi_{\lambda_0})\chi,\quad
\chi e^{-t_0(\mathbf X+\lambda_0)}(\mathbf X+\lambda_0)^j\Pi_{\lambda_0}\chi,\quad
j\geq 1
\end{equation}
are nilpotent; more precisely, they map for $1<k\leq J(\lambda_0)$,
\begin{equation}
  \label{e:nilpotency}
C_0^\infty(\mathcal U)\to\chi\Res^{(J(\lambda_0))},\quad
\chi\Res^{(k)}\to \chi\Res^{(k-1)},\quad
\chi\Res^{(1)}\to 0.
\end{equation}
(The propagation operator $e^{-t_0(\mathbf X+\lambda_0)}$ does not cause any trouble
since $\Gamma_+\cap \varphi^{-t_0}(\mathcal U)\subset\mathcal U$ by~\eqref{e:convex2}
and thus each element of $e^{-t_0(\mathbf X+\lambda)_0}\Res^{(k)}$ can be restricted to $\mathcal U$
to yield another element of $\Res^{(k)}$.) It follows from~\eqref{e:nilpotency} that the operators in~\eqref{e:nilpotent-operators}
have zero flat trace.
Since $\Pi_{\lambda_0}^2=\Pi_{\lambda_0}$ and $\chi=1$ near $K$, we have
$\tr^\flat (\chi\Pi_{\lambda_0}\chi)=\rank \Pi_{\lambda_0}$; \eqref{e:nilpotent-traces} follows.
\end{proof}
%%%%%%%%%%%%%%%%%%%%%%%%%%%%%%%%%%%%%%%%%%%%%%%%%%%%%%%%%%%%%%%%%%%%%%%%%%%%%%%%

%%%%%%%%%%%%%%%%%%%%%%%%%%%%%%%%%%%%%%%%%%%%%%%%%%%%%%%%%%%%%%%%%%%%%%%%%%%%%%%%
\subsection{Meromorphic extension of zeta functions}
  \label{s:zeta}
In this section, we prove Theorem~\ref{t:zeta}. Using the Taylor series
of $\log(1-x)$, we get for $\Re\lambda\gg 1$,
$$
\log\zeta_V(\lambda)=-\sum_{\gamma^\sharp}\sum_{k=1}^\infty {\exp(-kT_{\gamma^\sharp}(\lambda+V_{\gamma^\sharp}))\over k}
=-\sum_\gamma{T_\gamma^\sharp \exp(-T_\gamma(\lambda+V_\gamma))\over T_{\gamma}}
$$
where the last sum is over all closed trajectories $\gamma$ of $\varphi^t$, with periods $T_\gamma>0$
and primitive periods $T_\gamma^\sharp$,
and $V_\gamma$ is defined in~\eqref{e:V-gamma}. It follows that for $\Re\lambda\gg 1$,
$$
{\zeta_V'(\lambda)\over\zeta_V(\lambda)}
=\sum_\gamma T_\gamma^\sharp e^{-\lambda T_\gamma}e^{-T_\gamma V_\gamma}.
$$
To reduce the right-hand side to an expression that can be handled by Theorem~\ref{t:trace},
we make the assumption that, for the Poincar\'e determinants defined in~\eqref{e:fancy},
\begin{equation}
  \label{e:orientability}
|\det(I-\mathcal P_\gamma)|=(-1)^\beta\det(I-\mathcal P_\gamma)\quad\text{with $\beta$ independent of $\gamma$}.
\end{equation}
This condition holds when $E_s$ is orientable, with $\beta=\dim E_s$, see~\cite[\S2.2]{DyZw}.
See~\cite[Appendix~B]{glp} for methods which can be used to eliminate the orientability assumption.

Similarly to~\eqref{e:fancy}, for each $\ell=0,\dots,n-1$, with $n=\dim\mathcal U$, the trace
$$
\tr\wedge^\ell \mathcal P_\gamma
=\tr\wedge^\ell \mathcal P_{\gamma(t),T_\gamma},\quad
\wedge^\ell \mathcal P_{\gamma(t),T_\gamma}:
\wedge^\ell \mathcal E_0(\gamma(t))\to\wedge^\ell\mathcal E_0(\gamma(t))
$$
does not depend on $t$; here $\wedge^\ell$ denotes $\ell$th antisymmetric power.
Using the identity $\det(I-\mathcal P_\gamma)=\sum_{\ell=0}^{n-1}(-1)^\ell \tr\wedge^\ell \mathcal P_\gamma$,
we get for $\Re\lambda\gg 1$,
$$
{\zeta_V'(\lambda)\over\zeta_V(\lambda)}=\sum_{\ell=0}^{n-1}(-1)^{\ell+\beta}F_\ell(\lambda),\quad
F_\ell(\lambda)=
\sum_\gamma {T_\gamma^\sharp e^{-\lambda T_\gamma}e^{-T_\gamma V_\gamma}
\tr\wedge^\ell \mathcal P_\gamma\over |\det(I-\mathcal P_\gamma)|}.
$$
To show that $\zeta_V(\lambda)$ continues meromorphically to $\lambda\in\mathbb C$,
it is enough to show that for each $\ell$, the function $F_\ell(\lambda)$
continues meromorphically to $\lambda\in\mathbb C$ with simple poles and integer residues. This follows from Theorem~\ref{t:trace},
applied to the operator
$$
\mathbf X_\ell:=\mathcal L_X+V:C^\infty(\mathcal U;\wedge^\ell \mathcal E_0)\to C^\infty(\mathcal U;\wedge^\ell\mathcal E_0),
$$
where $\mathcal E_0$ is defined in~\eqref{e:E-0}, $\wedge^\ell\mathcal E_0$ is embedded into the bundle
$\Omega^{\ell}$ of differential $\ell$-forms on $\mathcal U$ as the kernel of the interior product operator
$\iota_X$, and $\mathcal L_X$ is the Lie derivative along $X$ on $\Omega^\ell$, restricted to $\wedge^\ell \mathcal E_0$.

%%%%%%%%%%%%%%%%%%%%%%%%%%%%%%%%%%%%%%%%%%%%%%%%%%%%%%%%%%%%%%%%%%%%%%%%%%%%%%%%
%%%%%%%%%%%%%%%%%%%%%%%%%%%%%%%%%%%%%%%%%%%%%%%%%%%%%%%%%%%%%%%%%%%%%%%%%%%%%%%%
\section{Examples}
  \label{s:examples}

%%%%%%%%%%%%%%%%%%%%%%%%%%%%%%%%%%%%%%%%%%%%%%%%%%%%%%%%%%%%%%%%%%%%%%%%%%%%%%%%
\subsection{A basic example}
  \label{s:examples-basic}

We start with the following basic example:
$$
\mathcal U=\{x_1^2+x_2^2<1\}\times\mathbb S^1_{x_3}\subset\mathbb R^2\times\mathbb S^1,\quad
\mathcal E=\mathbb C,\quad
X=\mathbf X=x_1\partial_{x_1}-x_2\partial_{x_2}+\partial_{x_3}.
$$
It is straightforward to verify that assumptions~\hyperlink{AA1}{\rm(A1)}--\hyperlink{AA5}{\rm(A5)}
from the introduction are satisfied, with
$$
\begin{gathered}
\varphi^t(x_1,x_2,x_3)=(e^t x_1,e^{-t}x_2,x_3)\quad\text{if }e^{2t}x_1^2+e^{-2t}x_2^2\leq 1;\\
\Gamma_+=[-1,1]_{x_1}\times\{0\}_{x_2}\times \mathbb S^1_{x_3},\
\Gamma_-=\{0\}_{x_1}\times [-1,1]_{x_2}\times \mathbb S^1_{x_3},\
K=\{(0,0)\}_{x_1}\times\mathbb S^1_{x_3};\\
E_u(0,0,x_3)=\mathbb R\partial_{x_1},\quad
E_s(0,0,x_3)=\mathbb R\partial_{x_2},
\end{gathered}
$$
and the extended dual stable/unstable bundles from Lemma~\ref{l:extended} given by
$$
E_+^*(x_1, 0, x_3) = \mathbb R dx_2,\quad
E_-^*(0, x_2, x_3) = \mathbb R dx_1.
$$
Then $u\in \mathcal D'(\mathcal U)$ satisfies
$\supp u\subset\Gamma_+$ and $\WF(u)\subset E_+^*$ if and only if
it has the form
$$
u=\sum_{j=0}^N u_{\ell}(x_1,x_3)\partial_{x_2}^\ell \delta(x_2),\quad
u_\ell\in C^\infty((-1,1)_{x_1}\times\mathbb S^1_{x_3}),
$$
here the fact that $u_\ell$ are smooth follows from the wavefront set condition.
A direct calculation shows that the space of resonant states~$\Res^{(1)}_{X}(\lambda)$
defined in~\eqref{e:res-spaces} is nontrivial if and only if
$$
\lambda=\lambda_{\ell,k}=-1-\ell+ik,\quad
\ell\in\mathbb N_0,\quad
k\in\mathbb Z
$$
and the spaces $\Res^{(j)}_{X}(\lambda_{\ell,k})$
are the same for all $j\geq 1$ and spanned by
$$
x_1^m \partial_{x_2}^{\ell-m} \delta(x_2) e^{-ikx_3},\quad
m=0,\dots,\ell.
$$
By Theorem~\ref{t:res-states}, the resonances of $X$ are exactly
$\lambda_{\ell,k}$, with $\rank \Pi_{\lambda_{\ell,k}}=\ell+1$.
Another way to see the same fact is to apply Theorem~\ref{t:trace}
from~\S\ref{s:trace}, with
$$
F_X(\lambda)=\pi\sum_{m\in\mathbb N}{e^{-2\pi m\lambda}\over \cosh(2\pi m)-1};
$$
we use that $F_X(\lambda)$ is holomorphic in $\{\Re\lambda>-1\}$ and satisfies the functional equation
$$
F_X(\lambda+1)+F_X(\lambda-1)-2F_X(\lambda)=2\pi\sum_{m\in\mathbb N}e^{-2\pi m\lambda}
={2\pi\over e^{2\pi\lambda}-1}.
$$
We remark that the assumptions from the introduction are also satisfied
for the vector fields $\pm (x_1\partial_{x_1}+x_2\partial_{x_2})+\partial_{x_3}$;
we leave the details to the reader.

A more general family of examples is given by suspensions of Axiom A maps
(such as Anosov maps or Smale horseshoes). For suspensions of Anosov maps
Pollicott--Ruelle resonances of the flow are determined from
the resonances of the map, see~\cite[Appendix~B]{long}.

%%%%%%%%%%%%%%%%%%%%%%%%%%%%%%%%%%%%%%%%%%%%%%%%%%%%%%%%%%%%%%%%%%%%%%%%%%%%%%%%
\subsection{Riemannian manifolds with boundary}
  \label{s:boundary}

Consider a smooth $m$-dimensional compact Riemannian manifold $(M,g)$ with strictly convex boundary; that is,
the second fundamental form at $\pl M$ with respect to the inward pointing normal
is positive definite. Let $\bbar{\mc{U}}=SM$ be its unit tangent bundle, and consider the vector field $X$ generating the geodesic flow on $SM$.
One can equivalently consider the geodesic flow on the unit cotangent bundle $S^*M$, which is naturally a contact flow.

The vector field $X$ satisfies assumptions~\hyperlink{AA1}{\rm(A1)}--\hyperlink{AA3}{\rm(A3)} in the introduction.
To see~\eqref{e:convex}, choose a coordinate system $x$ on $M$ such that
$M$ locally has the form $\{x_1\geq 0\}$. Let $x(t)\in M$ be a geodesic (on an extension
of $M$ past the boundary) such that $\dot x_1(t)=0$.
By the geodesic equation
$$
\ddot x_1(t)+\sum_{i,j}\Gamma^1_{ij}(x(t))\dot x_i(t)\dot x_j(t)=0
$$
and since the matrix $\Gamma^1_{ij}(x(0))_{i,j=2}^m$ is positive definite by the strict convexity of the boundary,
we get $\ddot x_1(t)<0$ as required.

We assume that the flow $\varphi^t=e^{tX}$ is hyperbolic on the trapped set $K\subset \mathcal U$
in the sense of assumption~\hyperlink{AA4}{\rm(A4)} in the introduction. This is in particular true
if $g$ has negative sectional curvature in a neighborhood of $K$, see for instance~\cite[\S3.9 and Theorem~3.2.17]{kl}.

We now discuss an application of the results of this paper to boundary problems for the geodesic flow.
For each $(x,v)\in\mc{U}$, define 
\[\ell_\pm(x,v)=\pm \sup \{t>0\mid \varphi^{\pm t}(x,v)\in \mc{U}\} \in [-\infty,\infty]\]
as the time of escape to $\pl M$ in forward ($+$) and backward ($-$) time. Note that
$$
\Gamma_\pm\cap\mc{U}=\{(x,v)\mid \ell_\mp(x,v)=\mp \infty\},\quad
K=\{(x,v)\mid \ell_+(x,v)=+\infty,\ \ell_-(x,v)=-\infty\}.
$$ 
Define the incoming ($-$), outgoing ($+$), and tangent ($0$) boundary by
$$
\pl_\pm SM:=\{(x,v)\in \pl SM\mid \mp \cjg d\rho,v\cjd>0\},\quad
\pl_0SM:=\{(x,v)\in \pl SM\mid \cjg d\rho,v\cjd=0\},
$$
where $\rho$ is a defining function of the boundary in $M$.
The Liouville measure $d\mu$ on $SM$ is invariant by the flow and it is straightforward to check that the boundary value problem
\begin{equation}
  \label{BVP}
(-X\pm\lambda)u_\pm=f\in L^2_{\comp}(\mathcal U),\quad u_\pm=0\quad\text{near }\pl_\pm SM,\quad
u_\pm\in L^2(SM),
\end{equation}
is uniquely solvable for $\Re\lambda>0$, and the solution is given by 
\begin{equation}
  \label{ump}
u_\pm(x,v)=\int_0^{\ell_\pm(x,v)}e^{-\la |t|}f(\varphi^t(x,v))dt.
\end{equation}
This defines a bounded map $R_\pm(\la): L^2(\mathcal U)\to L^2(\mathcal U)$ by putting $R_\pm(\la)f:=u_\pm$.
%%%%%%%%%%%%%%%%%%%%%%%%%%%%%%%%%%%%%%%%%%%%%%%%%%%%%%%%%%%%%%%%%%%%%%%%%%%%%%%%
\begin{prop}\label{prop:bvp}
Assume that $(M,g)$ is a compact Riemannian manifold with strictly convex boundary
and hyperbolic trapped set, and $\varphi^t:\mathcal U\to\mathcal U$,
$\overline{\mathcal U}=SM$, be the geodesic flow.
Let $E_\mp^*\subset T_{\Gamma_\mp}^*\overline{\mc{U}}$ be defined by Lemma~\ref{l:extended}. 
Then:

1. The operators $R_\pm(\la)$ have meromorphic continuation to 
$\la\in\cc$ as operators $R_\pm(\la): C_0^\infty(\mc{U})\to \mc{D}'(\mc{U})$, with poles of finite rank.

2. Assume that $\lambda\in\mathbb C$ is not a pole of $R_\pm$. Then for each $f\in C_0^\infty(\mathcal U)$,
$u_\pm=R_\pm(\lambda)f$ is the unique solution in $\mathcal D'(\mathcal U)$ to the problem
\begin{equation}
  \label{e:prob}
(-X\pm\lambda)u_\pm=f ,\quad u_\pm=0 \text{ near } \pl_\pm SM\cup \pl_0SM, \quad \WF(u_\pm)\subset E_\mp^*.
\end{equation}
Moreover, $R_\pm(\lambda)$ acts $H^s_0(\mathcal U)\to H^{-s}(\mathcal U)$ for all
$s>\gamma^{-1}\max(0,-\Re\lambda)$, where $\gamma>0$ is the constant in~\eqref{e:hyper}.
Finally, there exist conic neighborhoods $U_\pm$ of $E_\pm^*$ such that for
each compactly supported $A_\pm\in\Psi^0(\mathcal U)$ with $\WF'(A_\pm)\subset U_\pm$,
the operators $A_\pm R_\pm(\lambda)$ act $H^s_0(\mathcal U)\to H^s(\mathcal U)$.

3. Assume that $\lambda\in\mathbb C$ is a pole of $R_\pm$. Then there exists a nonzero solution
$u_\pm\in \mathcal D'(\mathcal U)$ to the problem~\eqref{e:prob} with $f=0$; in fact, $\supp u_\pm\subset \Gamma_\mp$.
\end{prop}
%%%%%%%%%%%%%%%%%%%%%%%%%%%%%%%%%%%%%%%%%%%%%%%%%%%%%%%%%%%%%%%%%%%%%%%%%%%%%%%%
\begin{proof}
We establish the properties of $R_-(\lambda)$; the properties of $R_+(\lambda)$ are obtained by flipping the sign of $X$.

1. Put $\mathcal E:=\mathbb C$, $\mathbf X:=X$ in assumption~\hyperlink{AA5}{\rm(A5)} in the introduction. Comparing~\eqref{ump}
with~\eqref{e:res-upper}, we see that $R_-(\lambda)f=-\mathbf R(\lambda)f$ for all $f\in C_0^\infty(\mathcal U)$
and $\Re\lambda>0$. It remains to apply Theorem~\ref{t:mer}.

2. The fact that $R_-(\lambda)f$ is a solution to~\eqref{e:prob} follows by analytic continuation
from~\eqref{BVP}; to see that $\WF(R_-(\lambda)f)\subset E_+^*$, we use~\eqref{e:wavefront}.
To see uniqueness, assume that $u_-\in\mathcal D'(\mathcal U)$ solves~\eqref{e:prob} with $f=0$.
Using the equation $(X+\lambda)u_-=0$ and the fact that $u_-$ vanishes near
$\pl_-SM\cup\pl_0 SM$, we see that $\supp u_-\subset\Gamma_+$. Then $u_-=0$ by Theorem~\ref{t:res-states}.

To see that $R_-(\lambda):H^s_0(\mathcal M)\to H^{-s}(\mathcal M)$ is bounded,
we use Lemma~\ref{l:meromorphic} and Lemma~\ref{l:samething},
together with the properties of anisotropic spaces $\mathcal H^r_h$
given in~\eqref{e:anisotc}, and the discussion on the admissible values of $r$
in Remark~(iii) following Lemma~\ref{l:keymic}.
In the latter step we use the fact that the Liouville measure is invariant under the flow.
By~\eqref{e:aniseq-1}, this also implies that $A_- R_-(\lambda)$ acts $H^s_0(\mathcal U)\to H^s(\mathcal U)$.

3. This is a restatement of the characterization of resonant states in Theorem~\ref{t:res-states}.
\end{proof}
%%%%%%%%%%%%%%%%%%%%%%%%%%%%%%%%%%%%%%%%%%%%%%%%%%%%%%%%%%%%%%%%%%%%%%%%%%%%%%%%

%%%%%%%%%%%%%%%%%%%%%%%%%%%%%%%%%%%%%%%%%%%%%%%%%%%%%%%%%%%%%%%%%%%%%%%%%%%%%%%%
\subsection{Complete Riemannian manifolds}
\label{s:examples-geodesic}

Another example which fits to our setting, which reduces to
the one discussed in~\S\ref{s:boundary},
is the case of a complete Riemannian manifold $(M,g)$ satisfying:
\begin{enumerate}
\item there exists a function $F\in C^\infty(M;\mathbb R)$
such that for each $a\geq 0$, $M_a:=\{F\leq a\}$ is a compact domain whose boundary
$\{F=a\}$
is smooth and strictly convex with respect to $g$;
\item the trapped set $K$ of the geodesic flow $\varphi^t:SM\to SM$ is hyperbolic in the sense
of assumption~\hyperlink{AA4}{\rm(A4)} in the introduction.
\end{enumerate}

A particular case of such manifold is given by  negatively curved complete Riemannian manifolds $(M,g)$ which admit a compact region $M_0$ with strictly convex smooth boundary $\pl M_0$ such that, if $\nu$ is the unit normal exterior pointing vector field to $\pl M_0$ and $\pi: SM\to M$ the projection on the base, 
the map
\[ \psi: [0,\infty)\x \pl M_0 \to M \setminus M_0^\circ , \quad \psi(t,x)=\pi(\varphi^t(x,\nu(x)))\]
is a smooth diffeomorphism. In the coordinates $(t,x)$ defined by $\psi$, the metric has the form
$dt^2+g_1(t,x,dx)$, therefore the function $t$ is the geodesic distance to $\pl M_0$.
Thus lifting everything to the universal cover and applying Theorem~4.1 (see in particular Remark~4.3 there;
we use that $M$ is negatively curved) 
in~\cite{BiO'N}, we have that $F:=t$ produces a strictly convex foliation, verifying assumption~(1).
Assumption~(2) follows
from~\cite[\S3.9 and Theorem~3.2.17]{kl}.
An asymptotically hyperbolic manifold in the sense of Mazzeo--Melrose~\cite{MaMe}
with negative curvature satisfies these properties, and thus in particular any
convex co-compact hyperbolic manifold (with constant negative curvature) does too. 

We define the incoming/outgoing tails $\Gamma_\pm\subset SM$ on the entire $M$ by
$$
(x,v)\notin\Gamma_\pm\ \Longrightarrow\ \varphi^t(x,v)\to \infty\quad\text{as }t\to \mp\infty.
$$
Note that $K=\Gamma_+\cap \Gamma_-$ is contained in $\{F\leq 0\}$. By Lemma~\ref{l:extended},
we define the vector bundles $E_\pm^*$ over $\Gamma_\pm$. For $f\in C_0^\infty(SM)$, define
$$
R(\lambda)f=\int_0^\infty e^{-\lambda t}(f\circ \varphi^{-t})\,dt,\quad
\Re\lambda>0.
$$
\begin{prop}\label{casecomplete}
Under assumptions~(1) and~(2) above, the operator
$R(\lambda)$ admits a meromorphic extension to $\lambda\in\mathbb C$ as an operator
\[
R(\la): C_0^\infty(SM)\to \mc{D}'(SM)
\]
with poles of finite multiplicity.
Moreover, $\la\in\cc$ is a pole of $R(\la)$ (that is, a Pollicott--Ruelle resonance)
if and only if there exists a non-zero $u\in \mc{D}'(SM)$ satisfying 
\begin{equation}\label{conditionres} 
(X+\la)u=0, \quad \supp u\subset \Gamma_+, \quad \WF(u)\subset E^*_+.
\end{equation}
\end{prop}
\begin{proof}
Applying Proposition~\ref{prop:bvp} to the manifolds $(M_a,g|_{M_a})$,
we continue meromorphically $R_a(\lambda)=\indic_{SM_a}R(\lambda)\indic_{SM_a}$ for all $a\geq 0$.
By analytic continuation, we have $\indic_{SM_a}R_b(\lambda)\indic_{SM_a}=R_a(\lambda)$ for $0\leq a\leq b$.
To show that the family of operators $R_a(\lambda)$ can be pieced together to
an operator $R(\lambda):C_0^\infty(SM)\to\mathcal D'(SM)$, it suffices to
show that each $\lambda$ has the same multiplicity (i.e., the rank of the operator
$\Pi_\lambda$ from Theorem~\ref{t:res-states}) as a resonance
of $R_a(\lambda)$ and $R_b(\lambda)$, for $0\leq a\leq b$.
By Theorem~\ref{t:res-states}, it is then enough to show that for each $j\geq 1$, the restriction operator
\begin{equation}
  \label{e:resti}
\begin{aligned}
\indic_{SM_a}:&\ \{u\in\mathcal D'(SM_b)\mid (X+\lambda)^ju=0,\ \supp u\subset\Gamma_+,\ \WF(u)\subset E_+^*\}
\\\to &\ \{u\in\mathcal D'(SM_a)\mid (X+\lambda)^ju=0,\ \supp u\subset\Gamma_+,\ \WF(u)\subset E_+^*\}
\end{aligned}
\end{equation}
is a linear isomorphism. (The case $j=1$ also gives the characterization~\eqref{conditionres}.)
The fact that~\eqref{e:resti} is an isomorphism follows by using the equation $(X+\lambda)^ju=0$
together with the fact that $\varphi^{-t}(SM_b\cap \Gamma_+)\subset SM_a$ for $t$ large enough;
the latter is a corollary of Lemma~\ref{l:convergence}.
\end{proof}

%%%%%%%%%%%%%%%%%%%%%%%%%%%%%%%%%%%%%%%%%%%%%%%%%%%%%%%%%%%%%%%%%%%%%%%%%%%%%%%%
%                               ACKNOWLEDGEMENTS                               %
%%%%%%%%%%%%%%%%%%%%%%%%%%%%%%%%%%%%%%%%%%%%%%%%%%%%%%%%%%%%%%%%%%%%%%%%%%%%%%%%
\smallsection{Acknowledgements}
We would like to thank Fr\'ed\'eric Faure for many discussions of resonances for
open systems (including the construction of~\cite{FaTs3}) and Maciej Zworski
for advice and support throughout this project. We are also grateful to Andr\'as Vasy
for a discussion of semiclassical propagation estimates and to Mark Pollicott and Fr\'ed\'eric Naud
for an overview of the history of the subject. Finally, we would like to thank
an anonymous referee for many suggestions for improving the manuscript.
This work was completed during
the time SD served as a Clay Fellow. CG was partially supported
by grants ANR-13-BS01-0007-01 and ANR-13-JS01-0006.

%%%%%%%%%%%%%%%%%%%%%%%%%%%%%%%%%%%%%%%%%%%%%%%%%%%%%%%%%%%%%%%%%%%%%%%%%%%%%%%%
%                                 BIBLIOGRAPHY                                 %
%%%%%%%%%%%%%%%%%%%%%%%%%%%%%%%%%%%%%%%%%%%%%%%%%%%%%%%%%%%%%%%%%%%%%%%%%%%%%%%%

% arXiv bibliography macro
\def\arXiv#1{\href{http://arxiv.org/abs/#1}{arXiv:#1}}

\end{document}